\documentclass[10pt]{article}

\usepackage{enumerate,xspace}
\usepackage{amsmath,amssymb,wasysym}
\usepackage[all]{xy}
\usepackage{proof}
\usepackage[svgnames]{xcolor}
\usepackage{pict2e} 
\usepackage{tikz}
\usepackage{mathtools}
\usepackage{latexsym}
\usepackage{dsfont}
\usepackage{multicol}
\usepackage{cmll}
\usepackage{multirow}
\usepackage{longtable}
\usepackage[sort,nocompress]{cite} 

\usepackage{lscape}
\usepackage{array}

\usepackage{hyperref} 
\hypersetup{
    colorlinks,
    citecolor=red,
    filecolor=red,
    linkcolor=blue,
    urlcolor=red
}

\makeatletter
\renewcommand*\env@matrix[1][*\c@MaxMatrixCols c]{%
  \hskip -\arraycolsep
  \let\@ifnextchar\new@ifnextchar
  \array{#1}}
\makeatother

\newtheorem{observation}{Remark}[section]
\newtheorem{lemma}[observation]{Lemma}  
\newtheorem{theorem}[observation]{Theorem}
\newtheorem{definition}[observation]{Definition}
\newtheorem{example}[observation]{Example}

\newtheorem{proposition}[observation]{Proposition} 
\newtheorem{corollary}[observation]{Corollary}

\usepackage{todonotes}

\newcommand{\wand}{\ensuremath{
  \mathrel{\vbox{\offinterlineskip\ialign{
    \hfil##\hfil\cr
    $\star$\cr
    \noalign{\kern-1ex}
    $\vert$\cr
}}}}}

\makeatletter

\newdimen\w@dth

\def\setw@dth#1#2{\setbox\z@\hbox{\scriptsize $#1$}\w@dth=\wd\z@
\setbox\@ne\hbox{\scriptsize $#2$}\ifnum\w@dth<\wd\@ne \w@dth=\wd\@ne \fi
\advance\w@dth by 1.2em}

\def\t@^#1_#2{\allowbreak\def\n@one{#1}\def\n@two{#2}\mathrel
{\setw@dth{#1}{#2}
\mathop{\hbox to \w@dth{\rightarrowfill}}\limits
\ifx\n@one\empty\else ^{\box\z@}\fi
\ifx\n@two\empty\else _{\box\@ne}\fi}}
\def\t@@^#1{\@ifnextchar_ {\t@^{#1}}{\t@^{#1}_{}}}

\def\t@left^#1_#2{\def\n@one{#1}\def\n@two{#2}\mathrel{\setw@dth{#1}{#2}
\mathop{\hbox to \w@dth{\leftarrowfill}}\limits
\ifx\n@one\empty\else ^{\box\z@}\fi
\ifx\n@two\empty\else _{\box\@ne}\fi}}
\def\t@@left^#1{\@ifnextchar_ {\t@left^{#1}}{\t@left^{#1}_{}}}

\def\two@^#1_#2{\def\n@one{#1}\def\n@two{#2}\mathrel{\setw@dth{#1}{#2}
\mathop{\vcenter{\hbox to \w@dth{\rightarrowfill}\kern-1.7ex
                 \hbox to \w@dth{\rightarrowfill}}%
       }\limits
\ifx\n@one\empty\else ^{\box\z@}\fi
\ifx\n@two\empty\else _{\box\@ne}\fi}}
\def\tw@@^#1{\@ifnextchar_ {\two@^{#1}}{\two@^{#1}_{}}}

\def\tofr@^#1_#2{\def\n@one{#1}\def\n@two{#2}\mathrel{\setw@dth{#1}{#2}
\mathop{\vcenter{\hbox to \w@dth{\rightarrowfill}\kern-1.7ex
                 \hbox to \w@dth{\leftarrowfill}}%
       }\limits
\ifx\n@one\empty\else ^{\box\z@}\fi
\ifx\n@two\empty\else _{\box\@ne}\fi}}
\def\t@fr@^#1{\@ifnextchar_ {\tofr@^{#1}}{\tofr@^{#1}_{}}}

\newdimen\W@dth
\def\setW@dth#1#2{\setbox\z@\hbox{$#1$}\W@dth=\wd\z@
\setbox\@ne\hbox{$#2$}\ifnum\W@dth<\wd\@ne \W@dth=\wd\@ne \fi
\advance\W@dth by 1.2em}

\def\T@^#1_#2{\allowbreak\def\N@one{#1}\def\N@two{#2}\mathrel
{\setW@dth{#1}{#2}
\mathop{\hbox to \W@dth{\rightarrowfill}}\limits
\ifx\N@one\empty\else ^{\box\z@}\fi
\ifx\N@two\empty\else _{\box\@ne}\fi}}
\def\T@@^#1{\@ifnextchar_ {\T@^{#1}}{\T@^{#1}_{}}}

\def\T@left^#1_#2{\def\N@one{#1}\def\N@two{#2}\mathrel{\setW@dth{#1}{#2}
\mathop{\hbox to \W@dth{\leftarrowfill}}\limits
\ifx\N@one\empty\else ^{\box\z@}\fi
\ifx\N@two\empty\else _{\box\@ne}\fi}}
\def\T@@left^#1{\@ifnextchar_ {\T@left^{#1}}{\T@left^{#1}_{}}}

\def\Tofr@^#1_#2{\def\N@one{#1}\def\N@two{#2}\mathrel{\setW@dth{#1}{#2}
\mathop{\vcenter{\hbox to \W@dth{\rightarrowfill}\kern-1.7ex
                 \hbox to \W@dth{\leftarrowfill}}%
       }\limits
\ifx\N@one\empty\else ^{\box\z@}\fi
\ifx\N@two\empty\else _{\box\@ne}\fi}}
\def\T@fr@^#1{\@ifnextchar_ {\Tofr@^{#1}}{\Tofr@^{#1}_{}}}

\def\Two@^#1_#2{\def\N@one{#1}\def\N@two{#2}\mathrel{\setW@dth{#1}{#2}
\mathop{\vcenter{\hbox to \W@dth{\rightarrowfill}\kern-1.7ex
                 \hbox to \W@dth{\rightarrowfill}}%
       }\limits
\ifx\N@one\empty\else ^{\box\z@}\fi
\ifx\N@two\empty\else _{\box\@ne}\fi}}
\def\Tw@@^#1{\@ifnextchar_ {\Two@^{#1}}{\Two@^{#1}_{}}}

\def\to{\@ifnextchar^ {\t@@}{\t@@^{}}}
\def\from{\@ifnextchar^ {\t@@left}{\t@@left^{}}}
\def\tofro{\@ifnextchar^ {\t@fr@}{\t@fr@^{}}}
\def\To{\@ifnextchar^ {\T@@}{\T@@^{}}}
\def\From{\@ifnextchar^ {\T@@left}{\T@@left^{}}}
\def\Two{\@ifnextchar^ {\Tw@@}{\Tw@@^{}}}
\def\Tofro{\@ifnextchar^ {\T@fr@}{\T@fr@^{}}}

\makeatother

\title{Itegories}
\author{Robin Cockett and Jean-Simon Pacaud Lemay}
\date{Dedicated to Phil Scott (1947 -- 2023)}							

\begin{document}
\allowdisplaybreaks

\maketitle

\begin{abstract} An itegory is a restriction category with a Kleene wand.  Cockett, D{\'\i}az-Bo{\"\i}ls, Gallagher, and Hrube{\v{s}} briefly introduced Kleene wands to capture iteration in restriction categories arising from complexity theory. The purpose of this paper is to develop in more detail the theory of Kleene wands and itegories.

A Kleene wand is a binary operator which takes in two disjoint partial maps, one of type ${X \to X}$ and the other of type ${X \to A}$, and produces a partial map $X \to A$. This latter map is interpreted as iterating the endomorphism until it lands in the domain of definition of the second map. In a setting with infinite disjoint joins, there is always a canonical Kleene wand realizing this intuition. 

The standard categorical interpretation of iteration is via trace operators on coproducts.  For extensive restriction categories, we explain in detail how having a Kleene wand is equivalent to this standard interpretation of iteration. This suggests that Kleene wands can be used to replace parametrized iteration and traces in restriction categories which lack coproducts. Further evidence of this is  exhibited by providing a matrix construction which embeds an itegory into a traced extensive restriction category. We also consider Kleene wands in classical restriction categories and show how, in this case, a Kleene wand is completely determined by its endomorphism component. 
 \end{abstract}

\noindent \small \textbf{Acknowledgements.} The authors would like to thank Masahito Hasegawa for many useful discussion, and would like to thank him and RIMS at Kyoto University for helping fund research visits so that the authors could work together on this project. The authors would also like to thank Chad Nester for useful preliminary discussions and his support of this project, also thank Ben MacAdam for reminding the authors about the first named author's previous work with D{\'\i}az-Bo{\"\i}ls, Gallagher, and Hrube{\v{s}} in \cite{cockett2012timed}. The authors are also grateful to the anonymous reviewers for their useful comments and suggestions. For this research, the first named author was partially funded by NSERC, while the second named author was financially supported by a JSPS PDF (P21746). \\ 

\tableofcontents

\newpage

We would like to dedicate this paper to Phil Scott (1947--2023). 

\paragraph{From Robin:} I have so many memories of time spent with Phil that it is hard to know where to start.  In the early 1990s Phil was instrumental in getting me an introduction at the University of Calgary and whence the job I have enjoyed since then.  Even back then -- in fact, maybe, especially back then -- Phil was highly influential and respected. The book with Lambek, ``Introduction to higher order categorical logic'', had truly changed the categorical world.

Phil loved visiting Edinburgh, and in 2018 he conspire to have us spend sabbatical time there together.  My wife, Polly, and I had long wanted to hike the West Highland Way (which starts near Glasgow). Doing this before our Edinburgh sojourn seemed like a perfect moment. I was therefore to meet my wife, who would be returning from adventures further north, at the start of the hike. Somehow, I needed to leave, in Edinburgh as I passed through on trains from Oxford, a bag of stuff not needed on the hike.  Phil concocted a simple logistical plan: as he would already be in Edinburgh, he would meet my midday train at the Edinburgh railway station: I would hand-off my bag to him and hop back on the train to Glasgow.  

Well, of course, everything went wrong. The night before the line to Edinburgh from England had been subject to major tree falls due to a storm. When a train finally got through -- in the late evening -- it was packed with unhappy passengers standing like sardines in the aisles.  As Phil did not believe in cell phones, there was no way to let him know how badly delayed I was.  When the train finally arrived at Edinburgh the passengers literally poured onto the platform.  This included me with a huge suitcase (on wheels) which I had been supposed to hand-off over six hours earlier.  I had no idea what I was now going to do with the bag!! 

Imagine, however, my disbelief as clinging to the hand-rail at the bottom of the exit stairway, being rudely buffeted by the frustrated crowd pouring forth from the train, stood Phil and his wife Marcia.  Relief hardly does justice to my reaction. Unbelievably they had waited for over six hours for a train to get through!  

The hand-off went smoothly ... and the sabbatical was a special and memorable one!

Phil was hugely generous (and reliable!) ... and a friend I sorely miss.

\paragraph{From JS:} Phil was both a friend and a mentor. Phil was an important figure in the early stages of my academic studies and had quite an impact on the direction of my research interests. Phil was my first math professor in undergrad at the University of Ottawa: he was the one who taught me how to do proper proofs. I took many classes from him in undergrad, and Phil quickly became my favourite lecturer. Near the end of undergrad when I was trying to decide what area of research to study, Phil was the one who first mentioned category theory to me by introducing me to star-autonomous categories and Linear Logic. I credit Phil for why and how I became a category theorist. After my time at the University of Ottawa, Phil continued to be a mentor: he was always very supportive and happy to discuss my latest research interests or new results. But more importantly, Phil became a friend. We would often catch up at conferences or whenever I was passing through Ottawa again. We also both visited Edinburgh at the same time, and often went on hikes together. Phil's passing is a great loss for many communities, including the Canadian mathematics community and the category theory community. Phil was incredibly kind, a fantastic mentor, and a great friend to many of us. Condolences to his loved ones. May he rest in peace. 


\section{Introduction}

Traced monoidal categories were introduced by Joyal, Street, and Verity in \cite{joyal_street_verity_1996}. A traced monoidal category is a braided (though often symmetric in practice) monoidal category equipped with a \emph{trace operator}, which is an operator $\mathsf{Tr}$ that takes in a map of type $f: X \otimes A \to X \otimes B$, which has a same object $X$ in both the domain and codomain, and produces a map of type $\mathsf{Tr}(f): A \to B$ called the \emph{trace} of $f$, where we say that we \emph{traced out} the object $X$. Trace operators can often intuitively be thought of as providing feedback, where $\mathsf{Tr}(f)$ can be interpreted as looping back the output $X$ into the input $X$ of $f$ repeatedly. This is why in the graphical calculus for traced monoidal categories \cite[Sec 5]{selinger1999categorical}, one draws the trace of $f$ by connecting the output wire $X$ back into the input wire for $X$ creating a cycle. Thus, graphically speaking, having a trace operator allows for loops.   

Traced monoidal categories are now very well studied, being an important framework in both mathematics and computer science, with a wide variety of different kinds of examples. In \cite{abramsky1996retracing}, Abramsky characterized two kinds of trace monoidal categories: wave style traces and particle style traces. On the one hand, wave style trace can be understood as passing information in a global information wave throughout the system. Examples of wave style traces can be found in \cite[Ex 4.5]{haghverdi2010geometry}, which include \emph{compact closed categories} \cite[Prop 3.1]{joyal_street_verity_1996}, where the trace operator captures the (partial) trace of matrices and is fundamental in categorical quantum foundations \cite{abramsky2009abstract}, and \emph{traced Cartesian monoidal categories} \cite[Sec 6.4]{selinger2010survey}, which are traced monoidal categories on the product -- in this case trace operator captures the notion of feedback via fixed points \cite[Thm 3.1]{Hasegawa97recursionfrom}. On the other hand, particle style traces are interpreted by streams of particles or tokens flowing around a network. Examples of particle style traces can be found in \cite[Ex 4.6]{haghverdi2010geometry}. Often times, particle style traces can be expressed using partially defined infinite sum operations, and so particle style traces are also often referred to as sum style traces as well. Haghverdi and Scott were particularly interested in these sum style traces, which led to the development of \emph{traced unique decomposition categories} \cite{haghverdi2000categorical} to help better understand these infinite sum formulas for wave style traces. Haghverdi and Scott, along with Abramsky, in a series of papers \cite{haghverdi2000categorical,abramsky2002geometry,haghverdi2010geometry,haghverdi2005geometry,haghverdi2006categorical,haghverdi2010towards} used traced unique decomposition categories to give deeper analysis and recapture models of Girard's Geometry of Interaction \cite{girard1989towards}. 

Most models of particle style traces are in fact \emph{traced coCartesian monoidal categories} \cite[Sec 6.4]{selinger2010survey}, which are traced monoidal categories whose monoidal product is a coproduct, where in this case the trace operator captures the notion of feedback instead via \emph{iteration}. It turns out, thanks to the couniversal property of the coproduct, that to give a trace operator on a coproduct $+$ is equivalent to giving a \emph{parametrized iteration operator} \cite[Prop 6.8]{selinger2010survey}, which is an operator which takes in a map of type $f: X \to X +A$, and produces a map of type $\mathsf{Iter}(f): X \to A$, called the \emph{iteration} of $f$. As noted by Selinger \cite[Remark 6.10]{selinger2010survey}, parametrized iteration operators were first defined by C{\u{a}}z{\u{a}}nescu and Ungureanu in \cite{cuazuanescu1982again} as an algebraic theory with iteration. Parametrized iteration operators, and hence traced coCartesian monoidal categories, and other categorial formulations of iteration have been extensively studied by C{\u{a}}z{\u{a}}nescu and {\c{S}}tef{\u{a}}nescu \cite{cuazuanescu1994feedback,cuazuanescu1990towards}, Bloom and {\'E}sik \cite{bloom1993iteration}, Ad{\'a}mek and Milius and Velebil \cite{adamek2006iterative,adamek2011elgot, adamek2010equational} along with Goncharov and Rauch \cite{goncharov2016complete}, Jacobs \cite{jacobs2010coalgebraic}, and many others. 

One of the main examples of a traced coCartesian monoidal category is the category of sets and partial functions. Here, the coproduct is given by disjoint union, while the parametrized iteration operator and trace operator are given by the natural feedback operator. So for a partial function $f: X \to X + A$, $f(x)$ can be thought of landing in either $X$ or $A$, but not both, or is undefined. So $\mathsf{Iter}(f): X \to A$ is defined as follows. Step 1: take $x$ and check if $f(x)$ is defined, if not $\mathsf{Iter}(f)(x)$ is undefined. Step 2: if $f(x)$ is defined and if $f(x) \in A$, then $\mathsf{Iter}(f)(x) = f(x)$. Step 3: if $f(x) \in X$ instead, then go back to Step 1 and repeat this process starting from $f(x)$ instead of $x$. This process continues until it stops after a finite number of steps: if it does not stop then $\mathsf{Iter}(f)(x)$ is undefined. The trace operator is defined via a similar process, where essentially for a partial function $f: X + A \to X + B$, its trace continuously iterates the partial function if it continues to land in $X$ and only terminates if after a finite number of iterations it eventually lands in $B$. Thus the trace operator and the parametrized iteration operator are forms of \emph{guarded} iteration, that is, an iteration that continues until a certain necessary condition, called the guard, is satisfied.

The trace operator for partial functions is not only a particle style trace, but really a sum style trace, where the partial infinite sum operation is given by the join of partial functions. In fact, partial functions (more precisely partial injections) was one of the main motivating examples that led Haghverdi and Scott to studying traced unique decomposition categories. More precisely, it turns out that partial functions are in fact an example of Arbib and Manes' notion of a \emph{partially additive category} \cite{manes2012algebraic}, which is canonically a traced unique decomposition category that is also a traced coCartesian monoidal category \cite[Prop 3.1.4]{haghverdi2000categorical}. Partial functions are also the canonical example of a \emph{restriction category}. This paper is motivated by the desire to have a restriction category perspective and explanation for the trace operator on partial functions. More precisely, we are interested in studying trace operators for \emph{extensive} restriction categories. 

Restriction categories were introduced by the first named author and Lack in \cite{cockett2002restriction}, and they provide the categorial foundations for the theory of partiality. Briefly, a restriction category (Def \ref{def:restcat}) is a category equipped with a \emph{restriction operator} which associates to every map $f: A \to B$, an endomorphism $\overline{f}: A \to A$ called its \emph{restriction}, which captures the domain of definition of $f$. Thus maps in a restriction category can be thought of as partial in an appropriate sense. Restriction categories are now a well established field of research and have been used to capture various important notions related partiality. In \cite{cockett2007restriction}, the first named author and Lack studied extensive restriction categories (Def \ref{def:ext-rest-cat}) which, as the name suggests,  are the appropriate restriction category analogue of (total) extensive categories \cite{carboni1993introduction}. Briefly, an extensive restriction category is a restriction category with \emph{restriction coproducts} (Def \ref{def:restcoprod}), which are coproducts compatible with the restriction, and such that every map has a \emph{decision} (Def \ref{def:decision}). For a map $f: A \to B_1 + \hdots + B_n$, its decision $\langle f \rangle: A \to A + \hdots + A$ splits up the domain of definition of $f$ into the parts that fall into each $B_i$ respectively. Thus, in an extensive restriction category, the restriction coproduct $+$ should indeed be thought of as being a disjoint union. Importantly, this allows extensive restriction categories to have a matrix calculus \cite{cockett2007restriction}, that is, a map of type $A_1 + \hdots + A_n \to B_1 + \hdots + B_m$ can be represented as an $n \times m$ matrix, and where composition is given by matrix multiplication (Sec \ref{sec:matrix}). 

Since restriction coproducts are just coproducts in the usual sense, one can consider trace operators and parametrized iteration operators for restriction coproducts, and thus also consider \emph{traced} extensive restriction categories (Def \ref{def:traced-ext}). In particular, the matrix calculus allows us to interpret the trace operator as acting on $2 \times 2$ square matrices and the parametrized iteration operator as acting on $1 \times 2$ row matrices, as well as re-expressing their axioms using matrices (Lemma \ref{lem:ext-iter}). Thus a map of type $X \to X + A$ is actually a $1 \times 2$ row matrix consisting of an endomorphism ${X \to X}$ and a map $X \to A$. Thus, for an extensive restriction category, a parametrized iteration operator is really an operator that takes an endomorphism $X \to X$ and a map $X \to A$ which are ``disjoint'' in an appropriate sense, and produces a map of type ${X \to A}$. This observation was made by the first named author, D{\'\i}az-Bo{\"\i}ls, Gallagher, and Hrube{\v{s}} in \cite{cockett2012timed}, where they considered parametrized iteration operators for \emph{distributive} restriction categories with the objective of constructing categories of ``timed sets'' corresponding to complexity classes. This led to the definition of a \emph{Kleene wand} operator \cite[Sec 4]{cockett2012timed}.

However, Kleene wands were only briefly introduced in \cite{cockett2012timed}, and only in the context of distributive extensive restriction categories. These have \textit{restriction products} which not all extensive restriction categories have. Moreover, a full proof that a Kleene wand was equivalent to a trace operator, was not provided.  This omission is corrected here: we show that, for extensive restriction categories, Kleene wands do provide an equivalent way of describing parametrized iteration and hence a trace operator (Thm \ref{thm:ext-wand-trace}).  

While \cite{cockett2012timed} was successful in developing categories of timed sets for various complexity classes, Kleene wands themselves were not further developed and, indeed, seemed to have fallen somewhat into obscurity. With the recent interest of studying trace operators in restriction categories, it seems like a good time to dust off Kleene wands. Thus the objective of this paper is to revisit Kleene wands (Def \ref{def:wand}), and to study them as a stand-alone concept (Sec \ref{sec:itegories}).  The ``advantage'' of a Kleene wand is that it can be defined in a setting without coproducts, and therefore can be used to provide a more general notion of iteration.  
  
As explained above, a key feature that makes the trace operator for partial functions work is that the coproduct is given by disjoint union. To capture this in a setting without coproducts, we introduce the concept of an \textbf{interference relation} (Def \ref{def:interference}), which is a way of describing when maps in a restriction category are disjoint. In particular, maps that are disjoint via an interference relation do have disjoint domains of definition, so when one is defined the other is undefined and vice-versa. We call a restriction category with a chosen interference relation an \textbf{interference restriction category} (Def \ref{def:inter-rest-cat}). Then a Kleene wand $\wand$ on an interference restriction category takes in a map $f: X \to X$ and a map $g: X \to A$ which are disjoint and produces a new map $f \wand g: X \to A$. This new map $f \wand g$ should be interpreted as iterating $f$ until it falls in the domain of definition $g$, in other words, iterating $f$ until the guard $g$ is satisfies. Thus, in an interference restriction category, a Kleene wand $\wand$ is a form of guarded iteration and we call an interference restriction category with a Kleene wand an \textbf{itegory}. 

We also show how \textbf{disjoint joins} (Def \ref{def:disjoint-join}) can be defined in a general interference restriction category.  Then show that, in a setting with countable disjoint joins, there is a canonical Kleene wand given by taking the disjoint join of all iterations of $f$ composed with $g$ (Thm \ref{thm:countable-wand}). We also explain how in a \emph{classical} restriction category \cite{cockett2009boolean} -- a setting where we have complements (Def \ref{def:classical}) -- to give a Kleene wand is equivalent to giving a \textbf{Kleene upper-star} (Def \ref{def:star}), which is an operator on endomorphisms only. 

Every extensive restriction category has a canonical interference relation (Thm \ref{thm:ext-int}) given by \textbf{separating decisions} (Def \ref{def:dec-sep}). Intuitively, two maps in an extensive restriction category are disjoint if there is a decision which allows us to separate their domains of definition. This allows us to show that, for an extensive restriction category, every Kleene wand induces a parametrized iteration operator (Prop \ref{prop:wand-to-iter}) and, conversely, every iteration operator induces a Kleene wand (Prop \ref{prop:iter-to-wand}).  Furthermore, these constructions are inverses of each other. Thus, for an extensive restriction category, a Kleene wand and a parametrized iteration operator are essentially the same thing. Thus, for an extensive restriction category, being traced is equivalent to being an itegory (Thm \ref{thm:ext-wand-trace}).  We also note that there is a matrix construction on an interference restriction category which can be used to obtain an extensive restriction category (Prop \ref{prop:construction2}). Indeed, applying this matrix construction to an itegory we obtain a traced extensive restriction category (Prop \ref{prop:matrix-ext-wand}). 

It is worth mentioning that developing interference relations and studying Kleene wand independently, brought to light an important subtlety of Kleene wands which had been overlooked. In \cite{cockett2012timed}, Kleene wands were only considered for the \emph{maximal} interference relation (Lemma \ref{lemma:two-interference}), which says that two maps are disjoint if, where one is defined, the other is not, and vice-versa. However, for an extensive restriction category, the canonical interference relation may not necessarily be the maximal interference relation, and it is for the canonical one that correspondence between parametrized iteration operators and Kleene wands hold. Fortunately, the main restriction categories of interest in \cite{cockett2012timed} were distributive extensive restriction categories with joins and for these the canonical interference relation and the maximal interference relation coincide. Thus, in this case, Kleene wands for the maximal interference relation do indeed correspond to parametrized iteration operators. In this paper, however, we study trace operators for arbitrary extensive restriction categories, so the axioms of a Kleene wand had to be reworked slightly. 

\paragraph{Outline and Contributions:} In Sec \ref{sec:rest}, we review the basics of restriction category theory and introduce our two main running examples: the category of partial functions and the category of partial recursive functions. 

In Sec \ref{sec:interference}, we introduce the concept of an interference relation (Def \ref{def:interference}) and of an interference restriction category (Def \ref{def:inter-rest-cat}). We show how giving an interference relation can be defined on parallel maps or equivalently on restriction idempotents (Prop \ref{prop:equiv-inter}). We prove various useful facts about interference relations (Lemma \ref{lem:interference1}), and also show that there is always a maximal and a minimal interference relation (Lemma \ref{lemma:two-interference}). 

In Sec \ref{sec:disjoint-joins}, we adapt the notion of disjoint joins to interference restriction categories (Def \ref{def:disjoint-join}). We explain what it means for an interference restriction category to have all finite disjoint joins (Lemma \ref{lemma:binary-finite}) and give a construction which embeds an interference restriction category into a disjoint interference restriction category (Prop \ref{prop:construction1}). 

In Sec \ref{sec:itegories}, we introduce Kleene wands for interference restriction categories, generalizing the definition slightly from \cite{cockett2012timed} from the maximal interference relation to any arbitrary interference relation. As mentioned above, we call interference restriction categories with disjoint joins and a Kleene wand an itegory. We provide various interesting identities that all Kleene wands satisfy (Lemma \ref{lem:strong-wand-alt}, \ref{lem:iteration}, \& \ref{lem:wand}). We also consider the notion of uniformity for Kleene wands (Def \ref{def:uni-wand}) and define the notion of a (strongly) inductive Kleene wand (Def \ref{def:ind-wand}). We note that an interference restriction category can have at most one inductive Kleene wand (Lemma \ref{lemma:induc-less}.(\ref{lemma:induc-less.ii})). Indeed, in the presence of countable disjoint joins, there is a canonical strongly inductive Kleene wand given by the expected join of iterations formula (Thm \ref{thm:countable-wand}). 

In Sec \ref{sec:kleene-star}, we introduce classical interference restriction categories (Def \ref{def:classical}) and Kleene upper-stars (Def \ref{def:star}). We then show that for a classical interference restriction category, there is a bijective correspondence between Kleene wands and Kleene upper-stars (Thm \ref{thm:star=wand}). 

In Sec \ref{sec:extensive}, we review extensive restriction categories and decisions. We then introduce the notion of being decision disjoint (Def \ref{def:perp-extensive}) and show that this induces an interference relation on an extensive restriction category, in other words, an extensive restriction category is an interference restriction category with binary disjoint joins (Thm \ref{thm:ext-int}). 

In Sec \ref{sec:matrix}, we review the matrix representation of extensive restriction categories and how it relates to decision disjointness. We also introduce a matrix construction on a finitely disjoint interference restriction category (Def \ref{def:matrix-const}) and show that this matrix construction is an extensive restriction category (Prop \ref{prop:construction2}). 

In Sec \ref{sec:trace}, we first review the correspondence between trace operators and parametrized iterations operators for coproducts. We express the axioms of a parametrized iteration operator and a trace operator using the matrix notation of an extensive restriction category (Lemma \ref{lem:ext-iter}). This leads to the main result of the paper, which is that, for extensive restriction categories, there is a bijective correspondence between parametrized iteration operators (equivalently trace operators) and Kleene wands (Thm \ref{thm:ext-wand-trace}). In the classical setting we extend this to including Kleene upper-stars (Prop \ref{prop:trace-star}). We conclude by explaining how computing traces of matrices of any size in a traced extensive restriction category reduces to computing traces of $2 \times 2$ matrices, which allows us to show that applying the matrix construction to an itegory results in a traced extensive restriction category (Prop \ref{prop:matrix-ext-wand}). 

\paragraph{Future Directions:} In future work, it would be interesting to study how Kleene wands connect up with other categorifications of iteration, such as iteration theories \cite{bloom1993iteration}, Elgot theories \cite{adamek2011elgot}, and Elgot monads \cite{adamek2010equational, goncharov2016complete}. Moreover, it would also be interesting to consider what happens to the Kleene wand when we apply the $\mathsf{INT}$-construction \cite[Sec 4]{joyal_street_verity_1996} or $\mathsf{GOI}$-construction \cite[Sec 3]{abramsky1996retracing} (which are isomorphic constructions \cite[Prop 2.3.6]{haghverdi2000categorical}) on a traced extensive restriction category to obtain a compact closed category. This would then allow us to relate Kleene wands to the Geometry of Interaction (GOI), either via GOI situations \cite[Def 4.1]{abramsky2002geometry} or via GOI interpretations \cite[Def 15]{haghverdi2006categorical}. On the other hand, Kleene wands can naturally be interpreted by the basic programming control structure:
\[ f \wand g := {\sf until}~~\overline{g}~~{\sf do}~~f~~{\sf end};~~ g\]
where $\overline{g}$ is interpreted as the test of whether a value lies in the domain of $g$.  
This suggests that the Kleene wand could be used as a primitive to interpret imperative 
style programs and flow diagrams which use ``looping'' constructs as above.  Thus, a natural 
question is how itegories relate to other models of computability. In particular, it would be interesting to revisit Kleene wands in Turing categories \cite{turing-categories,cockett2012timed} which are based on partial combinatory algebras. 

\paragraph{Conventions:} For an arbitrary category $\mathbb{X}$, we will denote objects using capital letters $A$, $B$, $C$ etc., homsets will be denoted as $\mathbb{X}(A,B)$ and maps will be denoted by lowercase letters $f,g,h, etc. \in \mathbb{X}(A,B)$. Arbitrary maps will be denoted using an arrow $f: A \to B$, identity maps as $1_A: A \to A$, and, for composition, we will use \emph{diagrammatic} notation
    (as opposed to \emph{applicative} notation), that is, the composition of $f: A \to B$ followed by $g: B \to C$ is denoted as $fg: A \to C$ (rather than $g \circ f$). 

\section{Restriction Categories}\label{sec:rest}

In this background section, we review restriction categories and the necessary basic theory for the story of this paper such as restriction idempotents, the canonical partial order on maps, and restriction zero maps, as well as some main examples. For a more in-depth introduction to restriction categories, we invite the reader to see \cite{cockett2002restriction,cockett2003restriction, cockett2007restriction, cockett2009boolean, cockett2014restriction}. 

\begin{definition}\label{def:restcat} For a category $\mathbb{X}$, a \textbf{restriction operator} $\overline{(\phantom{f})}$ \cite[Sec 2.1.1]{cockett2002restriction} is a family of operators:
\[ 
\overline{(\phantom{f})}: \mathbb{X}(A,B) \to \mathbb{X}(A,A) ~~~~~~~~~ \begin{matrix}[c] \infer{\overline{f}: A \to A}{f: A \to B} \end{matrix}
 \]
such that the following four axioms hold: 
\begin{enumerate}[{\bf [R.1]}]
\item $\overline{f}f = f$, for all $f: A \to B$;
\item $\overline{f} \overline{g} = \overline{g} \overline{f}$, for all $f: A \to B$ and $g: A \to C$;
\item $\overline{\overline{g}f} = \overline{g} \overline{f}$, for all $f: A \to B$ and $g: A \to C$;
\item $f\overline{g} =  \overline{fg} f$, for all $f: A \to B$ and $g: B \to C$. 
\end{enumerate}
For a map $f: A \to B$, the map ${\overline{f}: A \to A}$ is called the \textbf{restriction of $f$}. A \textbf{restriction category} is a category equipped with a restriction operator. Furthermore, in a restriction category: 
\begin{enumerate}[{\em (i)}]
\item A \textbf{total} map \cite[Sec 2.1.2]{cockett2002restriction} is a map ${f: A \to B}$ such that $\overline{f} = 1_A$;
\item For parallel maps $f: A \to B$ and $g: A \to B$ we say that $f \leq g$ if $\overline{f}g= f$ \cite[Sec 2.1.4]{cockett2002restriction}; 
\item A \textbf{restriction idempotent} \cite[Sec 2.3.3]{cockett2002restriction} is a map $e: A \to A$ such that $\overline{e} = e$. For every object $A \in \mathbb{X}$, we denote by $\mathcal{O}(A)$ the set of restriction idempotents of type $A \to A$. 
 \end{enumerate}  
\end{definition}

The main intuition for a restriction category is that maps $f$ are partial and the restriction $\overline{f}$ captures the domain of definition of $f$. As we will see in Ex \ref{ex:PAR} below, the standard intuition for $\overline{f}$ is checking if $x \in A$ is in the domain of $f$ or not, that is, $\overline{f}(x)=x$ when $f(x)$ is defined and $\overline{f}(x)$ is undefined when $f(x)$ is undefined. Thus total maps are interpreted as maps which are everywhere defined. 

The intuition for the partial order $\leq$ is that $f \leq g$ means whenever $f$ is defined, $g$ is also defined and is equal to $f$. Moreover, this partial order enriches every restriction category $\mathbb{X}$ over posets.  Thus, each homset $(\mathbb{X}(A,B), \leq)$ is a poset and composition is monotone, that is if $f \leq g$ then for suitable maps $h$ and $k$, $h f k \leq h g k$. Also observe that $f \leq g$ implies that $\overline{f} \leq \overline{g}$ (though the converse is not necessarily true). 

Restriction idempotents correspond to the domain of definition of maps, and thus $\mathcal{O}(A)$ may be interpreted as certain special subsets of the object $A$.  The name restriction idempotent is justified since a restriction idempotent $e$ is indeed idempotent, $e e= e$. In particular, every restriction $\overline{f}$ is a restriction idempotent, so we have $\overline{\overline{f}} = \overline{f}$ and $\overline{f}~\overline{f}=\overline{f}$. Furthermore, $\mathcal{O}(A)$ has a natural bounded meet-semilattice structure, where the meet is given by composition, $e_0 \wedge e_1 = e_0 e_1$, and the top element is the identity $1_A$. The induced partial order on $\mathcal{O}(A)$ on this lattice structure is the same as the restriction order.  Thus, for restriction idempotents, $e_0 \leq e_1$ if and only if $e_0 e_1 = e_0$, and also $e \leq 1_A$. 

A key notion for the story of this paper is restriction zero maps: these are intuitively maps which are nowhere defined. Categorically speaking, a restriction zero map is a zero map in the usual sense whose restriction is also a zero map.  Recall a category has zero maps if for every pair of objects $A$ and $B$ there is a map $0: A \to B$ which satisfies the annihilation property, that is, $f0=0=0f$ for all maps $f$. 

\begin{definition} A restriction category is said to have \textbf{restriction zeroes} \cite[Sec 2.2]{cockett2007restriction} if it has zero maps ${0: A \to B}$ such that $\overline{0} =0$.
\end{definition}

Observe that for any map $f$, $0 \leq f$. Thus restriction zero maps are the bottom elements for each homset. Moreover, the restriction zero map of type $0: A \to A$ is a restriction idempotent, and is, therefore, the bottom element of $\mathcal{O}(A)$. 

We conclude this background section by providing two running examples of restriction categories. For a list of many other examples of restriction categories, see \cite[Sec 2.1.3]{cockett2002restriction}.

\begin{example}\label{ex:PAR} \normalfont Let $\mathsf{PAR}$ be the category whose objects are sets and whose maps are partial functions between sets. Then $\mathsf{PAR}$ is a restriction category \cite[Ex 2.1.3.1]{cockett2002restriction} where the restriction of a partial function $f: X \to Y$ is the partial function ${\overline{f}: X \to X}$ defined as follows: 
\[ \overline{f}(x) = \begin{cases} x & \text{if } f(x) \downarrow \\
\uparrow & \text{if } f(x) \uparrow \end{cases} \]
where $\downarrow$ means defined and $\uparrow$ means undefined. The total maps in $\mathsf{PAR}$ are precisely the functions that are everywhere defined -- thus the subcategory of total maps of $\mathsf{PAR}$ is simply the category of sets and functions. For partial functions $f: X \to Y$ and $g: X \to Y$,  we have that $f \leq g$ if $g(x)=f(x)$ whenever $f(x) \downarrow$. The restriction idempotents for a set $X$ are in bijective correspondence to subsets of $X$ \cite[Ex 2.2.1]{cockett2012range}. In particular, for every subset $U \subseteq X$, define the partial function $e_U : X \to X$ as follows: 
 \begin{align*}
  e_U(x) \colon = \begin{cases} 
   x & \text{ if } x \in U \\
   \uparrow & \text{ if } x \notin U 
  \end{cases}
   \end{align*}
 Then clearly $\overline{e_U} = e_U$, so $e_U$ is a restriction idempotent, and every restriction idempotent is of this form. In particular, note that $\overline{f} = e_{\mathsf{dom}(f)}$. So $\mathcal{O}(X)$ is isomorphic to the powerset $\mathcal{P}(X)$ of $X$, $\mathcal{O}(X) \cong \mathcal{P}(X)$. Lastly, $\mathsf{PAR}$ has restriction zero maps where $0: X \to Y$ is the partial function which is nowhere defined, i.e., $0(x) \uparrow$ for all $x \in X$. 
\end{example}

\begin{example}\label{ex:REC} \normalfont Let $\mathbb{N}$ be the set of natural numbers, and let $\mathsf{REC}$ be the category whose objects are finite powers $\mathbb{N}^n$ and where a map $f: \mathbb{N}^n \to \mathbb{N}^m$ is a partial function which is an $m$-tuple $f = \langle f_1, \hdots, f_n \rangle$ of partial recursive functions in $n$ variables $f_i: \mathbb{N}^n \to \mathbb{N}$. Then $\mathsf{REC}$ is a restriction category \cite[Ex 2.1.(ii)]{cockett2014restriction} with the same restriction operator as $\mathsf{PAR}$ in the previous example, where note that if $f$ is a tuple of partial recursive functions, then $\overline{f}$ will also be a tuple of partial recursive functions. In other words, $\mathsf{REC}$ is a sub-restriction category of $\mathsf{PAR}$. The total maps in $\mathsf{REC}$ correspond to tuples of total recursive functions. The partial order $\leq$ is again the same as $\mathsf{PAR}$. On the other hand, the restriction idempotents for $\mathbb{N}^n$ in $\mathsf{REC}$ correspond to the recursive enumerable subsets of $\mathbb{N}^n$ \cite[Ex 2.2.3]{cockett2012range}. As such, if $U \subseteq \mathbb{N}^n$ is recursively enumerable, then $e_U: \mathbb{N}^n \to \mathbb{N}^n$ (as defined in the previous example) will be a map in $\mathsf{REC}$, and thus a restriction idempotent in $\mathsf{REC}$. Lastly, $\mathsf{REC}$ also has restriction zero maps defined in the same way as in $\mathsf{PAR}$ since the nowhere defined partial function is recursive. It is also worth mentioning that $\mathsf{REC}$ is the canonical example of a \textbf{Turing category} \cite[Ex 3.2.1]{turing-categories}, which briefly is a restriction category with an internal partial combinatory algebra satisfying a universal property which encodes computability of maps. 
\end{example}

\section{Interference}\label{sec:interference}
In this section we introduce the concept of an interference relation, which is a relation on maps of restriction category with restriction zeroes which tells us when two maps that have the same domain should be viewed as being formally disjoint. Formal disjointness must certainly imply that where one map is defined, the other is undefined, but could possibly also ask for more. This concept of disjointness plays key a role in the definition of an iteration operator, since in particular it is used to determine the termination condition of iterations. 

\begin{definition}\label{def:interference} For a restriction category $\mathbb{X}$ with restriction zeroes, an \textbf{interference relation} $\perp$ on maps is a family of relations (indexed by triples of objects $A, B, C \in \mathbb{X}$) on homsets $\perp^A_{B,C} \subseteq \mathbb{X}(A,B) \times \mathbb{X}(A,C)$, where we write $f \perp g$ when $(f,g) \in \perp^A_{B,C}$ and say that $f$ and $g$ are $\perp$\textbf{-disjoint}, and such that the following axioms hold: 
\begin{enumerate}[{\bf [{$\boldsymbol{\perp}$}.1]}]
\setcounter{enumi}{-1}
\item Zero: $1_A \perp 0$; 
\item Symmetry: If $f \perp g$, then $g \perp f$; 
\item Anti-reflexive: If $f \perp f$, then $f = 0$; 
\item Downward closed: If $f^\prime \perp g^\prime$, $f \leq f^\prime$, and $g \leq g^\prime$, then $f \perp g$; 
\item Composition: If $f \perp g$, then for all suitable maps $h$, $k$, and $k^\prime$, we have that $hfk \perp hgk^\prime$;
\item Restriction: If $f \perp g$, then $\overline{f} \perp \overline{g}$. 
\end{enumerate}
\end{definition}

It turns out that for an interference relation $\perp$, maps are $\perp$-disjoint if and only if their restrictions are $\perp$-disjoint. As such this implies that an interference relation induces a \textbf{restrictional coherence} \cite[Def 6.1]{cockett2009boolean} on parallel maps.

\begin{lemma}\label{lem:inter-rest} For an interference relation $\perp$, $f \perp g$ if and only if $\overline{f} \perp \overline{g}$.
\end{lemma}
\begin{proof} The $\Rightarrow$ direction is simply {\bf [{$\boldsymbol{\perp}$}.5]}. For the $\Leftarrow$ direction, suppose that $\overline{f} \perp \overline{g}$. Applying {\bf [{$\boldsymbol{\perp}$}.4]}, we get that $\overline{f} f \perp \overline{g} g$, which by \textbf{[R.1]} gives us that $f \perp g$, as desired. 
\end{proof}

Thus, an interference relation is completely determined by its behaviour on restriction idempotents. This means that an interference relation can equivalently be defined as a family of relations on restriction idempotents. 

\begin{definition} For a restriction category $\mathbb{X}$ with restriction zero maps, a \textbf{restrictional interference relation} $\perp$ is a family of relations (indexed by the objects of $\mathbb{X}$) between restriction idempotents $\perp_A^\mathcal{O} \subseteq \mathcal{O}(A) \times \mathcal{O}(A)$, where again we write $e \perp e^\prime$ if $(e,e^\prime) \in \perp_A^\mathcal{O}$, and such that the following axioms hold: 
\begin{enumerate}[{\bf [{$\mathcal{O}\!\boldsymbol{\perp}$}.1]}]
\setcounter{enumi}{-1}
\item Zero: $1_A \perp 0$; 
\item Symmetry: If $e \perp e^\prime$, then $e^\prime \perp e$; 
\item Anti-Reflexive: If $e \perp e$, then $e = 0$; 
\item Downward Closed: If $e^\prime_1 \perp e^\prime_2$, $e_1 \leq e^\prime_1$, and $e_2 \leq e_2^\prime$, then $e_1 \perp e_2$; 
\item Pre-Composition: If $e \perp e^\prime$ in ${\cal O}(A)$, then for any $h$ with codomain $A$, $\overline{he} \perp \overline{he^\prime}$. 
\end{enumerate}
\end{definition}

\begin{proposition}\label{prop:equiv-inter} For a restriction category with restriction zero maps, there is a bijective correspondence between interference relations on maps and restrictional interference relations. Explicitly,  
\begin{enumerate}[{\em (i)}]
\item Given an interference relation $\perp$, define the restrictional interference relation $\perp^\mathcal{O}$ as $e \perp^\mathcal{O} e^\prime$ if and only if $e \perp e^\prime$; 
\item Given a restrictional interference relation $\perp^\mathcal{O}$, define the interference relation $\perp$ as $f \perp  g$ if and only if $\overline{f} \perp^\mathcal{O} \overline{g}$. 
\end{enumerate}
\end{proposition}
\begin{proof} Starting with an interference relation $\perp$, clearly by definition $\perp^\mathcal{O}$ satisfies {\bf [{$\mathcal{O}\!\boldsymbol{\perp}$}.0]} to {\bf [{$\mathcal{O}\!\boldsymbol{\perp}$}.3]}, since they are the same as {\bf [{$\boldsymbol{\perp}$}.0]} to {\bf [{$\boldsymbol{\perp}$}.3]}. It remains to explain why {\bf [{$\mathcal{O}\!\boldsymbol{\perp}$}.4]} holds. So suppose that $e\perp^\mathcal{O} e^\prime$, meaning that $e \perp e^\prime$. Then by {\bf [{$\boldsymbol{\perp}$}.4]}, we have that $he \perp he^\prime$, and then by {\bf [{$\boldsymbol{\perp}$}.5]} that $\overline{he} \perp \overline{he^\prime}$, thus  $\overline{he} \perp^\mathcal{O} \overline{he^\prime}$. So we conclude that $\perp^\mathcal{O}$ is indeed a restrictional interference relation. 

Conversely, starting with a restrictional interference relation $\perp^\mathcal{O}$, we need to show that $\perp$ satisfies the six axioms {\bf [{$\boldsymbol{\perp}$}.0]} to {\bf [{$\boldsymbol{\perp}$}.5]}. 

\begin{enumerate}[{\bf [{$\boldsymbol{\perp}$}.1]}]
\setcounter{enumi}{-1}
\item By {\bf [{$\mathcal{O}\!\boldsymbol{\perp}$}.0]}, $1_A \perp^\mathcal{O} 0$. However, since $1_A$ and $0$ are both restriction idempotents, we also have that $\overline{1_A} \perp^\mathcal{O} \overline{0}$, which tells us that 
$1_A \perp  0$. 
\item Suppose that $f \perp  g$, meaning $\overline{f} \perp^\mathcal{O} \overline{g}$. By {\bf [{$\mathcal{O}\!\boldsymbol{\perp}$}.1]}, this means that $\overline{g} \perp^\mathcal{O} \overline{f}$, so we have that 
$g \perp  f$. 
\item Suppose that $f \perp  f$, that is $\overline{f} \perp^\mathcal{O} \overline{f}$. By {\bf [{$\mathcal{O}\!\boldsymbol{\perp}$}.2]}, this means that $\overline{f}=0$. However, by \textbf{[R.1]}, $\overline{f}=0$ implies that $f=0$. 
\item Suppose that $f^\prime \perp  g^\prime$, meaning $\overline{f^\prime} \perp^\mathcal{O} \overline{g^\prime}$, and that $f \leq f^\prime$ and $g \leq g^\prime$. This implies that $\overline{f} \leq \overline{f^\prime}$ and $\overline{g} \leq \overline{g^\prime}$ as well. So by {\bf [{$\mathcal{O}\!\boldsymbol{\perp}$}.3]}, we have that $\overline{f} \perp^\mathcal{O} \overline{g}$, and 
so $f \perp  g$. 
\item First recall that in a restriction category, we have the identity $\overline{fg} = \overline{f \overline{g}}$ \cite[Lemma 2.1.(iii)]{cockett2002restriction}. Now suppose that $f \perp  g$, meaning $\overline{f} \perp^\mathcal{O} \overline{g}$. Then for any suitable $h$, by {\bf [{$\mathcal{O}\!\boldsymbol{\perp}$}.4]} we have that $\overline{h\overline{f}} \perp^\mathcal{O} \overline{h\overline{g}}$. However, by the previous mentioned identity, we may rewrite this as $\overline{hf} \perp^\mathcal{O} \overline{hg}$. Now observe that for any suitable $k$ and $k^\prime$, $\overline{hfk} \leq \overline{hf}$ and $\overline{hgk^\prime} \leq \overline{hg}$. Then by {\bf [{$\mathcal{O}\!\boldsymbol{\perp}$}.3]}, we have that $\overline{hfk} \perp^\mathcal{O} \overline{hgk^\prime}$. As such, we have that $hfk \perp  hgk^\prime$. 
\item  Suppose that $f \perp  g$, meaning $\overline{f} \perp^\mathcal{O} \overline{g}$. However, since $\overline{f}$ and $\overline{g}$ are restriction idempotents, we can rewrite this as $\overline{\overline{f}}\perp^\mathcal{O} \overline{\overline{g}}$, meaning that $\overline{f} \perp  \overline{g}$. 
\end{enumerate}
So $\perp$ is indeed an interference relation on maps. 

It remains to explain why these constructions are inverses of each other. Clearly, starting with a restrictional interference relation $\perp^\mathcal{O}$, building its associated interference relation on maps, 
$\perp  = \{ (f,g) \vert~ \overline{f} \perp^\mathcal{O} \overline{g} \}$ and then taking the induced restrictional interference relation, $\{ (\overline{f},\overline{g}) \vert~ f \perp  g \}$ gives us back the starting restrictional interference relation $\perp^\mathcal{O}$. Conversely, starting with an interference relation on maps $\perp$, building its associated restrictional interference relation $\perp^\mathcal{O}$, and then taking the induced interference relation on maps results back to $\perp$ by Lemma \ref{lem:inter-rest}. 

\end{proof}
 
From now on, we will make no distinction between an interference relation and its corresponding restrictional interference relation and shall denote them both by $\perp$.   Here are some other basic yet important properties of interference relations:

\begin{lemma}\label{lem:interference1} Let $\perp$ be an interference relation on a restriction category $\mathbb{X}$ with restriction zeroes. Then:  
\begin{enumerate}[{\em (i)}]
\item \label{lem:interference1.0} For all maps $f$, $f \perp 0$; 
\item \label{lem:interference1.i} If $f \perp g$ then $\overline{f}g=0$, $\overline{g}f=0$, and $\overline{f}~\overline{g} = 0$;
\item \label{lem:interference1.total} If $f$ is total, then $f \perp g$ if and only if $g=0$. 
\item \label{lem:interference1.iv} If $e$ is a restriction idempotent, then $e \perp f$ if and only if $e \perp \overline{f}$.
\item \label{lem:interference1.v} If $f \perp g$, then $f \perp \overline{g}$ and $\overline{f} \perp g$. 
\end{enumerate}
\end{lemma}

In particular notice that Lemma \ref{lem:interference1}.(\ref{lem:interference1.i}) says that if maps are $\perp$-disjoint, then they are disjoint in the usual restriction category sense \cite[Prop 6.2]{cockett2009boolean}. 

\begin{proof} These are mostly straightforward. 
\begin{enumerate}[{\em (i)}]
\item By {\bf [{$\boldsymbol{\perp}$}.0]}, $1_A \perp 0$. Using {\bf [{$\boldsymbol{\perp}$}.4]}, pre-composing both sides by $f$, we get that $f \perp 0$. 
\item First note that $\overline{f}g=0$, $\overline{g}f=0$, and $\overline{f}~\overline{g} = 0$ are all equivalent conditions \cite[Prop 6.2]{cockett2009boolean}. So suppose that $f \perp g$. We will show that $\overline{f}\overline{g} = 0$ from the anti-reflexivity of $\perp$. So firstly if $f \perp g$, then by {\bf [{$\boldsymbol{\perp}$}.5]}, $\overline{f} \perp \overline{g}$. Now note that $\overline{f}~\overline{g~} \overline{f} = \overline{f}~\overline{g}$ and $\overline{f}~\overline{g}~ \overline{g} = \overline{f}~\overline{g}$. So by {\bf [{$\boldsymbol{\perp}$}.4]}, pre-composing both sides by $\overline{f}~\overline{g}$ we get that $\overline{f}~\overline{g} \perp \overline{f}~\overline{g}$. Then {\bf [{$\boldsymbol{\perp}$}.2]} implies that $\overline{f}~\overline{g} = 0$. 
\item The $\Leftarrow$ direction is simply {\bf [{$\boldsymbol{\perp}$}.0]}. So for the $\Rightarrow$ direction, suppose that $f$ is total and that $f \perp g$. Then by (\ref{lem:interference1.i}), we have that $\overline{f}g = 0$. However since $f$ is total, $\overline{f}=1$. So we have that $g=0$. 
\item This follows immediately from Lemma \ref{lem:inter-rest}. 
\item This follows from (\ref{lem:interference1.iv}) and {\bf [{$\boldsymbol{\perp}$}.5]}. 
\end{enumerate}
\end{proof}

A restriction category with restriction zeroes can have multiple different interference relations. In fact, there are always at least two possible ones: a minimal one $\perp_\delta$ and a maximal one $\perp_0$ (with respect to the usual ordering of relations). The minimal interference relation $\perp_\delta$ is the relation which only relates maps to zero, that is $f \perp_\delta g$ if and only if $f=0$ or $g=0$. On the other hand, the maximal interference relation $\perp_0$ is the relation which says that two maps are disjoint if, whenever one is defined, the other is not.  Thus, $f \perp_0 g$ if and only if $\overline{f}g =0$ (equivalently $\overline{g}f = 0$) or, in other words, they are disjoint in the sense of \cite[Prop 6.2]{cockett2009boolean}. 

\begin{lemma}\label{lemma:two-interference} For a restriction category $\mathbb{X}$ with restriction zero maps,
\begin{enumerate}[{\em (i)}]
\item $\perp_0$ is an interference relation with its associated restrictional interference relation given by $e \perp_0 e^\prime$ if and only if $ee^\prime =0$; 
\item$\perp_\delta$ is an interference relation with its associated restrictional interference relation given by $e \perp_\delta e^\prime$ if and only if  $e=0$ or $e^\prime=0$. 
\end{enumerate}
Furthermore, for every interference relation $\perp$ on $\mathbb{X}$, we have that $\perp_\delta \subseteq \perp \subseteq \perp_0$.
\end{lemma}
\begin{proof} Checking that $\perp_\delta$ is an interference relation is straightforward and is left as an exercise for the reader. We now check that $\perp_0$ is also an interference relation by proving it is a restrictional interference relation. 
\begin{enumerate}[{\bf [{$\mathcal{O}\!\boldsymbol{\perp}$}.1]}]
\setcounter{enumi}{-1}
\item By the annihilation property we have that $1_A 0 = 0$, so $1_A \perp_0 0$.
\item Recall that restriction idempotents commute, that is, $e e^\prime = e^\prime e$. Therefore, if $e \perp_0 e^\prime$, which means that $e e^\prime =0$, then $e^\prime e =0$ as well, and so $e^\prime \perp_0 e$.
\item Recall that restriction idempotents are idempotents, that is, if $e \in \mathcal{O}(A)$, then $ee=e$. Then if $e \perp_0 e$, which implies that $ee=0$, then we get that $e=0$.
\item Recall that for restriction idempotents, $e_1 \leq e^\prime_1$ and $e_2 \leq e_2^\prime$ imply that $e_1 e^\prime_1= e_1$ and $e^\prime_2 e_2 = e_2$. Then if $e^\prime_1 \perp_0 e^\prime_2$, which means $e^\prime_1 e^\prime_2 =0$, then we also get that $e_1 e_2 = e_1 e^\prime_1 e^\prime_2 e_2 = e_1 0 e_2 = 0$. So $e_1 e_2=0$ and so $e_1 \perp_0 e_2$.
\item First note that using \textbf{[R.3]} and \textbf{[R.4]}, it follows that $\overline{he}~\overline{he^\prime} = \overline{h e e^\prime}$. So if $e \perp_0 e^\prime$, which means $ee^\prime =0$, then we get that $\overline{he} \overline{he^\prime} = \overline{h e e^\prime} = \overline{h 0} = \overline{0}=0$. So have that $\overline{he} \perp_0 \overline{he^\prime}$. 
\end{enumerate}
So $\perp_0$ is a restrictional interference relation, and it is straightforward to see that its induced interference relation is precisely as defined above. So we conclude that $\perp_0$ is an interference relation. 

It remains to explain why $\perp_0$ is maximal and $\perp_\delta$ is minimal amongst interference relations. However we have already done this, since for any interference relation $\perp$, that $\perp \subseteq \perp_0$ is precisely the statement of Lemma \ref{lem:interference1}.(\ref{lem:interference1.i}), while that $\perp_\delta \subseteq \perp$ is precisely Lemma \ref{lem:interference1}.(\ref{lem:interference1.0}). 
\end{proof}

Thus, having an interference relation is chosen structure rather than a property and we say that: 

\begin{definition}\label{def:inter-rest-cat} An \textbf{interference restriction category} is pair $(\mathbb{X}, \perp)$ consisting of a restriction category $\mathbb{X}$ with restriction zeroes and a chosen a interference relation $\perp$. 
\end{definition}

Nevertheless, to avoid heavy nomenclature, we will say that $\mathbb{X}$ is a $\perp$-restriction category as a short hand for saying that $(\mathbb{X}, \perp)$ is an interference restriction category. We conclude this section with our main examples of interference relations:

\begin{example} \normalfont In $\mathsf{PAR}$, the maximal restrictional interference relation corresponds to checking when subsets are disjoint. Indeed, for subsets $U,V \subseteq X$, for their associated restriction idempotents we have that $e_U \perp_0 e_V$ if and only if $U \cap V = \emptyset$. Therefore, the corresponding maximal interference relation is given as follows: for partial functions ${f: X \to Y}$ and $g: X \to Z$, we have that $f \perp_0 g$ if and only if their domains are disjoint, $\mathsf{dom}(f) \cap \mathsf{dom}(g) = \emptyset$. In other words, $f \perp_0 g$ if $f(x) \uparrow$ whenever $g(x) \downarrow$ and $g(x) \uparrow$ whenever $f(x) \downarrow$. On the other hand, the minimal restrictional interference relation corresponds to checking if either subset is empty. So for $U,V \subseteq X$, for their associated restriction idempotents we have that $e_U \perp_\delta e_V$ if and only if $U= \emptyset$ or $V=\emptyset$. As such, for partial functions ${f: X \to Y}$ and ${g: X \to Z}$, we have that $f \perp_\delta g$ if and only if either $f$ or $g$ is nowhere define, that is, if $\mathsf{dom}(f) = \emptyset$ or $\mathsf{dom}(g) = \emptyset$. 
\end{example}

\begin{example} \normalfont Here is an example of an interference relation on $\mathsf{REC}$ which is neither the maximal one $\perp_0$ nor the minimal one $\perp_\delta$. Recall that a subset $U \subseteq \mathbb{N}^n$ is recursive (or computable) if the characteristic function $\chi_U: \mathbb{N}^n \to \mathbb{N}$ defined as $\chi_U(x) = 1$ if $x \in U$ and $\chi_X(x)=0$ if $x \notin U$ is a total recursive function. Then two subsets $U,V \subseteq \mathbb{N}^n$ are said to be \textbf{recursively separable} if there exists a recursive set $W \subseteq \mathbb{N}^n$ such that $U \subseteq W$ and $V \cap W = \emptyset$. Then define the restrictional interference relation $\perp_r$ for the associated restriction idempotents of recursively enumerable subsets $U,V \subseteq \mathbb{N}^n$ as $e_U \perp_r e_V$ if and only if $U$ and $V$ are recursively separable. Then the induced interference relation $\perp_r$ is given as follows: for two maps $F: \mathbb{N}^n \to \mathbb{N}^m$ and $G: \mathbb{N}^n \to \mathbb{N}^k$ in $\mathsf{REC}$, we have that $F \perp_r G$ if their domains $\mathsf{dom}(F)$ and $\mathsf{dom}(G)$ are recursively separable. Then $\perp_r$ is an interference coherence on $\mathsf{REC}$. On the one hand, clearly $\perp_r$ is not the minimal one $\perp_\delta$. On the other hand, note that $\perp_0$ on $\mathsf{REC}$ is given the same way as in $\mathsf{PAR}$, which means that $\perp_r$ and $\perp_0$ are indeed different since there exists disjoint recursively enumerable sets which are recursively inseparable (i.e. not recursively separable). 
\end{example}

\section{Disjoint Joins}\label{sec:disjoint-joins}

In this section we discuss the notion of \emph{disjoint joins} in an interference restriction category. Since the homsets of a restriction category are posets, we can consider the \emph{join} \cite[Def 6.7]{cockett2009boolean} of suitable parallel maps to be the join in usual sense, meaning the least upper bound with respect to $\leq$. However, for the story of this paper, we do not need all joins to exists like in a join restriction category \cite[Def 10.1]{cockett2009boolean}, but instead only require the joins of $\perp$-disjoint maps. As such, we call such joins: $\perp$\emph{-disjoint} joins, or simply $\perp$-joins, and denote them using $\sqcup$ instead of $\vee$, to further emphasize that the maps we are working with are intuitively disjoint. Readers familiar with joins in a restriction category may recall that joins can only be defined for maps that are \textbf{compatible} \cite[Prop 6.3]{cockett2009boolean}. Recall that parallel maps $f$ and $g$ are compatible, written as $f \smile g$, if $\overline{f}g = \overline{g}f$. However, by Lemma \ref{lem:interference1}.(\ref{lem:interference1.i}), if $f \perp g$, then $\overline{g}f=0=\overline{f}g$, and so $f \smile g$. Therefore, we can indeed consider the join of parallel $\perp$-disjoint maps. More generally, one can define define joins with respect to a restrictional coherence \cite[Def 6.7]{cockett2009boolean}. Thus, since every interference relation is a restrictional coherence, the following definition is \cite[Def 6.7]{cockett2009boolean} for the special case of an interference relation. 

\begin{definition}\label{def:disjoint-join} In a $\perp$-restriction category $\mathbb{X}$, the $\perp$\textbf{-join} (if it exists) of a family of pairwise $\perp$-disjoint parallel maps ${\lbrace f_i: A \to B \rbrace_{i \in I}} \subseteq \mathbb{X}(A,B)$ (where $I$ is an arbitrary index set and $f_i \perp f_j$ for all $i,j \in I$ with $i \neq j$) is a (necessarily unique) map ${\bigsqcup \limits_{i\in I} f_i: A \to B}$ such that: 
\begin{enumerate}[{\bf [$\boldsymbol{\sqcup}$.1]}]
\item $f_i \leq \bigsqcup \limits_{i\in I} f_i$ for all $i \in I$;
\item If $g: A \to B$ is a map such that $f_i \leq g$ for all $i \in I$, then $\bigsqcup \limits_{i\in I} f_i \leq g$. 
\end{enumerate}
\end{definition}

Note that $\perp$-joins are necessarily unique since {\bf [$\boldsymbol{\sqcup}$.1]} and {\bf [$\boldsymbol{\sqcup}$.2]} say that they are least upper-bounds, which are always unique. For a finite family indexed by $I= \lbrace 0, \hdots, n+1 \rbrace$, we write their $\perp$-join as $\bigsqcup \limits_{i\in I} f_i = f_0 \sqcup ... \sqcup f_{n+1}$. In the special case of a singleton $\lbrace f \rbrace$, its $\perp$-join is just $f$, while for the case of the empty set, the $\perp$-join of nothing is simply the zero map. When the $\perp$-joins of all suitable families of $\perp$-disjoint maps exists, we may also ask that they be compatible with composition. In \cite[Def 6.7]{cockett2009boolean}, preservations of joins by pre-composition was called being \textbf{stable}, while preservations of joins by post-composition was called being \textbf{universal}. It turns out that being stable implies being universal \cite[Lemma 6.10.(2)]{cockett2009boolean}, so we only need to ask that pre-composition preserves $\perp$-joins. Since for most of the story of this paper, we only really need \emph{binary} $\perp$-joins, we take the time to define this case explicitly. 

\begin{definition} A $\perp$-restriction category $\mathbb{X}$ is said to have \textbf{all binary $\perp$-joins} if for every pair of $\perp$-disjoint parallel maps $f: A \to B$ and $g: A \to B$ (i.e. $f \perp g$) their $\perp$-join $f \sqcup g: A \to B$ exists and: 
\begin{enumerate}[{\bf [$\boldsymbol{\sqcup}$.1]}]
\setcounter{enumi}{2}
\item For any map $h: A^\prime \to A$, $h \left( f \sqcup g \right) = hf \sqcup hg$. 
\end{enumerate}
Similarly, a $\perp$-restriction category $\mathbb{X}$ is said to have \textbf{all (finite) $\perp$-joins} if for every (finite) family of pairwise $\perp$-disjoint parallel maps ${\lbrace f_i: A \to B \rbrace_{i \in I}}$, their $\perp$-join ${\bigsqcup \limits_{i\in I} f_i: A \to B}$ exists and: 
\begin{enumerate}[{\bf [$\boldsymbol{\sqcup}$.1]}]
\setcounter{enumi}{2}
\item For any map $h: A^\prime \to A$, $h \bigsqcup \limits_{i\in I} f_i = \bigsqcup \limits_{i\in I} h f_i$. 
\end{enumerate}
\end{definition}

Since a $\perp$-join is a join in the sense of \cite[Def 6.7]{cockett2009boolean}, it therefore satisfies the same properties, some of which can be found in \cite[Prop 2.14]{cockett2012differential} and \cite[Lemma 4.3]{cockett2023classical}. In the result below, which  highlights some useful identities, $I$ is an arbitrary indexing set: if $I$ has size two we recover the binary case, if $I$ is finite we recover the finite case.

\begin{lemma}\label{lemma:disjoint-join} Let $\mathbb{X}$ be a $\perp$-restriction category with all (binary/finite) $\perp$-joins, then: 
\begin{enumerate}[{\em (i)}]
\item \label{lemma:join-<} $\overline{f_j} \bigsqcup \limits_{i\in I} f_i = f_j$ for all $j \in I$; 
\item \label{lemma:join-rest} $\overline{\bigsqcup \limits_{i\in I} f_i} = \bigsqcup \limits_{i\in I} \overline{f_i}$; 
\item \label{lemma:join-comp} $k \left(\bigsqcup \limits_{i\in I} f_i \right)k^\prime = \bigsqcup \limits_{i\in I} k f_i k^\prime$; 
\item \label{lemma:join-inter} If $\bigsqcup \limits_{i\in I} f_i \perp h$ then $f_j \perp h$ for all $j \in I$. 
\end{enumerate}
\end{lemma}
\begin{proof} The first three are simply \cite[Prop 2.14.(i)-(iii)]{cockett2012differential} for the special case of our $\perp$-joins, and are proven using the same calculations. For (\ref{lemma:join-inter}), suppose that $\bigsqcup \limits_{i\in I} f_i \perp h$. By {\bf [$\boldsymbol{\sqcup}$.1]}, $f_j \leq \bigsqcup \limits_{i\in I} f_i$, and so by {\bf [{$\boldsymbol{\perp}$}.3]}, we get that $f_j \perp h$. 
\hfill \end{proof}

It is important to note that while having finite $\perp$-joins obviously implies having binary $\perp$-joins, the converse is not necessarily true. Indeed, having binary $\perp$-joins does not imply we can build arbitrary finite $\perp$-joins since the converse of Lemma \ref{lemma:disjoint-join}.(\ref{lemma:join-inter}) does not hold in general. To see this, suppose we have three parallel maps such that $f \perp g$, $g \perp h$, and $f \perp h$. Then we get their binary $\perp$-joins $f \sqcup g$, $g \sqcup h$, and $f \sqcup h$. However since we do not necessarily have that say $f \sqcup g$ and $h$ are $\perp$-disjoint, we wouldn't be able to construct $\left( f \sqcup g \right) \sqcup h$. Thus, we also need to ask that $\perp$-joins preserve $\perp$-disjointness. 

\begin{definition}\label{def:perpdisjrestcat} In a $\perp$-restriction category $\mathbb{X}$, if the $\perp$-join ${\bigsqcup \limits_{i\in I} f_i: A \to B}$ of a family of pairwise $\perp$-disjoint parallel maps ${\lbrace f_i: A \to B \rbrace_{i \in I}}$ exists, then we say that $\bigsqcup \limits_{i\in I} f_i$ is \textbf{strong} if:
\begin{enumerate}[{\bf [$\boldsymbol{\sqcup}$.1]}]
\setcounter{enumi}{3}
\item For every map $h: A \to C$ such that $f_i \perp h$ for all $i \in I$, then $\bigsqcup \limits_{i\in I} f_i \perp h$. 
\end{enumerate}
A \textbf{(finitely) disjoint $\perp$-restriction category} is a $\perp$-restriction category $\mathbb{X}$ that has all (finite) $\perp$-joins and every (finite) $\perp$-join is strong. 
\end{definition}

We can now show that having strong binary $\perp$-joins is equivalent to having strong finite $\perp$-joins. 

\begin{lemma}\label{lemma:binary-finite} Let $\mathbb{X}$ be a $\perp$-restriction category. Then $\mathbb{X}$ is a finitely disjoint $\perp$-restriction category if and only if $\mathbb{X}$ has all binary $\perp$-joins and every binary $\perp$-join is strong. 
\end{lemma}
\begin{proof} The $\Leftarrow$ direction is automatic. For the $\Rightarrow$ direction, suppose we have all  binary $\perp$-joins and that they are strong. As mentioned above, the $\perp$-join of the empty family is the zero map, the $\perp$-join of a singleton is the map itself, and by assumption we also have binary $\perp$-joins. So we need to show that for any finite family of pairwise $\perp$-disjoint parallel maps of size greater than two, their $\perp$-join exists, is preserved by pre-composition, and is strong. We do so by induction, where the base case is a family of size $2$, which we already know holds. So suppose that we can build $n$-ary $\perp$-joins which are preserved by pre-composition {\bf [$\boldsymbol{\sqcup}$.3]} and strong {\bf [$\boldsymbol{\sqcup}$.4]}. Let $\lbrace f_0, f_1, \hdots, f_n, f_{n+1}\rbrace$ be a family of pairwise $\perp$-disjoint parallel maps. Note that $\lbrace f_0, f_1, \hdots, f_n\rbrace$ is also a family of pairwise $\perp$-disjoint parallel maps, so by the induction hypothesis, we get its $\perp$-join $f_0 \sqcup \hdots \sqcup f_n$. By {\bf [$\boldsymbol{\sqcup}$.4]}, we get that $f_0 \sqcup \hdots \sqcup f_n \perp f_{n+1}$, so we can define $f_0 \sqcup \hdots \sqcup f_n \sqcup f_{n+1}$ as the binary join $f_0 \sqcup \hdots \sqcup f_n \sqcup f_{n+1} := \left(f_0 \sqcup \hdots \sqcup f_n \right) \sqcup f_{n+1}$. It is straightforward to check that $f_0 \sqcup \hdots \sqcup f_n \sqcup f_{n+1}$ satisfies {\bf [$\boldsymbol{\sqcup}$.1]} and {\bf [$\boldsymbol{\sqcup}$.2]}, so $f_0 \sqcup \hdots \sqcup f_n \sqcup f_{n+1}$ is indeed the $\perp$-join of $\lbrace f_0, f_1, \hdots, f_n, f_{n+1}\rbrace$. Now for any suitable map $h$, by {\bf [$\boldsymbol{\sqcup}$.3]} for the binary case and the induction hypothesis, we can also easily check that $h\left( f_0 \sqcup \hdots \sqcup f_n \sqcup f_{n+1} \right) = hf_0 \sqcup \hdots \sqcup hf_n \sqcup hf_{n+1}$ as desired. Moreover, if $k$ is a map such that $k \perp f_i$, then by the induction hypothesis for {\bf [$\boldsymbol{\sqcup}$.4]}, we get that $k \perp f_0 \sqcup \hdots \sqcup f_n$. Then since $k \perp f_{n+1}$ as well, by {\bf [$\boldsymbol{\sqcup}$.4]} for the binary case, we get that $k \perp f_0 \sqcup \hdots \sqcup f_n \sqcup f_{n+1}$. So we conclude that $\mathbb{X}$ is a finitely disjoint $\perp$-restriction category. 
\end{proof}

Before finally giving examples, let us quickly take a look at disjoint joins for the maximal and minimal interference relations. We first show that for the maximal relation $\perp_0$, every $\perp_0$-join is strong.  In particular this means that if we have binary $\perp_0$-joins, then we have all finite $\perp_0$-joins as well. A finitely disjoint $\perp_0$-restriction category was called an \textbf{$\perp$-restriction category} in \cite[Def 7.5]{cockett2009boolean} (not be confused with our use of $\perp$-restriction category for an interference restriction category which may not have $\perp$-joins). 

\begin{lemma} Let $\mathbb{X}$ be a restriction category with restriction zeroes. Then any $\perp_0$-join which exists is strong. Therefore, $\mathbb{X}$ is a finitely disjoint $\perp_0$-restriction category if and only if $\mathbb{X}$ has all binary $\perp_0$-joins. 
\end{lemma}
\begin{proof} Let $\lbrace f_i \rbrace$ be a family of pairwise $\perp_0$-disjoint parallel maps whose $\perp_0$-join $\bigsqcup \limits_{i\in I} f_i$ exists. Now suppose we also have another map $h$ such that $h \perp_0 f_i$ for all $i$, which recall is equivalent to saying that $\overline{h}f_i = 0$. Then by {\bf [$\boldsymbol{\sqcup}$.3]}, we get that $\overline{h}\bigsqcup \limits_{i\in I} f_i = \bigsqcup \limits_{i\in I} \overline{h}f_i = \bigsqcup \limits_{i\in I} 0 =0$. So $h \perp \bigsqcup \limits_{i\in I} f_i$. So we get that every $\perp_0$-join is strong. Therefore, by Lemma \ref{lemma:binary-finite}, having binary $\perp_0$-joins is sufficient to get all finite $\perp_0$-joins.
\end{proof}

On the other hand, for the minimal interference relation $\perp_\delta$, all $\perp_\delta$-joins exist trivially and are trivially strong.  

\begin{lemma}\label{lemma:minimal-strong-comp} Let $\mathbb{X}$ be a restriction category with restriction zeroes. Then $\mathbb{X}$ is a disjoint $\perp_\delta$-restriction category. 
\end{lemma}
\begin{proof} We will first explain why every family of pairwise $\perp_\delta$-disjoint parallel maps can have at most two elements. Let ${\lbrace f_i: A \to B \rbrace_{i \in I}}$ be a family of pairwise $\perp_\delta$-disjoint parallel maps. Since $f_i \perp_\delta f_j$ for all $i \neq j$, this implies that there is at most one non-zero map $f \in \lbrace f_i: A \to B \rbrace_{i \in I}$ and the rest must be zero. When $\vert I \vert \geq 1$, if such a non-zero $f$ exists, our family is reduced to $\lbrace f, 0 \rbrace$, and if no non-zero map exists, then our family is reduced to $\lbrace 0 \rbrace$. Thus the only possible such families is the empty set $\emptyset$, singletons of every map $\lbrace f \rbrace$, and a two element set $\lbrace f, 0 \rbrace$ where here $f \neq 0$. Clearly, the $\perp_\delta$-join of $\emptyset$ is $0$; of $\lbrace f \rbrace$ is $f$; and of $\lbrace f, 0 \rbrace$ is $f$. So we conclude that all $\perp_\delta$-joins exists. It is straightforward to check that {\bf [$\boldsymbol{\sqcup}$.3]} and {\bf [$\boldsymbol{\sqcup}$.4]} also hold. 
\end{proof}

Here are now the disjoint joins in our main examples. 

\begin{example} \label{ex:PAR-disjoint} \normalfont $\mathsf{PAR}$ has all $\perp_0$-joins. For an $i\in I$ indexed family of parallel partial functions ${f_i: X \to Y}$ which are pairwise $\perp_0$-disjoint, which recall means that when $f_i(x) \downarrow$ then $f_j(x) \uparrow$ for all $j\neq i$, their $\perp_0$-join is the partial function $\bigsqcup \limits_{i\in I} f_i: X \to Y$ defined as follows:
\[ \left(\bigsqcup \limits_{i\in I} f_i \right)(x) =  \begin{cases} f_i(x) & \text{if there exists a (necessarily unique) $i\in I$ such that } f_i(x) \downarrow \\
\uparrow & \text{otherwise}  \end{cases}
\]
In particular, this means that $\mathsf{PAR}$ is a disjoint $\perp_0$-restriction category.
\end{example}

\begin{example}  \label{ex:REC-disjoint} \normalfont $\mathsf{REC}$ has all $\perp_r$-joins, which are defined as in the previous example. So $\mathsf{REC}$ is a disjoint $\perp_r$-restriction category.
\end{example}

We conclude this section by explaining how every interference restriction category embeds into a disjoint interference restriction category. This is similar to the join construction given in \cite[Sec 10]{cockett2009boolean}, but which is simplified since we only want disjoint joins and not all joins. So let $\mathbb{X}$ be a $\perp$-restriction category. For a family of pairwise $\perp$-disjoint parallel maps $F \subseteq \mathbb{X}(A,B)$, let $\downarrow \!\! F = \lbrace g: A \to B \vert~ \exists f \in F \text{ s.t. } g \leq f \rbrace$. Note that $\downarrow \!\! F$ need not be itself a family of pairwise $\perp$-disjoint parallel maps, and also that it is never empty since we always have that $0 \in~ \downarrow \!\! F$. Then, abusing notation slightly, define the disjoint $\perp$-restriction category $\mathsf{DJ}\left[ (\mathbb{X},\perp) \right]$ as follows: 
\begin{enumerate}[{\em (i)}]
\item The objects of $\mathsf{DJ}\left[ (\mathbb{X},\perp) \right]$ are the same  as $\mathbb{X}$; 
\item A map $A \to B$ in $\mathsf{DJ}\left[ (\mathbb{X},\perp) \right]$ is a $S = \downarrow \!\! F$ where $F \subseteq \mathbb{X}(A,B)$ is a family of pairwise $\perp$-disjoint parallel maps; 
\item Identity maps are the set of restriction idempotents $\mathcal{O}(A): A \to A$;
\item Composition of $S: A \to B$ and $T: B \to C$ is defined as $ST = \lbrace fg \vert~ f \in S, g\in T \rbrace$;
\item The restriction of $S: A \to B$ is defined as the set of the restrictions, $\overline{S} = \lbrace \overline{f} \vert~ f\in S \rbrace$;
\item The restriction zero maps are the singletons $\lbrace 0 \rbrace: A \to B$;
\item The interference relation is defined as $S \perp T$ if for all $f\in S$ and $g \in T$, $f \perp g$;
\item For a family index by a set $I$ of pairwise $\perp$-disjoint parallel maps $S_i: A \to B$, their $\perp$-join is defined as their union $\bigsqcup\limits_{i \in I} S_i := \bigcup\limits_{i \in I} S_i$.
\end{enumerate}

Now since we want to give an embedding, we also need to explain what we mean by a functor between interference restriction categories. So if $\mathbb{X}$ and $\mathbb{Y}$ are restriction categories, then a \textbf{restriction functor} \cite[Sec 2.2.1]{cockett2002restriction} is a functor ${\mathcal{F}: \mathbb{X} \to \mathbb{Y}}$ which preserves the restrictions, $\overline{\mathcal{F}(f)} = \mathcal{F}\left( \overline{f} \right)$. If $(\mathbb{X},\perp)$ and $(\mathbb{Y},\perp)$ are interference restriction categories, then an \textbf{interference restriction functor} $\mathcal{F}: (\mathbb{X},\perp) \to (\mathbb{Y},\perp)$ is a restriction functor $\mathcal{F}: \mathbb{X} \to \mathbb{Y}$ which preserves the restrictions zeroes, $\mathcal{F}(0) = 0$, and also preserves the interference relation in the sense that if $f \perp g$, then $\mathcal{F}(f) \perp \mathcal{F}(g)$. If there is no confusion, we will simply write an interference restriction functor as ${\mathcal{F}: \mathbb{X} \to \mathbb{Y}}$ and say that it is a $\perp$-restriction functor. 

\begin{proposition}\label{prop:construction1} For a $\perp$-restriction category $\mathbb{X}$, $ \mathsf{DJ}\left[ (\mathbb{X},\perp) \right]$ is a disjoint $\perp$-restriction category. Moreover $\mathcal{I}: (\mathbb{X},\perp) \to \mathsf{DJ}\left[ (\mathbb{X},\perp) \right]$, defined on objects as $\mathcal{I}(A) = A$ and on maps $\mathcal{I}(f) = \downarrow \!\! \lbrace f \rbrace$, is a faithful $\perp$-restriction functor. 
\end{proposition}
\begin{proof} We must first explain why this construction is well-defined. This follows from the fact that this construction is a subcategory of the infinite version of the join completion given in \cite[Sec 10]{cockett2009boolean}\footnote{In \cite[Sec 10]{cockett2009boolean}, the construction is given for finite families, however it is not difficult to see that the constructions and results still hold for arbitrary families.}. First note that $e \leq 1_A$ if and only if $e$ is a restriction idempotent, and therefore $\downarrow\!\!\lbrace 1_A \rbrace = \mathcal{O}(A)$, so identity maps are well-defined. Next, observe that if $F \subseteq \mathbb{X}(A,B)$ and $G \subseteq \mathbb{X}(B,C)$ are families of parallel $\perp$-maps, then by {\bf [{$\boldsymbol{\perp}$}.4]} and {\bf [{$\boldsymbol{\perp}$}.5]} respectively, we have that $FG = \lbrace fg \vert~ f\in F, g \in G \rbrace \subseteq \mathbb{X}(A,C)$ and $\overline{F} = \lbrace \overline{f} \vert~ f \in F \rbrace \subseteq \mathbb{X}(A,A)$ are as well. Then using the same arguments as in the proof of \cite[Prop 10.3]{cockett2009boolean}, we have that if $S = \downarrow \!\! F$ and $T = \downarrow \!\! G$, then $ST = \downarrow \!\! FG = \lbrace fg \vert~ f\in F, g \in G \rbrace$ and $\overline{S} = \downarrow \! \! \overline{F}$, so composition and restrictions are well-defined. Thus by similar arguments as in the proof of \cite[Prop 10.3]{cockett2009boolean}, we have that $\mathsf{DJ}\left[ (\mathbb{X},\perp) \right]$ is a restriction category. Now since the only maps less than $0$ is $0$, we have that $\downarrow \!\! \lbrace 0 \rbrace = \lbrace 0 \rbrace = Z$. Then it immediately follows that the $Z$s are indeed the restriction zeroes for $\mathsf{DJ}\left[ (\mathbb{X},\perp) \right]$. 

Now let us explain why $\perp$ is an interference relation. 
\begin{enumerate}[{\bf [{$\boldsymbol{\perp}$}.1]}]
\setcounter{enumi}{-1}
\item By Lemma \ref{lem:interference1}.(\ref{lem:interference1.0}), we have that for all $e \in \mathcal{O}(A)$, $e \perp 0$. Thus $\mathcal{O}(A) \perp \lbrace 0 \rbrace$. 

\item Suppose that $S \perp T$, that is, all $f \in S$ and $g \in T$, we have that $f \perp g$. Then by {\bf [{$\boldsymbol{\perp}$}.1]} we also have that $g \perp f$ for all $g \in T$ and $f \in S$. So $T \perp S$. 

\item Suppose that $S \perp S$, that is, for all $f \in S$, $f \perp f$. But then by {\bf [{$\boldsymbol{\perp}$}.2]}, $f =0$ for all $f \in S$. So $S = \lbrace 0 \rbrace$.  

\item By the same arguments as in the proof of \cite[Prop 10.5]{cockett2009boolean}, we get that $S \leq S^\prime$ in $\mathsf{DJ}\left[ (\mathbb{X},\perp) \right]$ if and only if $S \subseteq S^\prime$. So now suppose that $S^\prime \perp T^\prime$, and also that $S \leq S^\prime$ and $T \leq T^\prime$, or in other words, $S \subseteq S^\prime$ and $T \subseteq T^\prime$. Then for all $f \in S$ and $g \in T$, we also have that $f \in S^\prime$ and $g \in T^\prime$, and thus $f \perp g$. So $S \perp T$.  

\item Suppose that $S \perp T$, that is, all $f \in S$ and $g \in T$, we have that $f \perp g$. Now given suitable $U = \downarrow H$ and $V = \downarrow K$, consider $USV$ and $UTV$. Then a $x \in USV$ is of the form $u f v$ for some $u \in U$, $f \in S$, and $v \in V$, and similarly a $y \in UTV$ is of the form $x = u^\prime g v^\prime$ for some $u^\prime \in U$, $g \in T$, and $v^\prime \in V$. Now since $u, u^\prime \in U$, there are $h, h^\prime \in H$ such that $u \leq h$ and $u^\prime \leq h^\prime$. If $h \neq h^\prime$, then $h \perp h^\prime$, and then by {\bf [{$\boldsymbol{\perp}$}.3]}, we have that $u \perp u^\prime$. So by {\bf [{$\boldsymbol{\perp}$}.4]}, we get that $x = u f v \perp u^\prime g v^\prime = y$. Now consider instead if $h = h^\prime$. Now since $f \perp g$, by {\bf [{$\boldsymbol{\perp}$}.4]} we get $h f v \perp h g v^\prime$. Then since $u \leq h$ and $u\prime \leq h$, we get that $u f v \leq h fv$ and $u^\prime g v^\prime \leq h^\prime g v^\prime$, so by {\bf [{$\boldsymbol{\perp}$}.3]} we get that $x = u f v \perp u^\prime g v^\prime = y$. So $USV \perp UTV$. 

\item Suppose that $S \perp T$, that is, all $f \in S$ and $g \in T$, that $f \perp g$. Then by {\bf [{$\boldsymbol{\perp}$}.5]} we also have that $\overline{f} \perp \overline{g}$ for all $\overline{f} \in \overline{S}$ and $\overline{g} \in \overline{T}$. So $\overline{S} \perp \overline{T}$. 
\end{enumerate}

Thus $\mathsf{DJ}\left[ (\mathbb{X},\perp) \right]$ is a $\perp$-restriction category. It remains to discuss the disjoint joins. Now given a family of families of parallel $\perp$-disjoint maps $F_i \subseteq \mathbb{X}(A,B)$ such that $F_i \perp F_j$ for all $i \neq j$ (that is, $f \perp g$ for all $f \in F_i$ and $g \in F_j$), then their union $\bigcup\limits_{i \in I} F_i \subseteq \mathbb{X}(A,B)$ is also a family of pairwise $\perp$-disjoint parallel maps. Then by the same arguments as in the proof of \cite[Prop 10.5]{cockett2009boolean}, we see that for a family of pairwise $\perp$-disjoint maps $S_i = \downarrow \!\! F_i$, that $\bigsqcup\limits_{i \in I} S_i = \downarrow\left(\bigcup\limits_{i \in I} F_i  \right)$. By using again the same arguments as in the proof of \cite[Prop 10.5]{cockett2009boolean}, we get that {\bf [$\boldsymbol{\sqcup}$.1]}, {\bf [$\boldsymbol{\sqcup}$.2]}, and {\bf [$\boldsymbol{\sqcup}$.3]} hold. Finally, it remains to show {\bf [$\boldsymbol{\sqcup}$.4]}. So consider family of pairwise $\perp$-disjoint maps $S_i$ and another map $T$ such that $S_i \perp T$ for all $i \in I$. Then clearly $\bigcup\limits_{i \in I} S_i \perp T$ as well, and so {\bf [$\boldsymbol{\sqcup}$.4]} holds. So we conclude that $\mathsf{DJ}\left[ (\mathbb{X},\perp) \right]$ is a disjoint $\perp$-restriction category.

Finally, it remains to show that $\mathcal{I}$ is an interference restriction functor. By the same arguments as in the proof of \cite[Thm 10.6]{cockett2009boolean}, we get that $\mathcal{I}$ is a faithful restriction functor. Moreover it preserves restriction zeroes since $\mathcal{I}(0) = \downarrow \!\! \lbrace 0 \rbrace = Z$. Now suppose that $f \perp g$ in $(\mathbb{X},\perp)$. Then by {\bf [$\boldsymbol{\sqcup}$.3]}, for all $f^\prime \leq f$ and $g \prime \leq g$, we get that $f^\prime \perp g^\prime$. Thus $\mathcal{I}(f) = \downarrow \!\! \lbrace f \rbrace \perp \downarrow \!\! \lbrace 
g \rbrace = \mathcal{I}(g)$, and so we get that $\mathcal{I}$ is indeed a $\perp$-restriction functor.
\end{proof}

\section{Kleene Wands}\label{sec:itegories}

In this section, we introduce the notion of a \emph{Kleene wand} for an interference restriction category with binary disjoint joins by adapting the definition of Kleene wand from \cite[Sec 4]{cockett2012timed}. The main result of this paper is that for an extensive restriction category, to have an iteration (or trace) operator on the coproduct is equivalent to having a Kleene wand. As such, the axioms of a Kleene wand are analogues to of that of an iteration operator. However, while iteration operators require the presence of coproducts, a Kleene wand can be defined without them. Thus Kleene wands capture iteration in settings without coproducts. An \emph{itegory} is an interference restriction category with binary disjoint joins and Kleene wands. It is also worth noting that the definition of a Kleene wand in \cite[Sec 4]{cockett2012timed} was for the maximal interference relation $\perp_0$ and in a setting with \emph{restriction products} \cite[Sec 4.1]{cockett2007restriction}. So here we adapt some of the axioms for an arbitrary interference relation but do not consider having restriction products. 

\begin{definition}\label{def:wand} For a $\perp$-restriction category $\mathbb{X}$ with all binary $\perp$-joins, a \textbf{Kleene wand} $\wand$ is a family of operators (indexed by pairs of objects $X, A \in \mathbb{X}$):
\begin{align*}
\wand: \{ (f,g) \in \mathbb{X}(X,X) \times \mathbb{X}(X,A) | f \perp g \} \to \mathbb{X}(X,A) && \begin{matrix}[c] \infer{f \wand g : X \to A}{f: X \to X ~~~ g: X \to A ~~~ f \perp g} \end{matrix}
\end{align*} 
such that the following axioms hold: 
\begin{enumerate}[{\bf [$\boldsymbol{\wand}$.1]}]
\item Iteration: For maps $f: X \to X$ and $g: X \to A$, if $f \perp g$, then $g \sqcup f(f \wand g) = f \wand g$
\item Naturality: For maps $f: X \to X$, $g: X \to A$, and $h: A \to B$, if $f \perp g$, then $\left( f \wand g \right)h = f \wand gh$
\item Dinaturality: For maps $f: X \to Y$, $g: X \to A$, and $k: Y \to X$, if $f \perp g$ then $kf \wand kg = k\left( fk \wand g \right)$
\item Diagonal: For maps $f: X \to X$, $f^\prime: X \to X$, and $g: X \to A$, if $f \perp f^\prime$ and $f \sqcup f^\prime \perp g$, then $(f \wand f^\prime) \perp (f \wand g)$ and $(f \sqcup f^\prime) \wand g = (f \wand f^\prime) \wand (f \wand g)$. 
\end{enumerate}
A \textbf{$\perp$-itegory} is a pair $(\mathbb{X}, \wand)$, consisting of a $\perp$-restriction category $\mathbb{X}$ with all binary $\perp$-joins and a Kleene wand $\wand$. 
\end{definition}

We note that {\bf [$\boldsymbol{\wand}$.4]} is well defined by Lemma \ref{lemma:disjoint-join}.(\ref{lemma:join-inter}). The axioms of a Kleene wand are analogues of the axioms of an iteration operator, but expressed in a setting without coproducts. We first observe that if we have strong disjoint joins, then the fourth axiom can be simplified. 

\begin{lemma}\label{lem:strong-wand-alt} For a finitely disjoint $\perp$-restriction category $\mathbb{X}$, $\wand$ is a Kleene wand if and only if $\wand$ satisfies {\bf [$\boldsymbol{\wand}$.1]}, {\bf [$\boldsymbol{\wand}$.2]}, {\bf [$\boldsymbol{\wand}$.3]}, and the following: 
\begin{enumerate}[{\bf [$\boldsymbol{\wand}$.1.{a}]}]
\setcounter{enumi}{3}
\item Diagonal: For maps $f: X \to X$, $f^\prime: X \to X$, and $g: X \to A$, if $f \perp f^\prime$, $f \perp g$, and $f^\prime \perp g$, then $(f \sqcup f^\prime) \wand g = (f \wand g) \wand (f^\prime \wand g)$
\end{enumerate}
\end{lemma}
\begin{proof} For the $\Rightarrow$ direction, we need only to explain why {\bf [$\boldsymbol{\wand}$.4.a]} holds. So suppose that $f \perp f^\prime$, $f \perp g$, and $f^\prime \perp g$. By {\bf [$\boldsymbol{\sqcup}$.4]}, we get that $f \sqcup f^\prime \perp g$. Then we can apply {\bf [$\boldsymbol{\wand}$.4]} to get that $(f \wand f^\prime) \perp (f \wand g)$ and the rest of the statement of {\bf [$\boldsymbol{\wand}$.4.a]}. For the $\Leftarrow$ direction, we need to explain why {\bf [$\boldsymbol{\wand}$.4]} holds. So suppose that $f \sqcup f^\prime \perp g$. We first need to explain why $(f \wand f^\prime)$ and $(f \wand g)$ are $\perp$-disjoint. To do so, first note that by Lemma \ref{lemma:disjoint-join}.(\ref{lemma:join-inter}), we get that $f \perp f^\prime$, $f \perp g$, and $f^\prime \perp g$. Then observe that by using \textbf{[R.1]} and \textbf{[{$\boldsymbol{\wand}$}.2]}, we can compute that $f \wand g = f \wand (\overline{g} g) = (f \wand \overline{g})g$, so $f \wand g = (f \wand \overline{g})g$ and similarly $f \wand f^\prime = (f \wand \overline{f^\prime})f^\prime$. Now by Lemma \ref{lem:interference1}.(\ref{lem:interference1.v}), since $f \perp f^\prime$ and $f \perp g$, then $f \perp \overline{f^\prime}$ and $f \perp \overline{g}$. Moreover by {\bf [$\boldsymbol{\sqcup}$.4]}, we get that $f \perp \overline{f^\prime} \sqcup \overline{g}$, so we can consider $f \wand (\overline{f^\prime} \sqcup \overline{g})$. Now note that by Lemma \ref{lem:interference1}.(\ref{lem:interference1.i}) and Lemma \ref{lemma:disjoint-join}.(\ref{lemma:join-comp}), we get that $(\overline{f^\prime} \sqcup \overline{g})g = g$ and $(\overline{f^\prime} \sqcup \overline{g})f^\prime = f^\prime$. Therefore using these identities and \textbf{[{$\boldsymbol{\wand}$}.2]}, we also get that $\left(f \wand (\overline{f^\prime} \sqcup \overline{g})\right)f^\prime = f \wand f^\prime$ and $\left(f \wand (\overline{f^\prime} \sqcup \overline{g})\right)g = f \wand g$. Now since $f^\prime \perp g$, by {\bf [{$\boldsymbol{\perp}$}.4]}, we then get that $f \wand f^\prime = \left(f \wand (\overline{f^\prime} \sqcup \overline{g})\right)f^\prime \perp \left(f \wand (\overline{f^\prime} \sqcup \overline{g})\right)g = f \wand g$. So we do indeed get that $(f \wand f^\prime) \perp (f \wand g)$. Then we can apply {\bf [$\boldsymbol{\wand}$.4.a]} to get the rest of {\bf [$\boldsymbol{\wand}$.4]}. 
\end{proof}

Intuitively, the Kleene wand $f \wand g$ ``iterates'' the endomorphism $f$ until the ``guard'' $g$ is satisfied and then produces that output. This intuition is indeed justified by successive applications of the iteration axiom. 

\begin{lemma}\label{lem:iteration} Let $\mathbb{X}$ be a $\perp$-restriction category, and $f: X \to X$ and ${g: X \to A}$ maps such that $f \perp g$. Then for all $n,m \in \mathbb{N}$ and $n\neq m$, $f^ng \perp f^m g$ (where by convention, $f^0=1_X$). Furthermore, if $(\mathbb{X}, \wand)$ is a $\perp$-itegory, then for all $n \in \mathbb{N}$, the following equality holds:
\begin{align}\label{eq:iter-join-wand}
    f \wand g = g \sqcup fg \sqcup f^2g \sqcup \hdots \sqcup f^{n} g \sqcup f^{n+1}(f \wand g)
\end{align}
So in particular, $f^n g \leq f \wand g$ for all $n \in \mathbb{N}$. 
\end{lemma}
\begin{proof} Suppose that $f \perp g$. Then by \textbf{[{$\boldsymbol{\perp}$}.5]}, $\overline{f} \perp \overline{g}$. However, since $\overline{f^{n+1}g} \leq \overline{f}$ (using \textbf{[R.1]} and \textbf{[R.3]}), by \textbf{[{$\boldsymbol{\perp}$}.3]} we then have that $\overline{f^{n+1}g} \perp \overline{g}$. So by Lemma \ref{lem:inter-rest}, $f^{n+1} g \perp g$. Then by \textbf{[{$\boldsymbol{\perp}$}.4]}, for all $k \in \mathbb{N}$, pre-composing by $f^k$ gives us that $f^{n+k+1} g \perp f^k g$. From here we can conclude that $f^ng \perp f^m g$ for  $n,m \in \mathbb{N}$ and $n\neq m$, as desired. 

Now suppose $(\mathbb{X}, \wand)$ is a $\perp$-itegory. We prove (\ref{eq:iter-join-wand}) by induction on $n$. The base case $n=0$ is precisely {\bf [$\boldsymbol{\wand}$.1]}, since $g \sqcup f(f \wand g) = f \wand g$. So suppose that we have shown that (\ref{eq:iter-join-wand}) for all $0 \leq k \leq n$. We now show it for the case $n+1$ by showing that $f \wand g$ is the $\perp$-join of the family $\lbrace g, fg, \hdots, f^{n} g, f^{n+1}g, f^{n+2}(f \wand g) \rbrace$. By the induction hypothesis and {\bf [$\boldsymbol{\sqcup}$.1]}, we already have that $f^k g \leq f \wand g$ for all $0 \leq k \leq n$ and also that $f^{n+1}(f \wand g) \leq f \wand g$. So we also need to explain why $f\wand g$ is greater than $f^{n+1}g$ and $f^{n+2}(f \wand g)$. However using {\bf [$\boldsymbol{\wand}$.1]} and {\bf [$\boldsymbol{\sqcup}$.3]}, observe that $f^{n+1}(f \wand g) = f^{n+1}g \sqcup f^{n+2}(f \wand g)$. So since $f^{n+1}(f \wand g) \leq f \wand g$, by {\bf [$\boldsymbol{\sqcup}$.1]} and transitivity, we get that $f^{n+1}g \leq f \wand g$ and $f^{n+2}(f \wand g) \leq f \wand g$. So {\bf [$\boldsymbol{\sqcup}$.1]} holds for $n+1$. Now suppose that we have a map $h$ such that $f^k g \leq h$ for all $0 \leq k \leq n+1$ and $f^{n+2}(f \wand g) \leq h$. By {\bf [$\boldsymbol{\sqcup}$.2]}, we also get that $f^{n+1}g \sqcup f^{n+2}(f \wand g) \leq h$, or in other words, $f^{n+1}(f \wand g) \leq h$. So by the induction hypothesis and {\bf [$\boldsymbol{\sqcup}$.2]}, we get that $f \wand g \leq h$. So {\bf [$\boldsymbol{\sqcup}$.2]} holds for $n+2$. So we conclude that $f \wand g$ is the $\perp$-join of $\lbrace g, fg, \hdots, f^{n} g, f^{n+1}g, f^{n+2}(f \wand g) \rbrace$. Thus (\ref{eq:iter-join-wand}) holds for all $n \in \mathbb{N}$ as desired. 
\end{proof}

The above lemma suggests that a natural example of a Kleene wand should be the disjoint join of all $f^n g$ when it exists. This formula clearly realizes the intuition that $f$ is iterated until it falls in the domain of $g$. In fact, in Thm \ref{thm:countable-wand}, we will show that for a disjoint interference restriction category, this formula always defines a Kleene wand. In the meantime, here are our main examples of Kleene wands: 

\begin{example}\label{ex:PAR-wand} \normalfont $\mathsf{PAR}$ has a Kleene wand $\wand$ with respect to $\perp_0$ defined as follows for $\perp_0$-disjoint partial functions ${f: X \to X}$ and $g: X \to A$: 
\begin{align*}
    (f \wand g)(x) := \begin{cases} g(x) & \text{if $g(x) \downarrow$} \\
    g(f^{n+1}(x)) & \text{if there exists a (necessarily unique) $n \in \mathbb{N}$ such that $g(f^{n+1}(x)) \downarrow$} \\
    \uparrow & \text{otherwise} 
    \end{cases}
\end{align*}
    Note that this is well-defined since for all $m,n \in \mathbb{N}$ with $m\neq n$, if $g(f^m(x)) \downarrow$ then $g(f^n(x)) \uparrow$. Thus $(\mathsf{PAR}, \wand)$ is a $\perp_0$-itegory. 
\end{example}

\begin{example}\label{ex:REC-wand} \normalfont $(\mathsf{REC}, \wand)$ is a $\perp_r$-itegory where the Kleene wand $\wand$ is defined as in the previous example. 
\end{example}

\begin{example}\normalfont The category of timed sets as introduced in \cite{cockett2012timed} is a $\perp_0$-itegory. This was the original example for the original definition of a Kleene wand (which recall was only for $\perp_0$). We will not review timed sets here and invite the curious reader to see \cite{cockett2012timed} for more details. 
\end{example}

Here are some other useful basic properties about Kleene wands. 

\begin{lemma}\label{lem:wand} In a $\perp$-itegory $(\mathbb{X}, \wand)$, for maps $f: X \to X$, $g: X \to A$, and $g^\prime: X \to A$, which are pairwise $\perp$-disjoint, we have that: 
\begin{enumerate}[{\em (i)}]
\item \label{lem:wand.rest} $f \wand g = (f \wand \overline{g})g$
\item \label{lem:wand.0g} $0 \wand g = g$
\item \label{lem:wand.f0} $f \wand 0 = 0$
\item \label{lem:wand.perp} If $f \perp g \sqcup g^\prime$, then $f \wand g \perp f \wand g^\prime$ and $f \wand (g \sqcup g^\prime) = f \wand g \sqcup f \wand g^\prime$
\item If $f$ is total, then $f \wand g =0$. 
\item If $g$ is total, then $f \wand g= g$.
\end{enumerate}
\end{lemma}
\begin{proof} These are mostly straightforward to prove. 
    \begin{enumerate}[{\em (i)}]
\item This follows from \textbf{[R.1]} and \textbf{[{$\boldsymbol{\wand}$}.2]}, since $f \wand g = f \wand (\overline{g} g) = (f \wand \overline{g})g$. 

\item By Lemma \ref{lem:interference1}.(\ref{lem:interference1.0}), $0 \perp g$. So using \textbf{[{$\boldsymbol{\wand}$}.1]}, we have that $0 \wand g = g \sqcup 0(0 \wand g) = g \sqcup 0 = g$. 

\item Using (\ref{lem:wand.rest}), we get that $f \wand 0 = (f \wand \overline{0}) 0 = 0$. 

\item Suppose that $f \perp g \sqcup g^\prime$. By Lemma \ref{lem:interference1}.(\ref{lem:interference1.iv}) and Lemma \ref{lemma:disjoint-join}.(\ref{lemma:join-rest}), we also get that $f \perp \overline{g} \sqcup \overline{g^\prime}$. As such, we can take their Kleene wand: $f \wand (\overline{g} \sqcup \overline{g^\prime})$. Next note that by Lemma \ref{lem:interference1}.(\ref{lem:interference1.i}) and Lemma \ref{lemma:disjoint-join}.(\ref{lemma:join-comp}), we get that $(\overline{g} \sqcup \overline{g^\prime})g= g$ and $(\overline{g} \sqcup \overline{g^\prime})g^\prime= g^\prime$. Therefore, by \textbf{[{$\boldsymbol{\wand}$}.2]}, we get that $\left( f \wand (\overline{g} \sqcup \overline{g^\prime}) \right) g = f \wand g$ and $\left( f \wand (\overline{g} \sqcup \overline{g^\prime}) \right) g^\prime = f \wand g^\prime$. Now since $g \perp g^\prime$, by \textbf{[{$\boldsymbol{\perp}$}.4]} we get that $f \wand g = \left( f \wand (\overline{g} \sqcup \overline{g^\prime}) \right) g \perp \left( f \wand (\overline{g} \sqcup \overline{g^\prime}) \right) g^\prime = f \wand g^\prime$. So we have that $f \wand g \perp f \wand g^\prime$. Finally, using (\ref{lem:wand.rest}) and Lemma \ref{lemma:disjoint-join}.(\ref{lemma:join-comp}) and (\ref{lemma:join-rest}), we compute that: 
\begin{align*}
   f \wand (g \sqcup g^\prime) = \left( f \wand \overline{g \sqcup g^\prime} \right)(g \sqcup g^\prime) = \left( f \wand \left( \overline{g} \sqcup \overline{g^\prime} \right) \right)(g \sqcup g^\prime) =  \left( f \wand \left( \overline{g} \sqcup \overline{g^\prime} \right) \right)g \sqcup  \left( f \wand \left( \overline{g} \sqcup \overline{g^\prime} \right) \right)g^\prime = f \wand g \sqcup f \wand g^\prime 
\end{align*}
So $f \wand (g \sqcup g^\prime) = f \wand g \sqcup f \wand g^\prime$ as desired.  
\item If $f$ is total, then by Lemma \ref{lem:interference1}.(\ref{lem:interference1.total}), $f \perp g$ implies that $g=0$. So by (\ref{lem:wand.f0}), $f \wand g = f \wand 0 = 0$. 
\item If $g$ is total, then by Lemma \ref{lem:interference1}.(\ref{lem:interference1.total}), $f \perp g$ implies that $f=0$. So by (\ref{lem:wand.0g}), $f \wand g = 0 \wand g = g$.
\end{enumerate}
\end{proof}

We now turn our attention to discussing uniformity and inductivity of Kleene wands. These are analogues of their namesakes for iteration/trace operators, which are often desirable properties of well-behaved iteration/trace operators \cite{hasegawa2004uniformity}. 

\begin{definition}\label{def:uni-wand} In a $\perp$-itegory $(\mathbb{X}, \wand)$, the Kleene wand $\wand$ is said to be: 
\begin{enumerate}[{\em (i)}]
\item \textbf{Uniform} if when $f \perp g$ and $f^\prime \perp g^\prime$, and $hg^\prime = g$ and $hf^\prime= fh$, then $h \left( f^\prime \wand g^\prime \right)=f \wand g$;
\item \textbf{Lax uniform} if when $f \perp g$ and $f^\prime \perp g^\prime$, and $g \leq hg^\prime$ and  $fh\leq hf^\prime$, then $f \wand g \leq h \left( f^\prime \wand g^\prime \right)$;
\item \textbf{Colax uniform} if when $f \perp g$ and $f^\prime \perp g^\prime$, and $g \geq hg^\prime$ and $fh \geq hf^\prime$, then $f \wand g \geq h \left( f^\prime \wand g^\prime \right)$. 
\end{enumerate}
\end{definition}

Basic uniformity simply asks that the iteration behaves equivariently for homomorphic iteration data. The (co)lax versions of uniformity ask that (co)lax homomorphism between iteration data produces a commensurate (co)lax behaviour between iterants. Here are some equivalent alternative descriptions of uniformity. 

\begin{lemma}\label{lem:alt-uni} In a $\perp$-itegory $(\mathbb{X}, \wand)$,
\begin{enumerate}[{\em (i)}]
\item \label{lem:alt-uni.i} $\wand$ is uniform if and only if when $f \perp g$ and $f^\prime \perp g^\prime$, and $a g^\prime = g b$ and $af^\prime= fa$, then $a \left( f^\prime \wand g^\prime \right)=\left( f \wand g \right) b$;
\item \label{lem:alt-uni.ii}  $\wand$ is lax uniform if and only if when $f \perp g$ and $f^\prime \perp g^\prime$, and $g b \leq a g^\prime$ and $fa \leq af^\prime$, then $\left( f \wand g \right) b \leq a \left( f^\prime \wand g^\prime \right)$;
\item \label{lem:alt-uni.iii}  $\wand$ is colax uniform if and only if when $f \perp g$ and $f^\prime \perp g^\prime$, and $g b \geq a g^\prime$ and $fa \geq af^\prime$, then $\left( f \wand g \right) b \geq a \left( f^\prime \wand g^\prime \right)$;
\end{enumerate}
\end{lemma}
\begin{proof} For the $\Rightarrow$ direction of (\ref{lem:alt-uni.i}), suppose that $\wand$ is uniform. Now also suppose that $f \perp g$ and $f^\prime \perp g^\prime$, and also that $a g^\prime = g b$ and $af^\prime= fa$. Now by \textbf{[{$\boldsymbol{\perp}$}.4]}, we have that $f \perp gb$. As such, setting $h=a$, we can apply uniformity to get that $a(f^\prime \wand g^\prime) = f \wand gb$. However by \textbf{[{$\boldsymbol{\perp}$}.2]}, we have that $f \wand gb = (f \wand g)b$. So we get that $a(f^\prime \wand g^\prime) = f \wand gb$. Conversely, for the $\Leftarrow$ direction of (\ref{lem:alt-uni.i}), suppose that $f \perp g$ and $f^\prime \perp g^\prime$, and also that $hg^\prime = g$ and $hf^\prime= fh$. Setting $a=h$ and $b=1$, we get that $h \left( f^\prime \wand g^\prime \right)=f \wand g$, and so conclude that $\wand$ is uniform. Thus (\ref{lem:alt-uni.i}) holds as desired. Proving (\ref{lem:alt-uni.ii}) and (\ref{lem:alt-uni.iii}) is done in a similar fashion. 
\end{proof}

From the alternative description of (co)lax uniformity, we can show that a (co)lax uniform Kleene wand is monotone. 

\begin{corollary}\label{cor:wand.order} In a $\perp$-itegory $(\mathbb{X}, \wand)$, if $\wand$ is (co)lax uniform then if $f \perp g$ and $f^\prime \perp g^\prime$, and $g \leq g^\prime$ and $f \leq f^\prime$, then $f \wand g \leq f^\prime \wand g^\prime$.
\end{corollary}
\begin{proof} This follows immediately by setting $a=1$ and $b=1$ in Lemma \ref{lem:alt-uni}.(\ref{lem:alt-uni.ii}) (resp. (\ref{lem:alt-uni.iii})). 
\end{proof}

Note that being uniform and lax (resp. colax) uniform implies being colax (resp. lax) uniform, and also that being both lax and colax uniform implies being uniform. 

\begin{lemma} In a $\perp$-itegory $(\mathbb{X}, \wand)$, the following are equivalent: 
\begin{enumerate}[{\em (i)}]
\item $\wand$ is uniform and lax uniform;
\item $\wand$ is uniform and colax uniform;
\item $\wand$ is lax uniform and colax uniform. 
\end{enumerate}
\end{lemma}
\begin{proof} We show that $(i) \Rightarrow (ii) \Rightarrow (iii) \Rightarrow (i)$. Starting with $(i) \Rightarrow (ii)$, suppose that $\wand$ is uniform and lax uniform. We need to show that $\wand$ is also colax uniform. So suppose that $f \perp g$ and $f^\prime \perp g^\prime$, and also that $hg^\prime \leq g$ and $hf^\prime \leq  fh$. Then this says that $\overline{hg^\prime} g = hg^\prime$ and $\overline{h f^\prime}fh = hf^\prime$. Now note that $\overline{h f^\prime}f \leq f$ and $\overline{hg^\prime} g \leq g$, so by \textbf{[{$\boldsymbol{\perp}$}.2]}, $\overline{h f^\prime}f \perp \overline{hg^\prime} g$. Then by uniformity, we have that $h(f^\prime \wand g^\prime) = \overline{h f^\prime}f \wand \overline{hg^\prime} g$, while by Cor \ref{cor:wand.order}, we also have that $\overline{h f^\prime}f \wand \overline{hg^\prime} g \leq f \wand g$. So by transitivity we have that $h(f^\prime \wand g^\prime) \leq f \wand g$, and so we conclude that $\wand$ is colax. 

By a similar argument, we can also that $(ii) \Rightarrow (iii)$, since we would need to show that being uniform and colax uniform implies being lax uniform. So it remains to show that $(iii) \Rightarrow (i)$. So suppose that $\wand$ is both lax uniform and colax uniform. Now also suppose that $f \perp g$ and $f^\prime \perp g^\prime$, and also that $hg^\prime = g$ and $hf^\prime= fh$. Of course, this says that $hg^\prime \leq g$ and $hf^\prime \leq fh$, as well as $hg^\prime \geq g$ and $hf^\prime \geq fh$. So by applying lax uniformity we get that $f \wand g \leq h \left( f^\prime \wand g^\prime \right)$, while applying colax uniformity gives us that $h \left( f^\prime \wand g^\prime \right) \leq f \wand g$. So by antisymmetry, we get that $h \left( f^\prime \wand g^\prime \right)=f \wand g$, and conclude that $\wand$ is uniform. Thus $\wand$ is strongly uniform as desired. 
\end{proof}

\begin{definition} In a $\perp$-itegory $(\mathbb{X}, \wand)$, the Kleene wand $\wand$ is said to be \textbf{strongly uniform} if it is both uniform and (co)lax uniform (or equivalently lax uniform and colax uniform). 
\end{definition}

We can now define what it means for a Kleene wand to be (strongly) inductive. 

\begin{definition}\label{def:ind-wand} In a $\perp$-itegory $(\mathbb{X}, \wand)$, the Kleene wand $\wand$ is said to be: 
\begin{enumerate}[{\em (i)}]
\item \textbf{Inductive} if when $f \perp g$, $f h \leq h$, and $g \leq h$ then $f \wand g \leq h$
\item \textbf{Strongly inductive} if it is both strongly uniform and inductive. 
\end{enumerate}
A \textbf{(strongly) inductive $\perp$-itegory} is an itegory whose Kleene wand is (strongly) inductive. 
\end{definition}

We first note that an inductive Kleene wand, if one exists, is the minimal one amongst Kleene wands and, moreover, there is at most one inductive Kleene wand. 

\begin{lemma}\label{lemma:induc-less} Let $\wand$ be an inductive Kleene wand for a $\perp$-restriction category $\mathbb{X}$ with all binary $\perp$-joins, then:
\begin{enumerate}[{\em (i)}]
\item \label{lemma:induc-less.i} For any other Kleene wand $\wand^\prime$, we have for all $f \perp g$ that $f \wand g \leq f \wand^\prime g$.
\item \label{lemma:induc-less.ii} $\wand$ is the unique inductive Kleene wand on $\mathbb{X}$. 
\end{enumerate}
\end{lemma}
\begin{proof} For (\ref{lemma:induc-less.i}), suppose that $\wand$ is an inductive Kleene wand and $\wand^\prime$ is another Kleene wand. By \textbf{[$\boldsymbol{\wand}$.1]} for $\wand^\prime$, we have that $f(f \wand^\prime g) \leq f \wand^\prime g$ and $g \leq f \wand^\prime g$. So setting $h= f \wand^\prime g$, since $\wand$ is inductive, we get that $f \wand g \leq f \wand^\prime g$ as desired. Now for (\ref{lemma:induc-less.ii}), suppose that $\wand$ and $\wand^\prime$ are both inductive Kleene wands. By (\ref{lemma:induc-less.i}) this means that for all $f \perp g$, $f \wand g \leq f \wand^\prime g$ but also $f \wand^\prime g \leq f \wand g$. So by antisymmetry of $\leq$, we then get that $f \wand g = f \wand^\prime g$. So we conclude that an inductive Kleene wand, if one exists, is unique. 
\end{proof}

We now show that for a strongly inductive Kleene wand, there is an equivalent axiomatization which is much simpler.  

\begin{proposition}\label{prop:alt-wand} For a $\perp$-restriction category $\mathbb{X}$ with all binary $\perp$-joins, $\wand$ is a strongly inductive Kleene wand if and only if $\wand$ satisfies the following axioms: 
\begin{enumerate}[{\bf [{Alt.$\boldsymbol{\wand}$}.1]}]
\item If $f \perp g$, then $g \sqcup f(f \wand g) \leq f \wand g$; 
\item If $f \perp g$, $f^\prime \perp g^\prime$, and $a g^\prime = g b$ and $af^\prime= fa$, then $a \left( f^\prime \wand g^\prime \right)=\left( f \wand g \right) b$; 
\item If $f \perp g$, $f h \leq h$, and $g \leq h$ then $f \wand g \leq h$. 
\end{enumerate}
\end{proposition}
\begin{proof} The $\Rightarrow$ direction is automatic since: {\bf [{Alt.$\boldsymbol{\wand}$}.1]} is immediate from {\bf [$\boldsymbol{\wand}$.1]}, by Lemma \ref{lem:alt-uni}.(\ref{lem:alt-uni.i}) we get that {\bf [{Alt.$\boldsymbol{\wand}$}.2]} is an equivalent way of saying $\wand$ is uniform, and {\bf [{Alt.$\boldsymbol{\wand}$}.3]} is precisely the statement that $\wand$ is inductive. 

For the $\Leftarrow$ direction, we need to show that $\wand$ satisfies {\bf [$\boldsymbol{\wand}$.1]}, {\bf [$\boldsymbol{\wand}$.2]}, {\bf [$\boldsymbol{\wand}$.3]}, and {\bf [$\boldsymbol{\wand}$.4]}: \\ 

\noindent {\bf [$\boldsymbol{\wand}$.1]}: It suffices to show that $(f \wand g) \leq g \sqcup f(f \wand g)$. By {\bf [{Alt.$\boldsymbol{\wand}$}.1]} and since composition is monotone, we first get that $f\left(g \sqcup f(f \wand g) \right) \leq f(f\wand g) \leq g \sqcup f(f \wand g)$. So we have $f\left(g \sqcup f(f \wand g) \right) \leq g \sqcup f(f \wand g)$ and $g \leq g \sqcup f(f \wand g)$. Thus applying {\bf [{Alt.$\boldsymbol{\wand}$}.3]}, we get that $f \wand g \leq g \sqcup f(f \wand g)$. So by antisymmetry, we obtain that $f \wand g =  g \sqcup f(f \wand g)$. \\

\noindent {\bf [$\boldsymbol{\wand}$.2]}: If $f \perp g$, then starting with $f \perp g$ and $f \perp gh$, setting $a=1$ and $b=h$, by applying {\bf [{Alt.$\boldsymbol{\wand}$}.3]} we get $f \wand gh = (f \wand g)h$ as desired. \\

\noindent {\bf [$\boldsymbol{\wand}$.3]}: If $f \perp g$, then starting with $hf \perp hg$ and $fh \perp g$, setting $a=h$ and $b=1$, by applying {\bf [{Alt.$\boldsymbol{\wand}$}.3]} we get $h\left( fh \wand g \right)=hf \wand hg$ as desired.  \\

\noindent {\bf [$\boldsymbol{\wand}$.4]}: Suppose that $f \perp g$ and $f \sqcup g \perp h$, which by Lemma \ref{lemma:disjoint-join}.(\ref{lemma:join-inter}) also implies that $f \perp h$ and $g \perp h$. We first show that $(f \sqcup g) \wand h \leq (f \wand g) \wand (f \wand h)$. We first observe that using {\bf [$\boldsymbol{\wand}$.1]} we can compute that: 
\begin{align*}
    (f \wand g) \wand (f \wand h) &=~ (f \wand h) \sqcup (f \wand g)\left(  (f \wand g) \wand (f \wand h) \right) \tag{{\bf [$\boldsymbol{\wand}$.1]}} \\
    &=~ h \sqcup f(f\wand h) \sqcup \left( g \sqcup f(f \wand g) \right) \left(  (f \wand g) \wand (f \wand h) \right) \tag{{\bf [$\boldsymbol{\wand}$.1]}} \\
    &=~ h \sqcup f(f\wand h) \sqcup  g  \left(  (f \wand g) \wand (f \wand h) \right) \sqcup f(f \wand g) \left(  (f \wand g) \wand (f \wand h) \right) \tag{Lemma \ref{lemma:disjoint-join}.(\ref{lemma:join-comp})} \\
    &=~ h \sqcup f \left( (f\wand h) \sqcup (f \wand g) \left(  (f \wand g) \wand (f \wand h) \right) \right) \sqcup  g  \left(  (f \wand g) \wand (f \wand h) \right) \tag{Lemma \ref{lemma:disjoint-join}.(\ref{lemma:join-comp})}  \\
    &=~ h \sqcup f\left(  (f \wand g) \wand (f \wand h) \right) \sqcup g  \left(  (f \wand g) \wand (f \wand h) \right)  \tag{{\bf [$\boldsymbol{\wand}$.1]}} \\
    &=~ h \sqcup (f \sqcup g)\left(  (f \wand g) \wand (f \wand h) \right) \tag{Lemma \ref{lemma:disjoint-join}.(\ref{lemma:join-comp})}  
\end{align*}
So we have that $(f \wand g) \wand (f \wand h) = h \sqcup (f \sqcup g)\left(  (f \wand g) \wand (f \wand h) \right)$. This implies both that ${h \leq (f \wand g) \wand (f \wand h)}$ and $(f \sqcup g)\left(  (f \wand g) \wand (f \wand h) \right) \leq (f \wand g) \wand (f \wand h)$. Then applying \textbf{[Alt.$\boldsymbol{\wand}$.3]}, we get that $(f \sqcup g) \wand h \leq (f \wand g) \wand (f \wand h)$. 

Next we show that $(f \wand g) \wand (f \wand h) \leq (f \sqcup g) \wand h$ as well. So first observe that using  {\bf [$\boldsymbol{\wand}$.1]} again, we can compute that $(f \sqcup g) \wand h = h \sqcup f \left(  (f \sqcup g) \wand h \right) \sqcup g \left(  (f \sqcup g) \wand h \right)$. This implies both that $h \leq  (f \sqcup g) \wand h$ and $f \left(  (f \sqcup g) \wand h \right) \leq  (f \sqcup g) \wand h$. Thus applying \textbf{[Alt.$\boldsymbol{\wand}$.3]}, we get that $f \wand h \leq (f \sqcup g) \wand h$. On the other hand, we also have that $g \left(  (f \sqcup g) \wand h \right) \leq  (f \sqcup g) \wand h$ and $f \left(  (f \sqcup g) \wand h \right) \leq  (f \sqcup g) \wand h$. So applying \textbf{[Alt.$\boldsymbol{\wand}$.3]} here gives us that $f \wand g \left(  (f \sqcup g) \wand h \right) \leq (f \sqcup g) \wand h$. However, by  {\bf [$\boldsymbol{\wand}$.2]}, we have that $f \wand g \left(  (f \sqcup g) \wand h \right) = (f \wand g) \left(  (f \sqcup g) \wand h \right)$. So we then get that $(f \wand g) \left(  (f \sqcup g) \wand h \right) \leq  (f \sqcup g) \wand h$. Therefore, bringing this all together, we have that $(f \wand g) \left(  (f \sqcup g) \wand h \right) \leq  (f \sqcup g) \wand h $ and $f \wand h \leq (f \sqcup g) \wand h$. Then applying \textbf{[Alt.$\boldsymbol{\wand}$.3]} again, we get $(f \wand g) \wand (f \wand h) \leq (f \sqcup g) \wand h$. So by antisymmetry, we conclude that $(f \sqcup g) \wand h = (f \wand g) \wand (f \wand h)$ as desired. \\

So $\wand$ is a Kleene wand. We must also show that $\wand$ is inductive and strongly uniform. However, \textbf{[Alt.$\boldsymbol{\wand}$.3]} is precisely the statement that $\wand$ is inductive, while Lemma \ref{lem:alt-uni}.(\ref{lem:alt-uni.i}) tells us that \textbf{[Alt.$\boldsymbol{\wand}$.3]} is an equivalent way of saying $\wand$ is uniform. So it remains to show that $\wand$ is lax uniform as well. So suppose that $f \perp g$ and $f^\prime \perp g^\prime$, and also that $g \leq hg^\prime$ and $fh\leq hf^\prime$. By \textbf{[Alt.$\boldsymbol{\wand}$.1]}, recall that we have that $g^\prime \leq f^\prime \wand g^\prime$ and $f^\prime\left( f^\prime \wand g^\prime \right) \leq f^\prime \wand g^\prime$. Pre-composing by $h$, we then get $hg^\prime \leq h\left( f^\prime \wand g^\prime \right)$ and $hf^\prime\left( f^\prime \wand g^\prime \right) \leq h\left( f^\prime \wand g^\prime \right)$. By transitivity, we have that $g \leq h\left( f^\prime \wand g^\prime \right)$ and $fh \left( f^\prime \wand g^\prime \right) \leq h\left( f^\prime \wand g^\prime \right)$. Therefore, applying \textbf{[Alt.$\boldsymbol{\wand}$.3]}, we get that $f \wand g \leq h \left( f^\prime \wand g^\prime \right)$. So we have that $\wand$ is also lax uniform. Therefore, we conclude that $\wand$ is a strongly inductive Kleene wand. 
\end{proof}

By the above proposition, we can now easily show that in a setting with all disjoint joins, there is a canonical strongly inductive Kleene wand. So every (countable) disjoint interference restriction category is canonically a strongly inductive itegory. 

\begin{theorem} \label{thm:countable-wand} Let $\mathbb{X}$ be a disjoint $\perp$-restriction category. Then $\mathbb{X}$ has a strongly inductive Kleene wand $\wand$ defined as follows for $\perp$-disjoint maps $f: X \to X$ and $g: X \to A$:
\begin{align}\label{countable-wand}
    f \wand g \colon = \bigsqcup \limits_{n \in \mathbb{N}} f^n g
\end{align}
where by convention $f^0 = 1_X$. Thus , $(\mathbb{X}, \wand)$ is a strongly inductive $\perp$-itegory. 
\end{theorem}
\begin{proof} First note that $\wand$ is well-defined by Lemma \ref{lem:iteration}. By Prop \ref{prop:alt-wand}, to show that $\wand$ is a strongly inductive Kleene wand, it suffices to show that $\wand$ satisfies the three axioms \textbf{[Alt.$\boldsymbol{\wand}$.1]} to \textbf{[Alt.$\boldsymbol{\wand}$.3]}. \\

\noindent {\bf [{Alt.$\boldsymbol{\wand}$}.1]}: Using \textbf{[$\boldsymbol{\sqcup}$.4]}, we easily compute that: 
\[ f \wand g = \bigsqcup \limits_{n \in \mathbb{N}} f^n g = g \sqcup \bigsqcup \limits_{n \in \mathbb{N}} f^{n+1} g =  g \sqcup f \left( \bigsqcup \limits_{n \in \mathbb{N}} f^n g \right) = g \sqcup f(f\wand g) \]
So we have that $f \wand g = g \sqcup f(f\wand g)$, which of course means that $g \sqcup  f(f\wand g) \leq f\wand g$. \\

\noindent {\bf [{Alt.$\boldsymbol{\wand}$}.2]}: Suppose that $a g^\prime = g b$ and $af^\prime= fa$. Then by Lemma \ref{lemma:disjoint-join}.(\ref{lemma:join-rest}), we easily compute that: 
\begin{align*}
 a \left( \bigsqcup \limits_{n \in \mathbb{N}} {f^\prime}^n g^\prime \right) = \bigsqcup \limits_{n \in \mathbb{N}} a {f^\prime}^n g^\prime = \bigsqcup \limits_{n \in \mathbb{N}} f^n a g^\prime = \bigsqcup \limits_{n \in \mathbb{N}} f^n g b = \left( \bigsqcup \limits_{n \in \mathbb{N}} f^n g \right) b  
\end{align*}
So $a \left( f^\prime \wand g^\prime \right) = (f \wand g)b$. \\

\noindent {\bf [{Alt.$\boldsymbol{\wand}$}.3]}: Suppose that $f h \leq h$ and $g \leq h$. Note that $fh \leq h$ implies that $f^n h \leq h$  for all $n \in \mathbb{N}$ as well. Then we see that $f^n g \leq f^n h \leq h$. So $f^n g \leq h$ for all $n \in \mathbb{N}$. So by \textbf{[$\boldsymbol{\sqcup}$.2]}, $\bigsqcup \limits_{n \in \mathbb{N}} f^n g \leq h$, which we rewrite as $f \wand g \leq h$. \\

So we conclude that $\wand$ is a strongly inductive Kleene wand as desired. 
\end{proof}

\begin{example} \normalfont The Kleene wands given in Ex.\ref{ex:PAR-wand} and \ref{ex:REC-wand} are strongly inductive Kleene wands, and so $(\mathsf{PAR}, \wand)$ and $(\mathsf{REC}, \wand)$ are both strongly inductive itegories. 
\end{example}

Moreover, using the construction from Prop \ref{prop:construction1}, we observe that every interference restriction category embeds into a strongly inductive itegory. 

\begin{corollary} Let $\mathbb{X}$ be a $\perp$-restriction category. Then $\left( \mathsf{DJ}\left[ (\mathbb{X},\perp) \right], \wand \right)$ is a strongly inductive $\perp$-itegory, where $\wand$ is defined as in Thm \ref{thm:countable-wand}. 
\end{corollary}

In particular, we can in fact embed any restriction with restriction zeroes into a strongly inductive itegory by considering the maximal interference restriction $\perp_0$. It is not difficult to see that the induced interference restriction on $\mathsf{DJ}\left[ (\mathbb{X},\perp_0) \right]$ will also be the maximal one. 

\begin{corollary} Let $\mathbb{X}$ be a restriction category with restriciton zeroes. Then $\left( \mathsf{DJ}\left[ (\mathbb{X},\perp_0) \right], \wand \right)$ is a strongly inductive $\perp_0$-itegory, where $\wand$ is defined as in Thm \ref{thm:countable-wand}. 
\end{corollary}

We conclude this section by observing that the minimal interference relation has a unique but trivial Kleene wand, which is also strongly inductive. 

\begin{lemma} Let $\mathbb{X}$ be a restriction category with restriction zeroes. Then $(\mathbb{X}, \wand_\delta)$ is a strongly inductive $\perp_\delta$-itegory where $\wand_\delta$ is defined as follows: 
\begin{align}
f \wand_\delta g = g
\end{align}
Moreover $\wand_\delta$ is the unique Kleene wand for $(\mathbb{X}, \perp_\delta)$. 
\end{lemma}
\begin{proof} Since by Lemma \ref{lemma:minimal-strong-comp}, $\mathbb{X}$ is a disjoint $\perp_d$-restriction category, we show that $\wand_\delta$ is a strongly inductive Kleene wand by checking \textbf{[Alt.$\boldsymbol{\wand}$.1]} to \textbf{[Alt.$\boldsymbol{\wand}$.3]}. However note that \textbf{[Alt.$\boldsymbol{\wand}$.2]} and \textbf{[Alt.$\boldsymbol{\wand}$.3]} are immediate by definition of $\wand_\delta$. For \textbf{[Alt.$\boldsymbol{\wand}$.1]}, suppose that $f \perp_\delta g$ which implies that $f=0$ or $g=0$. We need to show that $g \sqcup fg \leq g$. So if $f=0$, then we have that $g \sqcup fg = g \sqcup 0g = g \sqcup 0 = g$. If $g=0$, then $g \sqcup fg = 0 \sqcup f0 = 0 \sqcup f0 =0 = g$. So in either case we have that $g \sqcup fg = g \leq g$, so \textbf{[Alt.$\boldsymbol{\wand}$.1]} holds. Thus $\wand_\delta$ is a strongly inductive Kleene wand. Now suppose we have another Kleene wand $\wand$ for $(\mathbb{X}, \perp_\delta)$, and suppose that $f \perp_\delta g$, so either $f=0$ or $g=0$. If $f=0$, by Lemma \ref{lem:wand}.(\ref{lem:wand.0g}), $f\wand g = 0 \wand g = g = f \wand_\delta g$. If $g=0$, by Lemma \ref{lem:wand}.(\ref{lem:wand.f0}), $f \wand g = f \wand 0 = 0 = g = f \wand_\delta g$. So we conclude that $\wand = \wand_\delta$, and therefore $\wand_\delta$ is the unique Kleene wand for $\perp_\delta$.  
\end{proof}

\section{Kleene Upper-Stars}\label{sec:kleene-star}

In this section, we take a detour to discuss Kleene wands in the \emph{classical} restriction category setting, that is, in the presence of \emph{relative complements}. In such a setting, we will show that Kleene wands are in fact completely determined by how they act on their endomorphism input. Thus, in the classical setting, Kleene wands are equivalent to an operator on endomorphisms, which we call a \textbf{Kleene upper-star}. 

We begin by discussing relative complements in a disjoint interference restriction category. Relative complements for join restriction categories were introduced in \cite[Sec 13]{cockett2009boolean}. Here, we slightly adjust the definition instead for disjoint $\perp$-restriction categories, where in particular, the relative complement is not simply disjoint in the sense of \cite[Prop 6.2]{cockett2009boolean} but $\perp$-disjoint. 

\begin{definition}\label{def:classical} In a $\perp$-restriction category $\mathbb{X}$ with all binary $\perp$-joins, for parallel maps ${f: A \to B}$ and $g: A \to B$ such that $f \leq g$, the $\perp$\textbf{-relative complement} of $f$ in $g$ (if it exists) is a (necessarily unique) map $g \backslash f: A \to B$ such that: 
\begin{enumerate}[{\bf [$\boldsymbol{\backslash}$.1]}]
\item $g \backslash f \perp f$
\item $(g \backslash f) \sqcup f =g$
\end{enumerate}
A \textbf{classical $\perp$-restriction category} is a $\perp$-restriction category $\mathbb{X}$ with all binary $\perp$-joins such that all $\perp$-relative complements exists. 
\end{definition}

Intuitively, $g \backslash f$ is undefined when $f$ is defined and is equal to $g$ whenever $f$ is undefined. Moreover, a $\perp$-relative complement is a relative complement in the sense of \cite[Sec 13]{cockett2009boolean} (which we can now call a $\perp_0$-relative complement), and so satisfies the same properties, some of which can be found in \cite[Lemma 13.14]{cockett2009boolean} and \cite[Lemma 4.9]{cockett2023classical}. 

Recall that $\mathcal{O}(A)$ is a bounded meet-semilattice, so every restriction idempotent $e: A \to A$ is always less than or equal to the identity $1_A: A \to A$, so $e \leq 1_A$. As such, we can consider the relative complement of a restriction idempotent in the identity, which we denote as $e^c := 1_A \backslash e$. It turns out that having complements of restriction idempotents is sufficient to have all relative complements. 

\begin{lemma} A $\perp$-restriction category $\mathbb{X}$ with all binary $\perp$-joins is classical if and only if for every restriction idempotent $e: A \to A$, there exists a (necessarily unique) restriction idempotent $e^c: A \to A$ such that $e \perp e^c$ and $e \sqcup e^c = 1_A$.  
\end{lemma}
\begin{proof} For $\Rightarrow$ direction, define $e^c := 1_A \backslash e$. Then $e \perp e^c$ and $e \sqcup e^c = 1_A$ are precisely {\bf [$\boldsymbol{\backslash}$.1]} and {\bf [$\boldsymbol{\backslash}$.2]} respectively. For the $\Leftarrow$ direction, define $g \backslash f := \overline{f}^c g$. Now since $\overline{f}^c \perp \overline{f}$, by Lemma \ref{lem:interference1}.(\ref{lem:interference1.iv}), we get that $\overline{f}^c \perp f$. Then by {\bf [$\boldsymbol{\perp}$.4]}, it follows that $g \backslash f = \overline{f}^c g \perp f$. Next, using that $f \leq g$ and Lemma \ref{lemma:disjoint-join}.(\ref{lemma:join-comp}), we compute that $(g \backslash f) \sqcup f  = \overline{f}^c g \sqcup f = \overline{f}^c g \sqcup \overline{f} g = \left( \overline{f}^c \sqcup \overline{f} \right) g = g$. So $ (g \backslash f) \sqcup f  = g$. Thus we conclude that $g \backslash f$ is indeed the $\perp$-relative complement of $f$ in $g$. 
\end{proof}

Here is now our main example, and also a non-example. 

\begin{example} \label{ex:PAR-classical} \normalfont $\mathsf{PAR}$ is a classical $\perp_0$-restriction category where for partial functions $f: X \to Y$ and ${g: X \to Y}$ such that $f \leq g$, the $\perp_0$-relative complement $g \backslash f: X \to Y$ is the partial function defined as follows: 
\[ (g \backslash f) (x) = \begin{cases} g(x) & \text{if } f(x) \uparrow \text{ and } g(x) \downarrow \\
\uparrow & \text{if } f(x) \downarrow \text{ or } g(x) \uparrow  \end{cases} \]
In particular, for a subset $U \subseteq X$, the complement of its associated restriction idempotent $e^c_U: X \to X$ is defined as the restriction idempotent associated to the complement of $U$, $e^c_U = e_{U^c}$. Explicitly:
 \begin{align*}
  e^c_U(x) \colon = \begin{cases} 
   x & \text{ if } x \notin U \\
   \uparrow & \text{ if } x \in U 
  \end{cases}
   \end{align*}
\end{example}

\begin{example} \normalfont $\mathsf{REC}$ is not classical since the complement of a recursively enumerable set is not necessarily recursively enumerable. As such, not every restriction idempotent in $\mathsf{REC}$ has a $\perp_r$-complement. 
\end{example}

We now consider Kleene wands in a classical interference category. 

\begin{definition} A \textbf{classical $\perp$-itegory} is a $\perp$-itegory $(\mathbb{X},\wand)$ where $\mathbb{X}$ is a classical $\perp$-restriction category.
\end{definition}

Observe that in a classical $\perp$-restriction category, $f \perp g$ implies that $\overline{g} \leq \overline{f}^c$, and so $\overline{f}^c g = g$. Then in a classical itegory, by \textbf{[$\boldsymbol{\wand}$.2]}, we get that $f \wand g = (f \wand \overline{f}^c)g$. As such, $\wand$ is completely characterized by its endomorphism component. Therefore for a classical interference restriction category, to give a Kleene wand amounts to giving an operation on endomorphisms. We call such an operation a Kleene upper-star. 

\begin{definition}\label{def:star} For a classical $\perp$-restriction category $\mathbb{X}$, a \textbf{Kleene upper-star} $\star$ is a family of operators (indexed by objects $X \in \mathbb{X}$):
\begin{align*}
\star: \mathbb{X}(X, X) \to  \mathbb{X}(X, X) && \begin{matrix}[c] \infer{f^\star : X \to X}{f: X \to X} \end{matrix}
\end{align*}
such that the following axioms hold: 
\begin{enumerate}[{\bf [{$\boldsymbol{\star}$}.1]}]
\item $\overline{f}^c \sqcup ff^\star = f^\star$
\item $(hf)^\star h = h (fh)^\star \overline{f}^c$
\item If $f \perp g$, then $(f \sqcup g)^\star = (f^\star g)^\star f^\star$
\end{enumerate}
\end{definition}

It is worth noting that the axioms of a Kleene upper-star have a similar flavour to the axioms of a \emph{repetition operator} \cite[Prop 6.11]{selinger2010survey}, which can be defined on categories enriched over commutative monoids. We now show that every Kleene wand induces a Kleene upper-star. 

\begin{proposition}\label{prop:wand-to-star} Let $(\mathbb{X}, \wand)$ be a classical $\perp$-itegory. Then $\mathbb{X}$ has a Kleene upper-star $\star$ defined as follows: 
\begin{align}
    f^\star := f \wand \overline{f}^c
\end{align}
\end{proposition}
\begin{proof} We show that $\star$ satisfies the three Kleene upper-star axioms.  
     \begin{enumerate}[{\bf [{$\boldsymbol{\star}$}.1]}]
\item Using {\bf [{$\boldsymbol{\wand}$}.1]} we get that $\overline{f}^c \sqcup ff^\star = \overline{f}^c \sqcup f \left( f \wand \overline{f}^c \right) = f \wand \overline{f}^c = f^\star$. 
\item Note that since $\overline{fh} \leq \overline{f}$, taking their complements swaps the order, so we have that $\overline{f}^c \leq \overline{fh}^c$, and so $\overline{f}^c \overline{fh}^c = \overline{f}^c$. Now the analogue of \textbf{[R.4]} for complements is that $f \overline{g}^c = \overline{fg}^c f$ \cite[Lemma 4.9.(ix)]{cockett2023classical}. Therefore, using this as well as {\bf [{$\boldsymbol{\wand}$}.2]} and {\bf [{$\boldsymbol{\wand}$}.3]}, we compute that: 
\begin{align*}
    (hf)^\star h = \left( hf \wand \overline{hf}^c \right) h = hf \wand \overline{hf}^c h = hf \wand h \overline{f}^c = h \left( fh \wand \overline{f}^c \right) =  h \left( fh \wand \overline{fh}^c \overline{f}^c \right) = h \left( fh \wand \overline{fh}^c  \right)\overline{f}^c = h (fh)^\star \overline{f}^c 
\end{align*}
\item Suppose that $f \perp g$. We have the de Morgan law for complements $\overline{f \sqcup g}^c = \overline{f}^c \overline{g}^c$ \cite[Lemma 4.9.(viii)]{cockett2023classical}. As such, we have that $f \sqcup g \perp \overline{f}^c \overline{g}^c$, and thus by {\bf [{$\boldsymbol{\wand}$}.4]}, we have that $f \wand g \perp f \wand \overline{f}^c\overline{g}^c$. Then using {\bf [{$\boldsymbol{\wand}$}.2]} and {\bf [{$\boldsymbol{\wand}$}.4]}, we compute that: 
\begin{align*}
  (f \sqcup g)^\star &=~ (f \sqcup g) \wand \overline{f \sqcup g}^c \tag{Def.} \\
  &=~  (f \sqcup g) \wand \overline{f}^c \overline{g}^c \tag{de Morgan Law} \\
  &=~(f \wand g) \wand (f \wand \overline{f}^c \overline{g}^c)  \tag{{\bf [$\boldsymbol{\wand}$.4]}} \\
  &=~ (f \wand \overline{f}^c g) \wand (f \wand \overline{f}^c \overline{g}^c) \\ 
   &=~(f \wand \overline{f}^c)g \wand (f \wand \overline{f}^c)\overline{g}^c  \tag{{\bf [$\boldsymbol{\wand}$.2]}} \\
   &=~ f^\star g \wand f^\star \overline{g}^c \tag{Def.} \\
&=~ f^\star g \wand \overline{f^\star g}^c f^\star  \tag{\textbf{[R.4]} for complements} \\
&=~ \left( f^\star g \wand \overline{f^\star g}^c \right) f^\star \tag{{\bf [$\boldsymbol{\wand}$.2]}} \\ 
  &=~  (f^\star g)^\star f^\star \tag{Def.}
\end{align*}
\end{enumerate}
So $\star$ is a Kleene upper-star. 
\end{proof}

Conversely, we now show that every Kleene upper-star induces a Kleene wand.

\begin{proposition}\label{prop:star-to-wand} Let $\mathbb{X}$ be a classical $\perp$-restriction category with a Kleene upper-star $\star$. Then $\mathbb{X}$ has a Kleene wand $\wand$ defined as follows: 
    \begin{align}
    f \wand g := f^\star g
\end{align}
\end{proposition}
\begin{proof} We show that $\wand$ satisfies the four Kleene wand axioms.  
  \begin{enumerate}[{\bf [$\boldsymbol{\wand}$.1]}]
\item First observe that if $f \perp g$, then $\overline{g} \leq \overline{f}^c$ which implies that $\overline{f}^c g = g$. So using Lemma \ref{lemma:disjoint-join}.(\ref{lemma:join-comp}) and {\bf [{$\boldsymbol{\star}$}.1]}, we get that: 
\begin{align*}
    g \sqcup f(f \wand g) = g \sqcup f f^\star g = \overline{f}^c g \sqcup f f^\star g = \left( \overline{f}^c \sqcup ff^\star \right) g = f^\star g  = f \wand g 
\end{align*}
\item This is automatic since $\left( f \wand g \right)h = f^\star g h = f \wand gh$. 
\item Again since $f \perp g$, we have that $\overline{f}^c g = g$. So using {\bf [{$\boldsymbol{\star}$}.2]} we compute that: 
\begin{align*}
    hf \wand hg = (hf)^\star hg =  h (fh)^\star   \overline{f}^c g = h (fh)^\star g =  h\left( fh \wand g \right)
\end{align*}
\item Suppose that $f \perp f^\prime$ and $f \sqcup f^\prime \perp g$. Now by Lemma \ref{lemma:disjoint-join}.(\ref{lemma:join-inter}), we have that $f^\prime \perp g$, and so by {\bf [{$\boldsymbol{\perp}$}.4]}, we have that $f^\star f^\prime \perp f^\star g$, or in other words $(f \wand f^\prime) \perp (f \wand g)$. Next, using {\bf [{$\boldsymbol{\star}$}.3]} we compute that: 
\begin{align*}
(f \sqcup f^\prime) \wand g = (f \sqcup f^\prime)^\star g = (f^\star f^\prime)^\star f^\star g = (f \wand f^\prime)^\star  (f \wand g) = (f \wand f^\prime) \wand (f \wand g)
\end{align*}
\end{enumerate}
So $\wand$ is a Kleene wand as desired. 
\end{proof}

Next we show that these constructions are inverses of each other. 

\begin{theorem}\label{thm:star=wand} For a classical $\perp$-restriction category $\mathbb{X}$, there is a bijective correspondence between Kleene wands and Kleene upper-stars. Therefore, a classical itegory is precisely a classical interference restriction category with a Kleene upper-star. 
\end{theorem}
\begin{proof} We need to show that the constructions of Prop \ref{prop:wand-to-star} and Prop \ref{prop:star-to-wand} are inverses of each other, that is, we need to show that $f \wand g = (f \wand \overline{f}^c)g$ and $f^\star = f^\star \overline{f}^c$. For the former, as explained above, since $f \perp g$, we have that $\overline{f}^c g = g$. Therefore using {\bf [$\boldsymbol{\wand}$.2]} we get that $f \wand g = f \wand \overline{f}^cg =  (f \wand \overline{f}^c)g$. So $f \wand g = (f \wand \overline{f}^c)g$. For the other identity, applying {\bf [{$\boldsymbol{\star}$}.2]} with $h=1_X$ gives us precisely that $f^\star = f^\star \overline{f}^c$. So the constructions from a Kleene wand to a Kleene upper-star and vice-versa are indeed inverses of each others. 
\end{proof}

In the setting where we have all disjoint joins, we can apply the above theorem to the canonical Kleene wand from Thm \ref{thm:countable-wand} to obtain a Kleene upper-star. This Kleene upper-star is similar to the repetition operator in a setting with infinite sums. 

\begin{corollary} A classical $\perp$-restriction category $\mathbb{X}$ which is a countable disjoint $\perp$-restriction category, has a Kleene upper-star $\star$ defined as follows on an endomorphism $f: X \to X$: 
\begin{align}
    f^\star \colon = \bigsqcup \limits_{n \in \mathbb{N}} f^n \overline{f}^c
\end{align}
\end{corollary}
\begin{proof} This is precisely the associated Kleene upper-star for the canonical Kleene wand defined in Thm \ref{thm:countable-wand}.    
\end{proof}

\begin{example} \normalfont $(\mathsf{PAR}, \wand)$ is a classical itegory whose induced Kleene upper-star is defined on a partial function ${f: X \to X}$ as:
\begin{align*}
    f^\star(x) := \begin{cases} x & \text{if $f(x) \uparrow$} \\
    f^{n+1}(x) & \text{if there exists a (necessarily unique) $n \in \mathbb{N}$ such that $f^{n+1}(x) \downarrow$ and $f^{n+2}(x) \uparrow$} \\
    \uparrow & \text{otherwise}. 
    \end{cases}
\end{align*} 
\end{example}

\section{Extensive Restriction Categories}\label{sec:extensive}

We now turn our attention to extensive restriction categories. The objective of this section is to show that every extensive restriction category has a canonical interference relation induced by decisions and, moreover, that extensive restriction categories have binary disjoint joins. This means we can certainly consider Kleene wands for extensive restriction categories. 

 The underlying structure of an extensive restriction category is that of a \emph{coCartesian restriction category}, that is, a restriction category with finite \emph{restriction coproducts}. Restriction coproducts are coproducts in the usual sense, with the added assumption that the coproduct injection maps are total. So for a category $\mathbb{X}$ with finite coproducts, we denote the initial object as $\mathsf{0}$, the coproduct as $+$, the injection maps as $\iota_j: A_j \to A_1 + \hdots + A_n$, and write the copairing operation as $\left[ \begin{smallmatrix} - \\ \vdots \\ - \end{smallmatrix}\right]$, that is, for a family of maps $f_1: A_1 \to B$, ..., $f_n: A_n \to B$, their copairing $\left[ \begin{smallmatrix} f_1 \\ \vdots \\ f_n \end{smallmatrix}\right]: A_1 + \hdots + A_n \to B$ is the unique map such that $\iota_j\left[ \begin{smallmatrix} f_1 \\ \vdots \\ f_n \end{smallmatrix}\right] = f_j$. Also, for a family of maps $g_1: A_1 \to B_1$, ..., $g_n: A_n \to B_n$, define their coproduct $g_1 + \hdots + g_n: A_1 + \hdots + A_n \to B_1 + \hdots + B_n$ as $g_1 + \hdots + g_n \colon = \left[ \begin{smallmatrix}  g_1 \iota_1 \\ \vdots \\  g_n \iota_n \end{smallmatrix}\right]$. This choice of writing the copairing operation as a vertical column is justified by the matrix calculus of an extensive restriction category, which we review in the next section. We also denote the canonical codiagonal maps as $\nabla := \left[ \begin{smallmatrix} 1_A \\ \vdots \\ 1_A \end{smallmatrix}\right]: A + \hdots + A \to A$. 

 \begin{definition}\label{def:restcoprod} A \textbf{coCartesian restriction category} is a restriction category $\mathbb{X}$ with finite \textbf{restriction coproducts} \cite[Sec 2.1]{cockett2007restriction}, that is, $\mathbb{X}$ has finite coproducts where all the injection maps ${\iota_j: A_j \to A_1 + \hdots + A_n}$ are total. 
\end{definition}

In a coCartesian restriction category with restriction zeroes, we can define ``quasi-projections'' for the restriction coproduct which make the injections into \emph{restriction isomorphisms} (also known as \emph{partial isomorphisms}). Recall that in a restriction category, a \textbf{restriction isomorphism} \cite[Sec 2.3.1]{cockett2002restriction} is a map $f: A \to B$ such that there is a map $f^\circ: B \to A$, called its \textbf{restriction inverse}, such that $f f^\circ = \overline{f}$ and $f^\circ f = \overline{f^\circ}$. Restriction inverses are unique, and if $f$ is a restriction isomorphism then its restriction inverse $f^\circ$ is also a restriction isomorphism, whose restriction inverse is $f$ \cite[Lem 2.18.(vii)]{cockett2002restriction}. 

\begin{definition} In a coCartesian restriction categories with restriction zeroes, for a finite family of objects $A_0, \hdots, A_n$, the \textbf{quasi-projections} are the maps $\iota^\circ_j: A_0 + \hdots + A_n \to A_j$ defined as the copairing of zeroes and the identity in the $j$th argument, $\iota^\circ_j = \left[ \begin{smallmatrix} 0 \\ \vdots \\ 0 \\ 1_{A_j} \\ 0 \\ \vdots\\ 0 \end{smallmatrix}\right]$. In other words, $\iota^\circ_j$ is the unique map such that $\iota_j \iota^\circ_j =1_{A_j}$ and $\iota_i \iota^\circ_j =0$ for $i \neq j$. 
\end{definition}

Other useful identities in coCartesian restriction categories with restriction zeroes can be found in \cite[Lemma 2.10]{cockett2007restriction} and \cite[Lemma 3.6]{cockett2023classical}. In particular, as mentioned, the injection maps are restriction isomorphisms and also the initial object $\mathsf{0}$ becomes a zero object, where to the unique map to and from $\mathsf{0}$ are precisely the restriction zero maps. 

As the name suggests, extensive restriction categories are the appropriate restriction category analogue of extensive categories \cite{carboni1993introduction}. More explicitly, an extensive restriction category is a restriction category with restriction coproducts and restriction zeroes, and such that every map whose codomain is a coproduct admits a \emph{decision}. 

\begin{definition}\label{def:ext-rest-cat} An \textbf{extensive restriction category} \cite[Sec 3]{cockett2007restriction} is a coCartesian restriction category with restriction zeroes such that for every map $f: A \to B_1 + \hdots + B_n$ there exists a map ${\langle f \rangle: A \to \underbrace{A + \hdots + A}_{n\text{-times}}}$ such that the following equalities hold: 
\begin{enumerate}[{\bf [D.1]}]
\item $\overline{f} = \langle f \rangle \nabla$ 
\item $\langle f \rangle \left( f + \hdots + f \right) = f \left( \iota_1 + \hdots + \iota_n \right)$
\end{enumerate}
or equivalently if $\langle f \rangle$ satisfies the following: 
\begin{enumerate}[{\bf [D.1]}]
\setcounter{enumi}{2}
\item $\langle f \rangle$ is a restriction inverse of ${\left[ \begin{smallmatrix} \iota^\circ_1 f \\ \vdots \\ \iota^\circ_n f \end{smallmatrix} \right]}: A + \hdots + A \to A$
\item $\overline{\langle f \rangle} = \overline{f}$. 
\end{enumerate}
The map $\langle f \rangle$ is called the \textbf{decision} \cite[Prop 2.11]{cockett2007restriction} of $f$. 
\end{definition}

To help keep track of types, it may be useful to write down \textbf{[D.1]} and \textbf{[D.2]} as commutative diagrams (for clarity we write brackets in the bottom corner):
 \begin{equation*}\begin{gathered}\label{diag:D}
\xymatrixcolsep{5pc}\xymatrix{ A \ar[r]^-{ \langle f \rangle } \ar[dr]_-{ \overline{f} } & A + \hdots + A \ar[d]^-{\nabla} &  A \ar[r]^-{ \langle f \rangle } \ar[d]_-{ f }  & A + \hdots + A \ar[d]^-{f + \hdots + f} \\ 
& A & B_1 + \hdots + B_n \ar[r]_-{ \iota_1 + \hdots + \iota_n } & (B_1 + \hdots + B_n) + \hdots + (B_1 + \hdots + B_n) } \end{gathered}\end{equation*}
where for the bottom arrow on the right, it is the coproduct of each injections $\iota_j: B_j \to B_1 + \hdots + B_n$. Moreover, note that by \textbf{[D.3]}, since restriction inverses are unique, the decision $\langle f \rangle$ is the unique map which satisfies \textbf{[D.1]} and \textbf{[D.2]}. 

Intuitively, in an extensive restriction category, the restriction coproduct should be thought as a disjoint union. So then given a map $f: A \to B_0 + \hdots + B_n$, when $f(x)$ is defined, it lands in only one of the $B_j$. As such, the decision of $f$ should be thought of as separating out where $f$ is defined in the appropriate place. So intuitively, if $f(x) \uparrow$ then $\langle f \rangle(x) \uparrow$, and when $f(x) \downarrow$ and $f(x) \in B_j$, then $\langle f \rangle$ sends $x$ to $x$ in the $j$th copy of $A$ in the codomain. So when $f$ is total this gives a decision procedure. This intuition is captured by our main example of an extensive restriction category in Ex \ref{ex:PAR-ext} below. We can also consider maps that are their own decision.

\begin{definition}\label{def:decision} In an extensive restriction category, a \textbf{decision} \cite[Prop 2.15]{cockett2007restriction} is map $d: A \to \underbrace{A + \hdots + A}_{n\text{-times}}$ such that $\langle d \rangle = d$, in other words, $d$ satisfies:
\begin{enumerate}[{\bf [d.1]}]
\item $\overline{d} = d \nabla$ 
\item $d  \left( d + \hdots + d \right) = d \left( \iota_1 + \hdots + \iota_n \right)$
\end{enumerate}
\end{definition}

Again, to help keep track of types, it may be useful to write down \textbf{[d.1]} and \textbf{[d.2]} as commutative diagrams (for clarity we will write brackets in the bottom corner):
 \begin{equation*}\begin{gathered}\label{diag:d}
\xymatrixcolsep{5pc}\xymatrix{ A \ar[r]^-{ d } \ar[dr]_-{ \overline{d} } & A + \hdots + A \ar[d]^-{\nabla} &  A \ar[r]^-{ d } \ar[d]_-{ d }  & A + \hdots + A \ar[d]^-{d + \hdots + d} \\ 
& A & A + \hdots + A \ar[r]_-{ \iota_1 + \hdots + \iota_n } & (A + \hdots + A) + \hdots + (A + \hdots + A) } \end{gathered}\end{equation*}
where for the bottom arrow on the right, it is the coproduct of each injections $\iota_j: A \to A + \hdots + A$. 

Of course, for every suitable $f$, its decision $\langle f \rangle$ is indeed a decision, so $\langle \langle f \rangle \rangle = \langle f \rangle$ \cite[Prop 2.15]{cockett2007restriction}. Equivalently, a decision can also be characterize as a restriction isomorphism whose restriction inverse is the copairing of restriction idempotents.

\begin{lemma}\label{lem:decision=iso} \cite[Prop 2.15]{cockett2007restriction} In an extensive restriction category, 
\begin{enumerate}[{\em (i)}]
\item \label{lem:decision=iso.1} A map $d: A \to A + \hdots +A$ is a decision if and only if $d$ is a restriction isomorphism such that for its restriction inverse $d^\circ: A + \hdots + A \to A$, $\iota_j d^\circ$ are restriction idempotents.
\item \label{lem:decision=iso.2} If $d$ is a decision, then it is a restriction isomorphism with restriction inverse $d^\circ = \left[ \begin{smallmatrix} d \iota^\circ_1 \\ \vdots \\  d \iota^\circ_n \end{smallmatrix}\right]$. 
\end{enumerate}
\end{lemma}

As such, intuitively, an $n$-ary decision of type $A$ should be interpreted as a collection of $n$ disjoint subsets of $A$, as is best seen in our main example below. 

\begin{example}\normalfont \label{ex:PAR-ext} $\mathsf{PAR}$ is an extensive restriction category where the coproduct is given by disjoint union, which we denote as $X_1 + \hdots + X_n = X_1 \sqcup \hdots \sqcup X_n = \lbrace (x,k) \vert~ 1 \leq k \leq n, x \in X_k \rbrace$, the initial object is the empty set  $\mathsf{0} = \emptyset$, and for a partial function $f: X \to Y_1 \sqcup \hdots \sqcup Y_n$, its decision is the partial function ${\langle f \rangle: X \to X \sqcup \hdots \sqcup X}$ defined as follows: 
\[ \langle f \rangle (x) = \begin{cases} (x, 1) & \text{if } f(x) \downarrow \text{ and } f(x) \in Y_1 \\
\vdots \\
(x,n) & \text{if } f(x) \downarrow \text{ and } f(x) \in Y_n \\
\uparrow & \text{if } f(x) \uparrow  \end{cases} \]
Decisions in $\mathsf{PAR}$ correspond to a finite family of disjoint subsets. Indeed, given a set $X$, an $n$-ary decision ${d: X \to X \sqcup \hdots \sqcup X}$ induces $n$ subsets of $X$ defined as $d_k = \lbrace x \in X \vert d(x) = (x,k) \rbrace$ for all $1 \leq k \leq n$ (which is well-defined since $d = \langle d \rangle$). Moreover, clearly for any $1 \leq i,j \leq n$ with $i \neq j$ we have that $d_i \cap d_j = \emptyset$. On the other hand, given a finite family of subsets $U_1, \hdots, U_n \in X$ which are pairwise disjoint, that is, $U_i \cap U_j = \emptyset$ when $ i\neq j$, we obtain a decision $d_{U_1,\hdots, U_n}: X \to X \sqcup \hdots \sqcup X$ defined as: 
\[ d_{U_1,\hdots, U_n} (x) = \begin{cases} (x, 1) & \text{if } x \in U_1 \\
\vdots \\
(x,n) & \text{if } x \in U_n \\
\uparrow & \text{if } x \notin U_i \text{ for all $1 \leq i \leq n$}  \end{cases} \]
\end{example}

\begin{example}\normalfont Another source of examples of extensive restriction categories worth mentioning are \emph{distributive} restriction categories with restriction zeroes. Briefly, a \textbf{distributive restriction category} \cite[Sec 5.3]{cockett2007restriction} is a coCartesian restriction category which also has \textbf{restriction products} $\times$ and a \textbf{restriction terminal object} $\mathsf{1}$ \cite[Sec 4.1]{cockett2007restriction}\footnote{It is important to recall that while restriction coproducts are the indeed usual coproducts, restriction products are not products in the usual sense. For more details on restriction products, we invite the reader to see \cite{cockett2007restriction,cockett2023classical,cockett2012differential}.} which distributes over the restriction coproducts, in the sense that $A \times (B + C) \cong (A \times B) + (A \times C)$. Every distributive restriction category with restriction zeroes is an extensive restriction category \cite[Thm 5.8]{cockett2007restriction}, where for a map ${f: A \to B_0 +  \hdots +  B_n}$, its decision ${\langle f \rangle: A \to A +  \hdots +  A}$ is defined as the following composite: 
 \begin{equation}\begin{gathered}\label{def:dec}
\begin{array}[c]{c} \langle f \rangle \end{array}
:= \begin{array}[c]{c} \xymatrixcolsep{4.75pc}\xymatrixrowsep{1pc}\xymatrix{ A \ar[r]^-{ \left\langle 1_A, f \right\rangle }  & A \times \left( B_0 +  \hdots +  B_n \right) \ar[r]^-{ 1_A \times \left(t_{B_0} + \hdots + t_{B_n}\right)} & A \times (\mathsf{1} + \hdots + \mathsf{1}) \cong A + \hdots + A } \end{array}  \end{gathered}\end{equation}
where $\langle -, - \rangle$ is the pairing operation for the restriction product, and $t_{B_i}: B_i \to \mathsf{1}$ is the unique total map to the restriction terminal object. As such, every classical distributive restriction category \cite[Sec 6]{cockett2023classical} (or equivalently any Kleisli category of the exception monad of a distributive category \cite[Sec 7]{cockett2023classical}) is an extensive restriction category as well. For more details on (classical) distributive restriction categories, we invite the curious reader to see \cite{cockett2007restriction,cockett2023classical}. 
\end{example}

We now turn our attention to showing that every extensive restriction category has a canonical interference relation induced by decisions. The natural notion of maps being disjoint in an extensive restriction category is if they can separated by a binary decision. 

\begin{definition}\label{def:perp-extensive} In an extensive restriction category, we say that maps $f: A \to B$ and ${g: A \to C}$ are \textbf{decision disjoint}, written as $f \perp_d g$, if there exists a decision $\langle f \vert g \rangle: A \to A + A$ such that $\langle f \vert g \rangle \iota^\circ_1 = \overline{f}$ and $\langle f \vert g \rangle \iota^\circ_2 = \overline{g}$. We call $\langle f \vert g \rangle$ the \textbf{separating decision} of $f$ and $g$. 
\end{definition}

Recall that a binary decision should intuitively be thought of as a pair of disjoint subsets. Thus two maps are decision disjoint if their domains are disjoint and if there is a decision which captures these two disjoint subsets. Being decision disjoint is also equivalent to asking that the copairing of the restrictions is a restriction isomorphism, whose restriction inverse will be the separating decision. From this, since restriction inverses are unique, we also get that separating decision are unique (which justifies saying \emph{the} separating decision instead of \emph{a} separating decision). 

\begin{lemma}\label{lem:decision-sep-iso} In an extensive restriction category, 
\begin{enumerate}[{\em (i)}]
\item \label{lem:decision-sep-iso.1} $f \perp_d g$ if and only if $\left[ \begin{smallmatrix} \overline{f} \\ \overline{g} \end{smallmatrix}\right]$ is a restriction isomorphism.
\item \label{lem:decision-sep-iso.2} If $f \perp_d g$, then their separating decision $\langle f \vert g \rangle$ is the restriction inverse of $\left[ \begin{smallmatrix} \overline{f} \\ \overline{g} \end{smallmatrix}\right]$.
\item \label{lem:decision-sep-iso.3} A separating decision, if it exists, is unique.
\end{enumerate} 
\end{lemma}
\begin{proof} Starting with (\ref{lem:decision-sep-iso.1}). For the $\Rightarrow$ direction, suppose that $f \perp_d g$, so we have a decision $\langle f \vert g \rangle$ such that $\langle f \vert g \rangle \iota^\circ_1 = \overline{f}$ and $\langle f \vert g \rangle \iota^\circ_2 = \overline{g}$. By Lemma \ref{lem:decision=iso}.(\ref{lem:decision=iso.2}), we know that $\langle f \vert g \rangle$ is a restriction isomorphism and its restriction inverse is $\langle f \vert g \rangle^\circ = \left[ \begin{smallmatrix} \langle f \vert g \rangle \iota^\circ_1 \\  \langle f \vert g \rangle \iota^\circ_2 \end{smallmatrix}\right] = \left[ \begin{smallmatrix} \overline{f} \\ \overline{g} \end{smallmatrix}\right]$. Since restriction inverses are themselves restriction isomorphisms, we have that $\left[ \begin{smallmatrix} \overline{f} \\ \overline{g} \end{smallmatrix}\right]$ is a restriction isomorphism. Conversely, for the $\Leftarrow$ direction, suppose that $\left[ \begin{smallmatrix} \overline{f} \\ \overline{g} \end{smallmatrix}\right]$ is a restriction isomorphism and call its restriction inverse $\langle f \vert g \rangle$. By Lemma \ref{lem:decision=iso}.(\ref{lem:decision=iso.1}), since $ \iota_1\left[ \begin{smallmatrix} \overline{f} \\ \overline{g} \end{smallmatrix}\right]  = \overline{f}$ and $\iota_2\left[ \begin{smallmatrix} \overline{f} \\ \overline{g} \end{smallmatrix}\right]  = \overline{g}$ are restriction idempotents, it follows that $\langle f \vert g \rangle$ is a decision. Thus by Lemma \ref{lem:decision=iso}.(\ref{lem:decision=iso.2}), we get that $\langle f \vert g \rangle$ is a restriction isomorphism whose restriction inverse is $\left[ \begin{smallmatrix} \langle f \vert g \rangle \iota^\circ_1 \\ \langle f \vert g \rangle \iota^\circ_2 \end{smallmatrix}\right]$. However by assumption, we know that since $\langle f \vert g \rangle$ is the restriction inverse of $\left[ \begin{smallmatrix} \overline{f} \\ \overline{g} \end{smallmatrix}\right]$, that $\left[ \begin{smallmatrix} \overline{f} \\ \overline{g} \end{smallmatrix}\right]$ is the restriction inverse of $\langle f \vert g \rangle$. Then since restriction inverses are unique, we get that $\left[ \begin{smallmatrix} \langle f \vert g \rangle \iota^\circ_1 \\ \langle f \vert g \rangle \iota^\circ_2 \end{smallmatrix}\right]=\left[ \begin{smallmatrix} \overline{f} \\ \overline{g} \end{smallmatrix}\right]$. Thus it follows that $\langle f \vert g \rangle \iota^\circ_1 = \overline{f}$ and $\langle f \vert g \rangle \iota^\circ_2 = \overline{g}$. So $f \perp_d g$ as desired. 

Then (\ref{lem:decision-sep-iso.2}) follows from what we've just shown. Lastly for (\ref{lem:decision-sep-iso.3}), since restriction inverse are unique, separating decision are unique as well. 
\hfill \end{proof}

We now show that decision separation is an interference relation and that extensive restriction categories have all binary disjoint joins.

\begin{theorem}\label{thm:ext-int} Let $\mathbb{X}$ be an extensive restriction category, then $\perp_d$ is an interference relation, making $\mathbb{X}$ into a $\perp_d$-restriction category which has all binary $\perp_d$-joins. 
\end{theorem}
\begin{proof} To prove that $\perp_d$ is an interference relation, we will prove that $\perp_d$ is a restrictional interference relation. Explicitly, the restriction idempotent version of $\perp_d$ is given as follows: for restriction idempotents $e_1$ and $e_2$, $e_1 \perp_d e_2$ if there exists a decision $\langle e_1 \vert e_2 \rangle: A \to A + A$ such that $\langle e_1 \vert e_2 \rangle \iota^\circ_1 = e_1$ and $\langle e_1 \vert e_2 \rangle \iota^\circ_2 = e_2$. Equivalently by Lemma \ref{lem:decision-sep-iso}.(\ref{lem:decision-sep-iso.1}), $e_1 \perp_d e_2$ if and only if $\left[ \begin{smallmatrix} e_1 \\ e_2 \end{smallmatrix}\right]$ is a restriction isomorphism. We will make use of both descriptions. 

\begin{enumerate}[{\bf [{$\mathcal{O}\!\boldsymbol{\perp}$}.1]}]
\setcounter{enumi}{-1}
\item Recall that $\left[ \begin{smallmatrix} 1_A \\ 0 \end{smallmatrix}\right] = \iota^\circ_1$ is the quasi-projection, which we know is a restriction isomorphism. So $1_A \perp 0$, and moreover $\langle 1_A \vert 0 \rangle = \iota_0$. 

\item Suppose that $e_1 \perp_d e_2$, which implies that $\left[ \begin{smallmatrix} e_1 \\ e_2 \end{smallmatrix}\right]$ is a restriction isomorphism. First consider the canonical symmetry isomorphism of the coproduct, $\sigma := \left[ \begin{smallmatrix} \iota_2 \\ \iota_1 \end{smallmatrix}\right]$, which is a restriction isomorphism (since every isomorphism is a restriction isomorphism). Now observe that $\left[ \begin{smallmatrix} e_2 \\ e_1 \end{smallmatrix}\right] = \sigma \left[ \begin{smallmatrix} e_1 \\ e_2\end{smallmatrix}\right]$. Then $\left[ \begin{smallmatrix} e_2 \\ e_1 \end{smallmatrix}\right]$ is the composite of restriction isomorphisms, and so is itself a restriction isomorphism. So $e_2 \perp_d e_1$, and moreover $\langle e_2 \vert e_1 \rangle =\langle e_1 \vert e_2 \rangle \sigma$. 

\item Suppose that $e \perp_{d} e$. Then we have a decision $\langle e \vert e \rangle$ such that $\langle e \vert e \rangle\iota^\circ_1=e$ and $\langle e \vert e \rangle\iota^\circ_2=e$. Now recall that by using \textbf{[R.4]}, it follows that for any map $h$ that $he= \overline{he}h$. As such, using that restriction idempotents are idempotent, the naturality of the quasi-projections, and \textbf{[d.2]}, we can compute that:
\[ e = ee = \langle e \vert e \rangle \iota^\circ_1 \langle e \vert e \rangle \iota^\circ_2 =  \langle e \vert e \rangle \left( \langle e \vert e \rangle + \langle e \vert e \rangle \right) \iota^\circ_1 \iota^\circ_2 = \langle e \vert e \rangle (\iota_1 + \iota_2) \iota^\circ_1 \iota^\circ_2 = \langle e \vert e \rangle \iota^\circ_1 \iota_1 \iota^\circ_2 =  \langle e \vert e \rangle \iota^\circ_1 0 = 0  \]
So $e=0$.

\item Suppose that $e^\prime_1 \perp_d e^\prime_2$, so $\left[ \begin{smallmatrix} e^\prime_1 \\ e^\prime_2 \end{smallmatrix}\right]$ is a restriction isomorphism, and also suppose that $e_1 \leq e_1^\prime$ and $e_2 \leq e^\prime_2$. First observe that from these inequalities, we get that $\left[ \begin{smallmatrix} e_1 \\ e_2 \end{smallmatrix}\right] = \left[ \begin{smallmatrix} e_1 e^\prime_1 \\ e_2 e_2^\prime \end{smallmatrix}\right] = (e_1 + e_2) \left[ \begin{smallmatrix} e^\prime_1 \\ e^\prime_2 \end{smallmatrix}\right]$. So $\left[ \begin{smallmatrix} e_1 \\ e_2 \end{smallmatrix}\right]= (e_1 + e_2) \left[ \begin{smallmatrix} e^\prime_1 \\ e^\prime_2 \end{smallmatrix}\right]$. However note that $e_1 + e_2$ is a restriction idempotent, and therefore also a restriction isomorphism. As such $\left[ \begin{smallmatrix} e_1 \\ e_2 \end{smallmatrix}\right]$ is the composite of restriction isomorphisms, and so is itself a restriction isomorphism. So $e_0 \perp_d e_1$, and moreover $\langle e_1 \vert e_2 \rangle = \langle e^\prime_1 \vert e^\prime_2 \rangle (e_1 + e_2)$.

\item Suppose that $e_1 \perp_d e_2$, which implies we have a decision $\langle e_1 \vert e_2 \rangle: A \to A + A$ such that $\langle e_1 \vert e_2 \rangle \iota^\circ_1 = e_1$ and $\langle e_1 \vert e_2 \rangle \iota^\circ_2 = e_2$. Given a map $h: A^\prime \to A$, consider the map $h\langle e_1 \vert e_2 \rangle: A^\prime \to A + A$. Then define $\langle \overline{he_0} \vert  \overline{he_1}\rangle := \left\langle h \langle e_1 \vert e_2 \rangle \right \rangle: A^\prime \to A^\prime + A^\prime$. By construction $\langle \overline{he_0} \vert  \overline{he_1}\rangle$ is decision. So it remains to show that it is the desired separating decision. To do so, we note that $\langle f \rangle \iota^\circ_j = \overline{f \iota^\circ_j}$ \cite[Proof of Prop 2.15]{cockett2007restriction}. Using this, we can compute that: 
\[ \langle \overline{he_1} \vert  \overline{he_2}\rangle \iota^\circ_1 = \left \langle h \langle e_1 \vert e_2 \rangle \right \rangle \iota^\circ_1 = \overline{h \langle e_2 \vert e_2 \rangle \iota^\circ_1} = \overline{ h e_1} \]
So $\langle \overline{he_1} \vert  \overline{he_2}\rangle \iota^\circ_1 =  \overline{h e_1}$, and similarly $\langle \overline{he_1} \vert  \overline{he_2}\rangle \iota^\circ_2 =  \overline{h e_2}$. So we conclude that $ \overline{he_1} \perp_d  \overline{he_2}$. 
\end{enumerate}

So we have that $\perp_d$ is a restrictional interference relation. Now this induces an interference relation $\perp_d$ where $f \perp_d g$ if and only if $\overline{f} \perp_d \overline{g}$, meaning that there exists a decision $\langle \overline{f} \vert \overline{g} \rangle$ such that $\langle \overline{f} \vert \overline{g} \rangle \iota^\circ_0 = \overline{f}$ and $\langle \overline{f} \vert \overline{g} \rangle \iota^\circ_1 = \overline{g}$. Then setting $\langle f \vert g \rangle := \langle \overline{f} \vert \overline{g} \rangle$, we see that the induced $\perp_d$ is defined in exactly the same way as in Def \ref{def:perp-extensive}. So $\perp_d$ is indeed an interference relation. 

We now show that we also have binary $\perp_d$-joins. So given a pair of parallel maps $f: A \to B$ and $g: A \to B$ which are $\perp_d$-disjoint, so $f \perp_d g$, define the map $f \sqcup g: A \to B$ as the following composite: 
 \begin{equation}\begin{gathered}\label{eq:sqcup-dec}
\begin{array}[c]{c} f \sqcup g \end{array}
:= \begin{array}[c]{c} \xymatrixcolsep{4.75pc}\xymatrixrowsep{1pc}\xymatrix{ A \ar[r]^-{ \left\langle f \vert g\right\rangle }  & A + A \ar[r]^-{ f + g} & B + B \ar[r]^-{ \nabla } & B } \end{array}  \end{gathered}\end{equation}
\begin{enumerate}[{\bf [$\boldsymbol{\sqcup}$.1]}]
\item Using Lemma \ref{lem:decision=iso}.(\ref{lem:decision=iso.2}), it is easy to check that for a decision $d$ that $\overline{d \iota^\circ_j} d = \overline{d \iota^\circ_j} \iota_j$. In particular, this implies that $\overline{f} \left\langle f \vert g\right\rangle = f$ and $\overline{g} \left\langle f \vert g\right\rangle = g$. So using this fact, it follows that we can compute: 
\[ \overline{f}(f \sqcup g) = \overline{f} \left\langle f \vert g\right\rangle (f+g) \nabla = \overline{f} \iota_1  (f+g) \nabla = \overline{f} f \iota_1 \nabla = f \iota_1 \nabla = f \]
So $f \leq f \sqcup g$, and similarly we can also show that $g \leq f \sqcup g$. 

\item Suppose that we have another map $h$ such that $f \leq h$ and $g \leq h$.  This also implies that $f+g \leq h+h$. As such, using \textbf{[d.1]} and the fact that if $e$ is a restriction idempotent then $eh \leq h$, we can compute that: 
\[ f \sqcup g = \left\langle f \vert g\right\rangle (f+g) \nabla \leq \left\langle f \vert g\right\rangle (h+h) \nabla =  \left\langle f \vert g\right\rangle \nabla h = \overline{\left\langle f \vert g\right\rangle} h \leq h \]
So $f \sqcup g \leq h$. 

\item Consider a map $k: A^\prime \to A$. By the proof of {\bf [{$\mathcal{O}\!\boldsymbol{\perp}$}.4]} above, it follows that $\langle kf \vert kg \rangle = \langle k \langle f \vert g \rangle \rangle$. Next we compute that: 
\begin{align*}
&\langle k \langle f \vert g \rangle \rangle (k+k) (f+g) =~ \langle k \langle f \vert g \rangle \rangle (k+k) (\overline{f} + \overline{g}) (f+g) \tag{\textbf{[R.1]}} \\
&=~ \langle k \langle f \vert g \rangle \rangle (k+k) (\langle f \vert g \rangle + \langle f \vert g \rangle)(\iota^\circ_1 + \iota^\circ_2) (f+g) \tag{Sep. Dec.} \\
&=~ k \langle f \vert g \rangle (\iota_1 + \iota_2) (\iota^\circ_1 + \iota^\circ_2) (f+g) \tag{\textbf{[D.2]}} \\
&=~ k \langle f \vert g \rangle (f+g) \tag{Def. of $\iota^\circ_j$}
\end{align*}
In other words, we have that $\langle kf \vert kg \rangle(kf+kg) =  k \langle f \vert g \rangle (f+g)$. From this we can easily compute that: 
\[ kf \sqcup kg = \left\langle kf \vert kg\right\rangle (kf+kg) \nabla =  k \langle f \vert g \rangle (f+g) \nabla = k \left( f \sqcup g\right) \]
So $kf \sqcup kg = k \left( f \sqcup g\right)$ as desired. 
\end{enumerate}
So we conclude that an extensive restriction category is a $\perp_d$-restriction category with all binary $\perp_d$-joins. 
\end{proof}

In general, however, extensive restriction categories may not have arbitrary finite $\perp_d$-joins, since $\perp_d$-joins are not necessarily strong. The issue is that for a finite family of pairwise $\perp_d$-disjoint maps, even if pairwise we have binary separating decision for them, to build their join one really requires an $n$-ary separating decision. 

\begin{definition}\label{def:dec-sep} In an extensive restriction category $\mathbb{X}$, a finite family of maps $\lbrace f_1, \hdots, f_n \rbrace$ that have the same domain $A$ are said to be \textbf{decision disjoint} if there exists an $n$-ary decision $\langle f_1\vert\hdots\vert f_n \rangle: A \to \underbrace{A + \hdots + A}_{n\text{-times}}$ such that $\langle f_1\vert\hdots\vert f_n \rangle \iota^\circ_j = \overline{f_j}$. We call $\langle f_1\vert\hdots\vert f_n \rangle$ the \textbf{separating decision} of the $f_j$s. 
\end{definition}

Of course, similarly to the binary case, we can also say being decision disjoint is equivalent to asking that the copairing of the restrictions is a restriction isomorphism. Moreover, being $n$-ary decision disjoint implies that they are pairwise binary decision disjoint. 

\begin{lemma}\label{lem:n-decision-sep-iso} In an extensive restriction category, 
\begin{enumerate}[{\em (i)}]
\item \label{lem:n-decision-sep-iso.1} $\lbrace f_1, \hdots, f_n \rbrace$ are decision disjoint if and only if $\left[ \begin{smallmatrix} \overline{f_1} \\ \vdots \\ \overline{f_n} \end{smallmatrix}\right]$ is a restriction isomorphism.
\item \label{lem:n-decision-sep-iso.2} If $\lbrace f_1, \hdots, f_n \rbrace$ are decision disjoint, then their separating decision $\langle f_1\vert\hdots\vert f_n \rangle$ is the restriction inverse of $\left[ \begin{smallmatrix} \overline{f_1} \\ \vdots \\ \overline{f_n} \end{smallmatrix}\right]$.
\item \label{lem:n-decision-sep-iso.3} A separating decision, if it exists, is unique.
\item \label{lem:n-decision-sep-iso.4} If $\lbrace f_1, \hdots, f_n \rbrace$ are decision disjoint, then $\lbrace f_1, \hdots, f_n \rbrace$ is also a family of pairwise $\perp_d$-disjoint maps, that is, $f_i \perp_d f_j$ for all $i \neq j$. 
\item \label{lem:n-decision-sep-iso.5}The quasi-projections ${\iota^\circ_j: A_1 + \hdots + A_n \to A_j}$ are decision disjoint and therefore pairwise $\perp$-disjoint.
\end{enumerate} 
\end{lemma}
\begin{proof} Parts (\ref{lem:n-decision-sep-iso.1}), (\ref{lem:n-decision-sep-iso.2}), and (\ref{lem:n-decision-sep-iso.3}) are $n$-ary versions of Lemma \ref{lem:decision-sep-iso}.(\ref{lem:decision-sep-iso.1}), (\ref{lem:decision-sep-iso.2}), and (\ref{lem:decision-sep-iso.3}) respectively, so they are proven in similar fashion. For (\ref{lem:n-decision-sep-iso.4}), suppose $\lbrace f_1, \hdots, f_n \rbrace$ are decision disjoint, so we have an $n$-ary separating decision $\langle f_1\vert\hdots\vert f_n \rangle$. Now let $\iota^\circ_{i,j}: A_1 + \hdots + A_n \to A_i + A_j$ be the obvious quasi-projection which picks out the $A_i$ and $A_j$ components. Then for $f_i$ and $f_j$, with $i \neq j$, define $\langle f_i \vert f_j \rangle = \lbrace f_1, \hdots, f_n \rbrace \iota^\circ_{i,j}$. It is straightforward to check that $\langle f_i \vert f_j \rangle$ is indeed a binary separating decision for $f_i$ and $f_j$, and therefore we have that $f_i \perp_d f_j$ as desired. Lastly for (\ref{lem:n-decision-sep-iso.5}), it is straightforward to see that the quasi-projections $\lbrace \iota^\circ_1, \hdots, \iota^\circ_n \rbrace$ are  decision disjoint, with their separating decision being $\iota_1 + \hdots + \iota_n$. As such by (\ref{lem:n-decision-sep-iso.4}), the quasi-projection are indeed also a family of $\perp_d$-disjoint maps.
\end{proof}

It is again worth stressing that the converse of Lemma \ref{lem:n-decision-sep-iso}.(\ref{lem:n-decision-sep-iso.4}) is not necessarily true. Indeed, even if $\lbrace f_1, \hdots, f_n \rbrace$ are pairwise $\perp_d$-disjoint, they are not necessarily decision disjoint since we cannot in general build the necessary $n$-ary separating decision from the binary ones. However, as mentioned above, we do get disjoint joins of decision disjoint maps by the using the $n$-ary decision separation. So while we may not have all finite $\perp_d$-joins, we do get disjoint joins of a particular class of $\perp_d$-disjoint maps, the ones that are also decision disjoint. 

\begin{lemma}In an extensive restriction category, for a finite family of parallel maps $\lbrace f_1, \hdots, f_n \vert~ f_i: A \to B \rbrace$ which are decision disjoint, their $\perp_d$-join exists and is defined as follows: 
 \begin{equation}\begin{gathered}\label{eq:sqcup-n-dec}
\begin{array}[c]{c} f_1 \sqcup \hdots \sqcup f_n \end{array}
:= \begin{array}[c]{c} \xymatrixcolsep{4.75pc}\xymatrixrowsep{1pc}\xymatrix{ A \ar[r]^-{ \left\langle f_1 \vert \hdots \vert f_n\right\rangle }  & A + \hdots + A \ar[r]^-{f_1 + \hdots + f_n} & B + \hdots + B \ar[r]^-{ \nabla } & B } \end{array}  \end{gathered}\end{equation}
Moreover, composition preserves these joins in the sense that Lemma \ref{lemma:disjoint-join}.(\ref{lemma:join-comp}) holds. 
\end{lemma}
\begin{proof} This is shown in similar way to how the binary case was shown in the proof of Thm \ref{thm:ext-int}. So we leave this as an excercise for the reader.
\end{proof}

With all that said, in many examples of extensive restriction categories, $\perp_d$-joins are in fact strong. When this is the case, the converse of Lemma \ref{lem:n-decision-sep-iso}.(\ref{lem:n-decision-sep-iso.4}) is true. 

\begin{lemma}\label{lemma:strong-perpd} In an extensive restriction category such that finite $\perp_d$-join are strong, then $\lbrace f_1, \hdots, f_n \rbrace$ are decision disjoint if and only if $\lbrace f_1, \hdots, f_n \rbrace$ are pairwise $\perp_d$-disjoint.
\end{lemma}
\begin{proof} The $\Rightarrow$ direction is Lemma \ref{lem:n-decision-sep-iso}.(\ref{lem:n-decision-sep-iso.4}). For the $\Leftarrow$ direction, suppose that $\lbrace f_1, \hdots, f_n \rbrace$ are pairwise $\perp_d$-disjoint. Then by {\bf [$\boldsymbol{\perp}$.4]} and {\bf [$\boldsymbol{\perp}$.5]}, we also get that $\lbrace \overline{f_1} \iota_1, \hdots, \overline{f_n} \iota_n \rbrace$ are pairwise $\perp_d$-disjoint. Then we can take their join, and so define $\langle f_1 \vert \hdots \vert f_n \rangle: A \to A + \hdots + A$ as follows:
\begin{align} 
\langle f_1 \vert \hdots \vert f_n \rangle := \overline{f}\iota_1 \sqcup \hdots \sqcup \overline{f} \iota_n
\end{align}
We show that this is a separating decision by showing that it is the restriction inverse of $\left[ \begin{smallmatrix} \overline{f_1} \\ \vdots \\ \overline{f_n} \end{smallmatrix}\right]$. So we first compute that (recall that in a restriction category, if $y$ is total then $\overline{xy} = \overline{x}$): 
\begin{align*}
\langle f_1 \vert \hdots \vert f_n \rangle \begin{bmatrix} \overline{f_1} \\ \vdots \\ \overline{f_n} \end{bmatrix} &=~ \left( \overline{f_1}\iota_1 \sqcup \hdots \sqcup \overline{f_n} \iota_n \right) \begin{bmatrix} \overline{f} \\ \overline{g} \end{bmatrix} \tag{Def.} \\
&=~ \overline{f_1}\iota_1\begin{bmatrix} \overline{f_1} \\ \vdots \\ \overline{f_n} \end{bmatrix} \sqcup \hdots  \overline{f_n} \iota_n \begin{bmatrix} \overline{f_1} \\ \vdots \\ \overline{f_n} \end{bmatrix} \tag{Lemma \ref{lemma:disjoint-join}.(\ref{lemma:join-comp})} \\
&=~ \overline{f_1}~\overline{f_1} \sqcup \hdots \sqcup \overline{f_n}~\overline{f_n} \\
&=~ \overline{f_1} \sqcup \hdots \sqcup \overline{f_n} \tag{Rest. Idem.} \\
&=~ \overline{\overline{f_1}} \sqcup \hdots \sqcup \overline{\overline{f_n}} \tag{Rest. Idem.} \\
&=~ \overline{\overline{f_1}\iota_1} \sqcup \hdots \sqcup \overline{\overline{f_n}\iota_n} \tag{$\iota_j$ total} \\
&=~ \overline{\overline{f_1}\iota_1 \sqcup \hdots \sqcup \overline{f_n}\iota_n} \tag{Lemma \ref{lemma:disjoint-join}.(\ref{lemma:join-rest})} \\
&=~ \overline{\langle f_1 \vert \hdots \vert f_n \rangle} \tag{Def.}
\end{align*}
Next we compute that:
\begin{align*}
 \iota_j \begin{bmatrix} \overline{f_1} \\ \vdots \\ \overline{f_n} \end{bmatrix} \langle f_1 \vert \hdots \vert f_n \rangle = \overline{f_j} \langle f_1 \vert \hdots \vert f_n \rangle = \overline{f_j}  \left( \overline{f_1}\iota_1 \sqcup \hdots \sqcup \overline{f_n} \iota_n \right) =  \overline{f_j}~\overline{f_1}\iota_1 \sqcup \hdots \sqcup \overline{f_j}~\overline{f_n} \iota_n  = 0 \sqcup \hdots \sqcup 0 \sqcup \overline{f_j} \iota_j \sqcup 0 \sqcup \hdots \sqcup 0  = \overline{f_j}\iota_j
\end{align*}
Then by the universal property of the coproduct, we get that $\left[ \begin{smallmatrix} \overline{f_1} \\ \vdots \\ \overline{f_n} \end{smallmatrix}\right] \langle f_1 \vert \hdots \vert f_n \rangle = \overline{f_1} + \hdots + \overline{f_n}$. However, recall that the restriction of the copairing is the coproduct of the restrictions \cite[Lemma 2.1]{cockett2007restriction}, so $\left[ \begin{smallmatrix} \overline{f_1} \\ \vdots \\ \overline{f_n} \end{smallmatrix}\right] \langle f_1 \vert \hdots \vert f_n \rangle = \overline{f_1} + \hdots + \overline{f_n} = \overline{\left[ \begin{smallmatrix} \overline{f_1} \\ \vdots \\ \overline{f_n} \end{smallmatrix}\right]}$. So we conclude that $\left[ \begin{smallmatrix} \overline{f_1} \\ \vdots \\ \overline{f_n} \end{smallmatrix}\right]$ is the restriction inverse of $\left[ \begin{smallmatrix} \overline{f_1} \\ \vdots \\ \overline{f_n} \end{smallmatrix}\right]$. Therefore by Lemma \ref{lem:n-decision-sep-iso}.(\ref{lem:n-decision-sep-iso.1}), we have that $\lbrace f_1, \hdots, f_n \rbrace$ are decision disjoint. 
\end{proof}

A natural question to ask is if an extensive restriction category can have more than one (non-trivial) interference relation. It turns out that under the mild natural assumption that the quasi-projection are pairwise disjoint and binary disjoint joins exist, the only interference relation which satisfies these assumptions is precisely the decision separation interference. As such, this justifies that $\perp_d$ can indeed be considered as the canonical interference relation for an extensive restriction category. 

\begin{lemma} Let $\mathbb{X}$ be an extensive restriction category. Suppose that $\mathbb{X}$ admits an interference relation $\perp$ such that all binary $\perp$-joins exist and such that for any finite family of object $A_1, \hdots, A_n$, the quasi-projections $\iota^\circ_i: A_1 + \hdots + A_n \to A_i$ are pairwise $\perp$-disjoint, that is, $\iota_i^\circ \perp \iota_j^\circ$ for all $1 \leq i,j \leq n$. Then $\perp = \perp_d$. 
\end{lemma}
\begin{proof} Suppose that $f \perp g$. We need to show that $f$ and $g$ are also decision disjoint. To do so, note that since $f \perp g$, by {\bf [$\boldsymbol{\perp}$.4]} we also have that $f \iota_1 \perp g \iota_2$. Since we have binary $\perp$-joins, we define the map $\langle f \vert g \rangle: A \to A + A$ as $\langle f \vert g \rangle := \overline{f}\iota_1 \sqcup \overline{g} \iota_1$. We can show that $\langle f \vert g \rangle$ and $\left[ \begin{smallmatrix} \overline{f} \\ \overline{g} \end{smallmatrix}\right]$ are restriction inverses of each other by using similar calculations as in the proof of Lemma \ref{lemma:strong-perpd}. So by Lemma \ref{lem:decision-sep-iso}.(\ref{lem:decision-sep-iso.1}) and (\ref{lem:decision-sep-iso.2}), we get that $\langle f \vert g \rangle$ is a separating decision for $f$ and $g$, and thus $f \perp_d g$. So we get that $\perp \subseteq \perp_d$. Conversely, suppose that instead $f \perp_d g$. By assumption we have that $\iota^\circ_1 \perp \iota^\circ_2$. By {\bf [$\boldsymbol{\perp}$.4]}, we get that $\overline{f} = \langle f \vert g \rangle \iota^\circ_1 \perp \langle f \vert g \rangle \iota^\circ_2 = \overline{g}$, and so $\overline{f} \perp \overline{g}$. Then by Lemma \ref{lem:inter-rest}, we get that $f \perp g$, and thus $\perp_d \subseteq \perp$. So we conclude that $\perp = \perp_d$. 
\end{proof}

In particular, we can consider the case when there are disjoint joins for the maximal interference relation. 

\begin{corollary}\label{cor:perp0-ext} Let $\mathbb{X}$ be an extensive restriction category such that $\mathbb{X}$ is a finitely disjoint $\perp_0$-restriction category. Then $\perp_0=\perp_d$, and therefore being decision disjoint is equivalent to being pairwise $\perp_0$-disjoint. 
\end{corollary}
\begin{proof} It is straightforward to check that the quasi-projections are pairwise $\perp_0$-disjoint. Then by applying the above lemma, we get that $\perp_0=\perp_d$. 
\end{proof}

As such, for many important examples of extensive restriction categories (such as the ones considered in \cite{cockett2012timed}), $\perp_d = \perp_0$ and so $\perp_d$-joins are in fact strong, which implies that we can build $n$-ary separating decisions from binary ones. This is also case for an extensive restriction category with all joins. 

\begin{example}\normalfont Since $\mathsf{PAR}$ has all $\perp_0$-joins, it follows that a family of partial functions $\lbrace f_1: X \to Y_1, \hdots, f_n: X \to Y_n \rbrace$ is decision disjoint precisely if they are all pairwise $\perp_0$-disjoint, which recall means that their domains of definitions are all pairwise disjoint, $\mathsf{dom}(f_i) \cap \mathsf{dom}(f_j) = \emptyset$ for $i \neq j$. As such, their separating decision ${\langle f_1 \vert \hdots \vert f_n \rangle: X \to X \sqcup \hdots \sqcup X}$ is defined as follows:  
\[ \langle f_1 \vert \hdots \vert f_n \rangle (x) = \begin{cases} (x, 1) & \text{if } f_1(x) \downarrow \\
\vdots \\
(x,n) & \text{if } f_n(x) \downarrow \\
\uparrow & \text{if } f(x) \uparrow  \end{cases} \]
\end{example}

\begin{example}\normalfont A classical distributive restriction category has all finite joins of compatible maps by definition. So in particular, a classical distributive restriction category has all finite $\perp_0$-joins, and so we get that being decision disjoint is equivalent to being pairwise $\perp_0$-disjoint. 
\end{example}

\section{Matrix Representation for Extensive Restriction Categories}\label{sec:matrix}

In this section, we review the matrix calculus for an extensive restriction category as introduced in \cite{cockett2007restriction}. We will then provide a matrix construction on a finitely disjoint interference restriction category which produces an extensive restriction category. 

In an extensive restriction category, every map of type $A_1 + \hdots + A_n \to B_1 + \hdots + B_m$ can be uniquely represented as an $n\times m$ matrix where the $(i,j)$ coordinate is a map of type $A_i \to B_j$, and such that the rows are decision disjoint \cite[Thm 2.12]{cockett2007restriction}. Explicitly, starting with a map of type $F: A_1 + \hdots + A_n \to B_1 + \hdots + B_m$, the map $F_{i,j}: A_i \to B_j$ is defined by pre-composing $F$ with the injection and post-composing with the quasi-projection:
\begin{align*}
F_{i,j} = \iota_i F \iota^\circ_j
\end{align*}
Moreover, for all $1 \leq i \leq n$, the family $\lbrace F_{i,1}, \hdots, F_{i,m} \rbrace$ is decision disjoint whose separating decision is defined as the decision of $\iota_i F: A_i \to B_1 + \hdots + B_m$, that is:
\begin{align*}
\langle F_{i,1} \vert \hdots \vert F_{i,m} \rangle := \langle \iota_i F \rangle: A_i \to A_i + \hdots + A_i
\end{align*}
Conversely, given a family of maps $\lbrace f_{i,j}: A_i \to B_j\vert~ 1 \leq i \leq n, 1 \leq j \leq m\rbrace$ such that for each $i$, $\lbrace f_{i,1}, \hdots, f_{i,m} \rbrace$ is decision disjoint, define the map $f: A_1 + \hdots + A_n \to B_1 + \hdots + B_m$ as the copairing: 
\begin{align*}
f := \begin{bmatrix} \langle f_{1,1}\vert \hdots \vert f_{1,m} \rangle (f_{1,1} + \hdots + f_{1,m}) \\
\langle f_{2,1}\vert \hdots \vert f_{2,m} \rangle (f_{2,1} + \hdots + f_{2,m}) \\
\vdots \\
\langle f_{n,1}\vert \hdots \vert f_{n,m} \rangle (f_{n,1} + \hdots + f_{n,m})
\end{bmatrix}
\end{align*}
These constructions are inverses of each other, see the proof of \cite[Thm 2.12]{cockett2007restriction} for details. 

Thus, when we have a map of type $f: A_1 + \hdots + A_n \to B_1 + \hdots + B_m$, we can write it as a matrix, 
\[ f = \begin{bmatrix} f_{1,1} & f_{1,2} & \hdots & f_{1,m} \\
f_{2,1} & f_{2,2} & \hdots & f_{2,m} \\
\vdots & \ddots & \hdots & \vdots \\
f_{n,1} & f_{n,2} & \hdots & f_{n,m} 
\end{bmatrix},\]
where the components are maps of type $f_{i,j}: A_i \to B_j$, and the rows $\lbrace f_{i,1}, \hdots, f_{i,m} \rbrace$ (for a fixed $i$) are decision disjoint. By Lemma \ref{lem:n-decision-sep-iso}.(\ref{lem:n-decision-sep-iso.4}), this also means that for each row, the component of a row are pairwise decision disjoint, that is, $f_{i,j} \perp_d f_{i,j^\prime}$ for all $1 \leq i \leq n$ and $1 \leq j,j^\prime \leq m$. Sometimes as a shorthand, we may simply write our matrices as $f = [ f_{i,j} ]_{\substack{1 \leq i \leq n \\ 1 \leq j \leq m}}$, or simply $f = [f_{i,j}]$ when there is no confusion over the bounds of $i$ and $j$. 

Matrices of size $n \times 1$ are column vectors: they correspond precisely to copairing, while matrices of size $1 \times m$ correspond precisely to a family of maps which are decision disjoint. On the other hand, composition in an extensive restriction category corresponds to matrix multiplication, where the disjoint join plays the role of the sum \cite[Prop 2.13]{cockett2007restriction}. So given maps $f: A_1 + \hdots + A_n \to B_1 + \hdots + B_m$ and $g: B_1 + \hdots + B_m \to C_1 + \hdots + C_p$, the matrix representation of their composite $fg: A_1 + \hdots + A_n \to C_1 + \hdots + C_p$ is given by:
\begin{align}
(fg)_{i,k} = \bigsqcup\limits^m_{j=1} f_{i,j}g_{j,k} && fg = \left[ \bigsqcup\limits^m_{j=1} f_{i,j}g_{j,k} \right]_{\substack{1 \leq i \leq n \\ 1 \leq k \leq p}}
\end{align}
Note that this is well-defined since for each $i$ and $k$, $\lbrace f_{i,j}g_{j,k} \vert 1 \leq j \leq m \rbrace$ are decision disjoint with separating decision:
\begin{align}
\langle f_{i,1}g_{1,k} \vert \hdots \vert f_{i,m}g_{m,k} \rangle := \left \langle \langle f_{i,1} \vert \hdots \vert f_{i,m} \rangle \left(f_{i,1}g_{1,k} + \hdots + f_{i,m}g_{m,k}\right) \right\rangle
\end{align}
Keen eyed readers may note that this formula is not exactly the same one given in \cite[Prop 2.13]{cockett2007restriction} where $(fg)_{i,k}$ is instead given by:
\begin{align}
\langle f_{i,1} \vert \hdots \vert f_{i,m} \rangle \left(f_{i,1}g_{1,k} + \hdots + f_{i,m}g_{m,k}\right) \nabla 
\end{align}
However, by \textbf{[D.2]} it follows that:
\begin{align*}  \langle f_{i,1} \vert \hdots \vert f_{i,m} \rangle \left(f_{i,1}g_{1,k} + \hdots + f_{i,m}g_{m,k}\right) = \langle f_{i,1}g_{1,k} \vert \hdots \vert f_{i,m}g_{m,k} \rangle \left(f_{i,1}g_{1,k} + \hdots + f_{i,m}g_{m,k}\right) 
\end{align*}
and thus the two formulas are the same. 

With this matrix representation, decisions correspond precisely to row matrices of restriction idempotents. So if $d: A \to A + \hdots + A$ is a decision, then it is a row matrix of the form: 
\[d = \begin{bmatrix} e_1 &  \hdots & e_n \end{bmatrix}\]
where each $e_i: A \to A$ is a restriction idempotent. 

We conclude this section by showing that every finitely disjoint interference restriction category embeds into an extensive restriction category via a matrix construction. This construction is essentially the same as one given in \cite[Sec 9]{cockett2009boolean}, which was given for the maximal interference relation $\perp_0$: here we slightly generalize the construction to be for an arbitrary interference relation. 

\begin{definition}\label{def:matrix-const} Let $\mathbb{X}$ be a finitely disjoint $\perp$-restriction category. Define the extensive restriction category $\mathsf{MAT}\left[ (\mathbb{X},\perp) \right]$ as follows: 
\begin{enumerate}[{\em (i)}]
\item The objects of $\mathsf{MAT}\left[ (\mathbb{X},\perp) \right]$ are finite lists $(A_1, \hdots, A_n)$ of objects $A_i$ of $\mathbb{X}$, and including the empty list $()$; 
\item A map $F: (A_1, \hdots, A_n) \to (B_1, \hdots, B_n)$ in $\mathsf{MAT}\left[ (\mathbb{X},\perp) \right]$ is a $n\times m$ matrix:
\[ [ F_{i,j} ]_{\substack{1 \leq i \leq n \\ 1 \leq j \leq m}} \]
whose components are maps of type $F_{i,j}: A_i \to B_j$ and such that for each $1 \leq i \leq n$, each row $\lbrace F_{i,1}, \hdots, F_{i,m} \rbrace$ is a $\perp$-disjoint family, that is, $F_{i,j} \perp F_{i,j^\prime}$ for all $1 \leq i \leq n$ and $1 \leq j,j^\prime \leq m$.
\item The identity is the diagonal matrix
\[1_{(A_1, \hdots, A_n)} = Diag[1_{A_1}, \hdots, 1_{A_n}]\] 
that is, ${1_{(A_1, \hdots, A_n)}}_{i,j} = 0$ if $i \neq j$ and ${1_{(A_1, \hdots, A_n)}}_{i,i} = 1_{A_i}$;
\item Composition of $F: (A_1, \hdots, A_n) \to (B_1, \hdots, B_n)$ and $G: (B_1, \hdots, B_n) \to (C_1, \hdots, C_p)$ is defined as:
\[FG = \left[ \bigsqcup\limits^m_{j=1} F_{i,j}G_{j,k} \right]_{\substack{1 \leq i \leq n \\ 1 \leq k \leq p}}: (A_1, \hdots, A_n) \to (C_1, \hdots, C_p)\]
\item The restriction of $F: (A_1, \hdots, A_n) \to (B_1, \hdots, B_m)$ is defined as the diagonal matrix:
\[\overline{F} = Diag\left[\bigsqcup\limits^m_{j=1}\overline{F_{1,j}}, \hdots, \bigsqcup\limits^m_{j=1}\overline{F_{n,j}}\right]: (A_1, \hdots, A_n) \to (A_1, \hdots, A_n)\] 
that is, $\overline{F}_{i,j} = 0$ if $i\neq j$ and $\overline{F}_{i,i} = \bigsqcup\limits^m_{j=1}\overline{F_{i,j}}$;
\item The (binary) restriction coproduct is given by the concatenation of lists, 
\[(A_{_(1,1)}, \hdots, A_{(1,n_1)}) + \hdots + (A_{(m,1)}, \hdots, A_{(m,n_m)}) = (A_{_(1,1)}, \hdots, A_{(1,n_1)}, \hdots, A_{(m,1)}, \hdots, A_{(m,n_m)})\] 
and where the injections:
\[I_j: (A_{(j,1)}, \hdots, A_{(j,n_j)}) \to (A_{_(1,1)}, \hdots, A_{(1,n_1)}) + \hdots + (A_{(m,1)}, \hdots, A_{(m,n_m)})\] 
is defined as ${I_j}_{i,k} = 0$ if $k \neq i+j$ and ${I_j}_{i, i+j}= 1_{A_{(j,i)}}$;
\item The restriction zero maps $0: (A_1, \hdots, A_n) \to (B_1, \hdots, B_m)$ are matrices with all zero components, $0_{i,j} =0$, and where the zero object is the empty list $()$. 
\end{enumerate}
\end{definition}

Since we want to give an embedding, we also need to explain what we mean by a functor between (finitely) disjoint interference restriction categories. If $\mathbb{X}$ and $\mathbb{Y}$ are (finitely) disjoint $\perp$-interference restriction categories, then a \textbf{(finitely) disjoint $\perp$-interference restriction functor} is a $\perp$-restriction functor $\mathcal{F}: (\mathbb{X},\perp) \to (\mathbb{Y},\perp)$ which preserves the (finite) $\perp$-joins, $\mathcal{F}(\bigsqcup\limits_{i \in I} f_i) = \bigsqcup\limits_{i \in I} \mathcal{F}(f_i)$. 

\begin{proposition}\label{prop:construction2} For a finitely disjoint $\perp$-restriction category $\mathbb{X}$, $\mathsf{MAT}\left[ (\mathbb{X},\perp) \right]$ is an extensive restriction category, and furthermore: 
\begin{enumerate}[{\em (i)}]
\item All finite $\perp_d$-joins are strong, so $\mathsf{MAT}\left[ (\mathbb{X},\perp) \right]$ is also a finitely disjoint $\perp_d$-restriction category.
\item\label{mat.1} $F \perp_d G$ if and only if $F_{i,j} \perp G_{i,k}$ for all $i,j$ and $k$. 
\item If $\lbrace F_1, \hdots, F_n \rbrace$ are pairwise $\perp_d$-disjoint parallel $n \times m$ matrices, then $\bigsqcup\limits^n_{k=1} F_k = \left[\bigsqcup\limits^n_{k=1} {F_k}_{i,j} \right]_{\substack{1 \leq i \leq n \\ 1 \leq j \leq m}}$. 
\end{enumerate}
 Furthermore, $\mathcal{I}: \mathbb{X} \to \mathsf{MAT}\left[ (\mathbb{X},\perp) \right]$, defined on objects as $\mathcal{I}(A) = (A)$ and on maps $\mathcal{I}(f) = [\mathcal{I}(f)_{1,1}] = [f]$ (that is, $\mathcal{I}(f)$ is a $1 \times 1$ whose only coefficient is $\mathcal{I}(f)_{1,1} = f$), is a finitely disjoint $\perp$-restriction functor. 
\end{proposition}
\begin{proof} Let us first explain why this construction is well defined. Clearly, a row in either an identity matrix or an injection matrix is indeed $\perp$-disjoint by Lemma \ref{lem:interference1}.(\ref{lem:interference1.0}), which gives us $1 \perp 0$ and $0 \perp 0$. Now given two composable matrices (omitting bounds for readability) $F = [F_{i,j}]$ and $G = [G_{j,k}]$, since $F_{i,j} \perp F_{i,j^\prime}$, by {\bf [{$\boldsymbol{\perp}$}.4]} we get that $F_{i,j}G_{j,k} \perp F_{i,j^\prime}G_{j^\prime,k}$. Then applying {\bf [$\boldsymbol{\sqcup}$.4]}, we get that $\bigsqcup\limits^m_{j=1} F_{i,j}G_{j,k} \perp \bigsqcup\limits^m_{j=1} F_{i,j}G_{j,k^\prime}$, and so a row of $FG$ is $\perp$-disjoint. Then by same arguments as in proofs of \cite[Prop 9.6 and Prop 9.9]{cockett2009boolean}, we get that $\mathsf{MAT}\left[ (\mathbb{X},\perp) \right]$ is a coCartesian restriction category with restriction zeroes. 

Next we need to explain why $\mathsf{MAT}\left[ (\mathbb{X},\perp) \right]$ is also extensive. However by \cite[Prop 2.18]{cockett2007restriction}, it suffices to show that we have binary decisions. So for a matrix $F: (A_1, \hdots, A_n) \to (B_1, \hdots, B_m) + (C_1, \hdots, C_p)$, define $\langle F \rangle: (A_1, \hdots, A_n) \to (A_1, \hdots, A_n) + (A_1, \hdots, A_n)$ as the matrix $\langle F \rangle_{i,j} = 0$ for $j\neq i$ or $i+n$, $\langle F \rangle_{i,i} = \bigsqcup\limits^m_{k=1} \overline{F_{i,k}}$, and $\langle F \rangle_{i,i+n} = \bigsqcup\limits^p_{k^\prime=1} \overline{F_{i,m+k^\prime}}$. This is well-defined since each row of $F$ is a family $\perp$-disjoint maps, and so by {\bf [{$\boldsymbol{\perp}$}.5]}, we get that $\lbrace \overline{F_{i,1}}, \hdots, \overline{F_{i,m}} \rbrace$ and $\lbrace \overline{F_{i,m+1}}, \hdots, \overline{F_{i,m+p}} \rbrace$ are also families of parallel $\perp$-disjoint maps, so we may their $\perp$-joins. Moreover, by {\bf [$\boldsymbol{\sqcup}$.4]} we also get that $\bigsqcup\limits^m_{k=1} \overline{F_{i,k}} \perp \bigsqcup\limits^p_{k^\prime=1} \overline{F_{i,m+k^\prime}}$, and therefore it follows that each row of $\langle F \rangle$ is also $\perp$-disjoint. Thus $\langle F \rangle$ is indeed a map in $\mathsf{MAT}\left[ (\mathbb{X},\perp) \right]$. Next we show that this is a decision for $F$. 

\begin{enumerate}[{\bf [D.1]}]
\item First note that $\nabla: (A_1, \hdots, A_n) + (A_1, \hdots, A_n) \to (A_1, \hdots, A_n)$ is given by $\nabla_{i,j} = 0$ if $j\neq i$ or $n+i$, and $\nabla_{i,i} = 1_{A_i}$ and $\nabla_{i,n+i} = 1_{A_i}$. Then it straightforward to compute that $\left( \langle F \rangle \nabla \right)_{i,j} = 0$ if $i \neq j$ and $\left( \langle F \rangle \nabla \right)_{i,i} = \bigsqcup\limits^m_{k=1} \overline{F_{i,k}} \sqcup \bigsqcup\limits^p_{k^\prime=1} \overline{F_{i,m+k^\prime}} = \bigsqcup\limits^{m+p}_{k=1} \overline{F_{i,k}} = \overline{F}_{i,i}$. So we get that $\langle F \rangle \nabla = \overline{F}$. 

\item Abusing notation slightly, since $F$ is matrix, we may already write $F$ as a block matrix $F = \left[ \begin{smallmatrix} F_1 & F_2 \end{smallmatrix}\right]$ where $F_1: (A_1, \hdots, A_n) \to (B_1, \hdots, B_m)$ is an $n \times m$ matrix given by ${F_1}_{i,j} = F_{i,j}$, and $F_2: (A_1, \hdots, A_n) \to (C_1, \hdots, C_p)$ is an $n \times p$ matrix given by ${F_2}_{i,j} = F_{i, m+j}$. Then it follows that $\langle F \rangle_{i,i} = \overline{F_1}_{i,i}$ and $\langle F \rangle_{i,n+i} = \overline{F_2}_{i,i}$. As such, $\langle F \rangle$ can be given as the block matrix $\langle F \rangle = \begin{bmatrix} \overline{F_1} & \overline{F_2} \end{bmatrix}$. Moreover, we also get that $\overline{F_1} F = \begin{bmatrix} F_1 & 0 \end{bmatrix}$ and $\overline{F_2} F = \begin{bmatrix} 0 & F_2 \end{bmatrix}$. Similarly, we may also write $F + F$ as the block matrix $F + F = \left[ \begin{smallmatrix} F & 0 \\
0 & F\end{smallmatrix}\right]$. Then by matrix multiplication, we get that: 
\[ \langle F \rangle (F + F) = \begin{bmatrix} \overline{F_1} & \overline{F_2} \end{bmatrix} \begin{bmatrix} F & 0 \\
0 & F\end{bmatrix} = \begin{bmatrix} \begin{bmatrix} F_1 & 0 \end{bmatrix} & \begin{bmatrix} 0 & F_2 \end{bmatrix} \end{bmatrix} \]
On the other hand, similarly we may write $\iota_1 + \iota_2 = \left[ \begin{smallmatrix} \iota_1 & 0 \\
0 & \iota_2 \end{smallmatrix}\right]$, and also that $F_1 \iota_1 = \begin{bmatrix} F_1 & 0 \end{bmatrix}$ and $F_2 \iota_2 = \begin{bmatrix} 0 & F_2 \end{bmatrix}$. As such, we get that: 
\[ F (\iota_1 + \iota_2) = \begin{bmatrix} F_1 & F_2 \end{bmatrix} \begin{bmatrix} \iota_1 & 0 \\
0 & \iota_2 \end{bmatrix} = \begin{bmatrix} \begin{bmatrix} F_1 & 0 \end{bmatrix} & \begin{bmatrix} 0 & F_2 \end{bmatrix} \end{bmatrix} \]
So get that $\langle F \rangle (F + F) = F (\iota_1 + \iota_2)$.
\end{enumerate}
Thus $\langle F \rangle$ is indeed the decision of $F$, which allows us to conclude that $\mathsf{MAT}\left[ (\mathbb{X},\perp) \right]$ is an extensive restriction category. 

Since we have an extensive restriction category, by Thm \ref{thm:ext-int} we get that $ \mathsf{MAT}\left[ (\mathbb{X},\perp) \right]$ is a $\perp_d$-restriction category with all binary $\perp_d$-joins. We now show that these binary $\perp_d$-joins are in fact strong. To do so, let us first work out the binary $\perp_d$-joins, for which we need to consider the separating decision. So let us now prove (\ref{mat.1}). 

So suppose that $F \perp_d G$, then there is a separating decision $\langle F \vert G \rangle$. Since $\langle F \vert G \rangle$ is a separating decision, we get that $\langle F \vert G \rangle_{i,i}= \overline{F}_{i,i}$ and $\langle F \vert G \rangle_{i,n+i}= \overline{G}_{i,i}$. However since each row of $\langle F \vert G \rangle$ is $\perp$-disjoint, we get that $\overline{F}_{i,i} \perp \overline{G}_{i,i}$. Then by definition of $\overline{F}_{i,i}$ and $\overline{G}_{i,i}$, from Lemma \ref{lemma:disjoint-join}.(\ref{lemma:join-inter}), it follows that $\overline{F_{i,j}} \perp \overline{G_{i,k}}$, and then by Lemma \ref{lem:inter-rest} we get $F_{i,j} \perp G_{i,k}$. Conversely, if $F_{i,j} \perp G_{i,k}$ for all $i,j$ and $k$, then we can build the separating decision $\langle F \vert G \rangle$ as $\langle F \vert G \rangle_{i,j}$ if $j\neq i$ or $n+i$ and $\langle F \vert G \rangle_{i,i}= \overline{F}_{i,i}$ and $\langle F \vert G \rangle_{i,n+i}= \overline{G}_{i,i}$. By similarly arguments as above, it follows that $\langle F \vert G \rangle$ is a separating decision, so $F \perp_d G$. Moreover, it may be worth noting that as a block matrix, we get $\langle F \vert G \rangle = \begin{bmatrix} \overline{F} & \overline{G} \end{bmatrix}$. Now suppose that $F$ and $G$ are parallel $\perp_d$-disjoint maps. Then since $F_{i,j} \perp G_{i,j}$ for all $i$ and $j$, we can take their $\perp$-join $F_{i,j} \sqcup G_{i,j}$. Then using (\ref{eq:sqcup-dec}), $F \sqcup G = \langle F \vert G \rangle(F+G)\nabla$, or by the using the same arguments as in the proof of \cite[Prop 9.9]{cockett2009boolean}, we get that $(F \sqcup G)_{i,j} = F_{i,j} \sqcup G_{i,j}$. Now also suppose that $H \perp_d F$ and $G \perp_d F$, which implies that $H_{i,j} \perp F_{i,j}$ and $H_{i,j} \perp G_{i,j}$. Then by {\bf [{$\boldsymbol{\sqcup}$}.4]}, we get that $H_{i,j} \perp F_{i,j} \sqcup G_{i,j}$, so therefore by (\ref{mat.1}) we have that $H \perp_d F \sqcup G$. Thus binary $\perp_d$-joins strong in $\mathsf{MAT}\left[ (\mathbb{X},\perp) \right]$, and so by Lemma \ref{lemma:binary-finite}, $\mathsf{MAT}\left[ (\mathbb{X},\perp) \right]$ is a finitely disjoint $\perp_d$-restriction category. Moreover, it then follows that $\bigsqcup\limits^n_{k=1} F_k = \left[\bigsqcup\limits^n_{k=1} {F_k}_{i,j} \right]_{\substack{1 \leq i \leq n \\ 1 \leq j \leq m}}$ as desired. 

Lastly, by the same arguments as in the proof of \cite[Prop 9.10]{cockett2009boolean}, we also get that $\mathcal{I}$ preserve the disjoint interference restriction structure. 
\end{proof}

If we apply this to the maximal interference relation $\perp_0$, we get precisely the matrix construction from \cite[Sec 9]{cockett2009boolean}. Moreover, by Cor \ref{cor:perp0-ext}, we get that the induced interference relation on the matrix construction is precisely the maximal one as well. 

\begin{corollary} Let $\mathbb{X}$ be a restriction category with restriction zeroes. Then $\mathsf{MAT}\left[ (\mathbb{X},\perp_0) \right]$ is an extensive restriction category and moreover its separating decision $\perp_d$ coincides with $\perp_0$. 
\end{corollary}

Lastly, we also note that if the base category has all disjoint joins, then so does the matrix category. 

\begin{lemma}\label{lemma:mat-inf} For a disjoint $\perp$-restriction category $\mathbb{X}$, $ \mathsf{MAT}\left[ (\mathbb{X},\perp) \right]$ is a disjoint $\perp_d$-category, where if $\lbrace F_k \vert~ k \in K \rbrace$ is a family of pairwise $\perp_d$-disjoint parallel $n \times m$ matrices, then $\bigsqcup\limits_{k \in I} F_k = \left[\bigsqcup\limits_{k \in K} {F_k}_{i,j} \right]_{\substack{1 \leq i \leq n \\ 1 \leq j \leq m}}$. 
\end{lemma}
\begin{proof} This follows by a similar arguments to the finite disjoint join case above. 
\end{proof}

\section{Iteration and Trace}\label{sec:trace}

In this section we show that for an extensive restriction category, to give a trace operator is equivalent to giving a Kleene wand, which is one of the main results of this paper. We begin by first reviewing trace operators on coproducts, and how to give a trace operator is equivalent to giving a \emph{parametrized iteration operator}. This perspective will be very useful since, as we will see below, for an extensive restriction category, a parametrized iteration operator is essentially a Kleene wand.

Briefly recall that a \textbf{traced symmetric monoidal category} is a symmetric\footnote{It is worth noting that in \cite{joyal_street_verity_1996}, Joyal, Street, and Verity first introduced traced monoidal categories as \emph{braided} monoidal categories equipped with a trace operator. However since we are dealing with coproducts, we are working in the symmetric setting.} monoidal category with monoidal product with a trace operator. Since every category with finite coproducts is a symmetric monoidal category \cite[Sec 6.2]{selinger2010survey}, where the monoidal product is the coproduct and the monoidal unit is the initial object, one can indeed consider a trace operator for coproducts. As such, the following definition below of a trace operator is the same as for a symmetric monoidal category, but where we have already specified our monoidal structure to be that given by coproducts. There are many equivalent lists of axioms for a trace operator on symmetric monoidal category. Here, we will use the (strict version) of the one presented in \cite[Def 2.2]{Hasegawa2023tracedmonadshopf}. For other equivalent axiomatizations of the trace operator, see for example \cite[Def 4.1]{haghverdi2010geometry} or \cite[Def 2.1]{Hasegawa97recursionfrom} or \cite[Sec 3]{hasegawa2009traced}. For a more in-depth introduction to traced symmetric monoidal categories, including their graphical calculus, we invite the reader to see \cite{abramsky2002geometry,haghverdi2000categorical, haghverdi2010geometry, hasegawa2009traced, Hasegawa97recursionfrom, hasegawa2004uniformity, Hasegawa2023tracedmonadshopf, selinger2010survey, joyal_street_verity_1996}. 

\begin{definition} A \textbf{traced coCartesian monoidal category} \cite[Sec 6.4]{selinger2010survey} is a category $\mathbb{X}$ with finite coproducts which comes equipped with a \textbf{trace operator} $\mathsf{Tr}$ \cite[Def 2.2]{Hasegawa2023tracedmonadshopf} which is a family of operators (indexed by triples of objects $X, A, B \in \mathbb{X}$):
\begin{align*}
\mathsf{Tr}^X_{A,B}: \mathbb{X}(X + A, X + B) \to \mathbb{X}(A,B) && \infer{\mathsf{Tr}^X_{A,B}(f) : A \to B}{f: X + A \to X + B}
\end{align*}
such that the following axioms hold: 
\begin{description}
\item[\textbf{[Tightening]:}] For every map $f: X + A \to X + B$, $g: A^\prime \to A$, and $h: B \to B^\prime$ the following equality holds: 
 \begin{equation}\label{tightening1}\begin{gathered}\mathsf{Tr}^X_{A^\prime,B^\prime}\left( (1_X + g) f (1_X + h)  \right) = g\mathsf{Tr}^X_{A,B}(f)h 
   \end{gathered}\end  {equation}
\item[\textbf{[Sliding]:}] For every map $f: X + A \to X^\prime + B$ and $k: X^\prime \to X$, the following equality holds: 
 \begin{equation}\label{sliding}\begin{gathered} \mathsf{Tr}^X_{A,B}\left(f  (k+1_B) \right) = \mathsf{Tr}^{X^\prime}_{A,B}\left((k + 1_A)   f \right)
 \end{gathered}\end  {equation}
\item[\textbf{[Vanishing]:}] For every map $f: X + Y + A \to X + Y + B$, the following equality holds: 
 \begin{equation}\label{vanishing1}\begin{gathered}  \mathsf{Tr}^{X + Y}_{A,B}(f) =     \mathsf{Tr}^Y_{A,B} \left( \mathsf{Tr}^X_{Y + A,Y + B}\left( f \right) \right) 
  \end{gathered}\end  {equation}
\item[\textbf{[Superposing]:}] For every map $f: X + A \to X + B$ and every map $g: C \to D$ the following equality holds:
 \begin{equation}\label{superposing}\begin{gathered} \mathsf{Tr}^X_{A + C,B + D}\left(  f + g \right) =  \mathsf{Tr}^X_{A,B}(f) + g
  \end{gathered}\end{equation}
\item[\textbf{[Yanking]:}] For every object $X$, the following equality holds: 
\begin{equation}\label{yanking}\begin{gathered} \mathsf{Tr}^X_{X,X}( \sigma ) = 1_X
  \end{gathered}\end  {equation}
  where recall that $\sigma: X + X \to X + X$ is the canonical symmetry isomorphism for the coproduct. 
  \end{description}
For a map $f: X + A \to X + B$, the map $\mathsf{Tr}^X_{A,B}(f): A \to B$ is called the \textbf{trace} of $f$. 
\end{definition}

In \cite{abramsky1996retracing}, Abramsky referred to trace operators on coproducts as ``particle style traces'', as the trace can intuitively be interpreted by streams of particles or tokens flowing around a network -- see also \cite[Sec 1.1]{abramsky2002geometry} and \cite[Sec 4.1]{haghverdi2010geometry} for more explination about this intuition. Another desirable property one might ask of the trace operator is that it be \emph{uniform}. 

\begin{definition} For a traced coCartesian monoidal category $\mathbb{X}$, its trace operator $\mathsf{Tr}$ is said to be \textbf{uniform} \cite[Def 2.2]{hasegawa2004uniformity} if the following holds:
\begin{description}
\item[\textbf{[Uniform]:}]  For all $f: X + A \to X + B$, $f^\prime: Y + A \to Y + B$, and $h: X \to Y$ such that $f(h + 1_B) = (h + 1_A)f^\prime$, then $\mathsf{Tr}^X_{A,B}(f) = \mathsf{Tr}^Y_{A,B}(f^\prime)$. 
  \end{description}
\end{definition}

As explained by Selinger in \cite[Sec 3.2]{selinger1999categorical}, uniformity is a proof principle for showing that two traces are equal. We note that often times uniformity is only with respect to a certain subclass of maps, called \emph{strict} maps \cite[Def 2.1]{hasegawa2004uniformity}. However in this paper, we consider uniformity to hold for all maps, or in other words, that all maps are strict. 

As mentioned above, trace operators on coproducts admit an alternative characterization via another kind of operator. Indeed, consider a map of type $f: X + A \to X + B$. By the couniversal property of the coproduct, this is actually a copairing of maps $f_1: X \to X+B$ and $f_2: A \to X+B$, so $f = \left[ \begin{smallmatrix} f_1 \\ f_2
 \end{smallmatrix}\right]$. Using \textbf{[Tightening]} and \textbf{[Sliding]}, it turns out that the following equality holds: 
 \[ \mathsf{Tr}^X_{A,B}(f) = f_2 \begin{bmatrix} \mathsf{Tr}^X_{X,B}(\nabla f_1) \\ 1_B  \end{bmatrix}  \]
where note that $\mathsf{Tr}^X_{X,B}(\nabla f_1): X \to B$. Therefore, we see that the traced out part of a map of type $X + A \to X + B$ is essentially determined by its $X \to X + B$ component. As such, this gives rise to an operator which takes in a map of type $X \to X + B$ and produces a map of type $X \to B$, called a \emph{parametrized iteration operation}.  

\begin{definition} For a category $\mathbb{X}$ with finite coproducts, a \textbf{parametrized iteration operator} $\mathsf{Iter}$ \cite[Prop 6.8]{selinger2010survey} is a family of operators (indexed by pairs of objects $X, A \in \mathbb{X}$):
\begin{align*}
\mathsf{Iter}^X_{A}: \mathbb{X}(X, X + A) \to \mathbb{X}(X,A) && \infer{\mathsf{Iter}^X_{A}(f) : X \to A}{f: X \to X + A}
\end{align*}
such that the following axioms hold: 
\begin{description} 
\item[\textbf{[Parametrized Iteration]}] For every map $f: X \to X + A$ the following equality holds: 
 \begin{equation}\label{iteration}\begin{gathered}
 f  \begin{bmatrix} \mathsf{Iter}^X_A(f) \\ 1_A  
 \end{bmatrix}  = \mathsf{Iter}^X_A(f) 
   \end{gathered}\end  {equation}
   \item[\textbf{[Naturality]}] For every map $f: X \to X + A^\prime$ and every map $h: A^\prime \to A$ the following equality holds: 
 \begin{equation}\label{iternat1}\begin{gathered}
 \mathsf{Iter}^X_A\left( f  (1_X + h) \right) =  \mathsf{Iter}^X_{A^\prime}\left( f \right) h
   \end{gathered}\end  {equation}
  \item[\textbf{[Dinaturality]}] For every map $f: X^\prime \to X + A$ and every map $k: X \to X^\prime$ the following equality holds:
    \begin{equation}\label{iternat2}\begin{gathered}
 \mathsf{Iter}^X_A\left( k f \right) =  k \mathsf{Iter}^{X^\prime}_{A}\left( f  (k + 1_A) \right) 
   \end{gathered}\end  {equation}
   \item[\textbf{[Diagonal Property]}] For every map $f: X \to X + X + A$ the following equality holds: 
    \begin{equation}\label{iterdiag}\begin{gathered}
 \mathsf{Iter}^X_A\left( \mathsf{Iter}^X_{X+A}(f) \right) =   \mathsf{Iter}^X_A\left(f (\nabla_X + 1_A)  \right)
   \end{gathered}\end  {equation}   
\end{description}
For a map $f:X \to X + A$, the map $\mathsf{Iter}^X_{A}(f): X \to A$ is called the \textbf{iteration} of $f$. 
\end{definition}

As the name suggests, a parametrized iteration operator can intuitively be interpreted as iterating the map ${f: X \to X + A}$ if it continues to land in $X$ and only stopping when it lands in $A$. This intuition is best captured by our main example of partial functions given below. As with trace operators, a desirable property for a parametrized iteration operator is to be \emph{uniform}. 

\begin{definition} For a category $\mathbb{X}$ with finite coproducts, a parametrized iteration operator $\mathsf{Iter}$ is said to be \textbf{uniform} \cite[Sec 4.3]{hasegawa2004uniformity} if the following holds: 
\begin{description}
\item[\textbf{[Uniform]:}] For all $f: X \to X + A$, $f^\prime: Y \to Y + A$, and $h: X \to Y$ such that $f(h + 1_A) = hf^\prime$, then $\mathsf{Iter}^X_{A}(f) = h \mathsf{Iter}^Y_{A}(f^\prime)$. 
  \end{description}
  \end{definition}

We now review how having a trace operator is equivalent to having a parametrized iteration operator and, in particular, how to build one from the other. 

\begin{proposition}\cite[Prop 6.8]{selinger2010survey} For a category with finite coproducts, there is a bijective correspondence between (uniform) trace operators and (uniform) parametrized iteration operators. Explicitly: 
\begin{enumerate}[{\em (i)}]
\item Let $\mathsf{Iter}$ be a (uniform) parametrized iteration operator, then, for a map $f: X + A \to X + B$, define its trace ${\mathsf{Tr}^X_{A,B}(f): A \to B}$ as 
\begin{align} \mathsf{Tr}^X_{A,B}(f) := \iota_2 f \begin{bmatrix} \mathsf{Iter}^X_B(\iota_1  f) \\ 1_B
 \end{bmatrix} 
\end{align}
then $\mathsf{Tr}$ is a (uniform) trace operator. 
\item Let $\mathsf{Tr}$ be a (uniform) trace operator, then, for a map $f: X \to X + B$, define its iteration $\mathsf{Iter}^X_{B}(f): X \to B$ as
\begin{align}
\mathsf{Iter}^X_{B}(f) := \mathsf{Tr}^X_{X,B}\left(\begin{bmatrix} f \\ f
 \end{bmatrix} \right) = \mathsf{Tr}^X_{X,B}\left( \nabla_X  f \right)  
 \end{align}
then $\mathsf{Iter}$ is a (uniform) parametrized iteration operator. 
\end{enumerate}
Furthermore, these constructions are inverses of each other. In other words, a coCartesian traced monoidal category is precisely a category finite coproducts equipped with a parametrized iteration operator. 
\end{proposition}

Our main example of a coCartesian traced monoidal category is the category of sets and partial functions. Other examples of coCartesian traced monoidal categories can be found in \cite[Ex 4.6]{haghverdi2010geometry}. In particular, every partially additive category is a coCartesian traced monoidal category \cite[Prop 3.1.4]{haghverdi2000categorical}. 

\begin{example}\label{ex:PAR-trace} \normalfont $\mathsf{PAR}$ is a traced coCartesian monoidal category where the trace operator and parametrized iteration operator are induced by the natural feedback operator. So for a partial function $f: X \to X \sqcup Y$, its iteration ${\mathsf{Iter}^X_{Y}(f): X \to Y}$ is the partial function defined as follows:   
\begin{align*}
\mathsf{Iter}^X_{Y}(f)(x) = \begin{cases} y & \text{ if } f(x) = (y,2) \\
 y & \text{ if } \exists~ n \in \mathbb{N}, n \geq 1 ~ \exists~ x_0, x_1, \hdots, x_n \in X \text{ s.t. } f(x) = (x_0, 1) \\
 &\text{ and } f(x_0) = (x_1, 1) \text{ and } \hdots \text{ and } f(x_n) = (y,2)  \\
\uparrow  
 & \text{otherwise}\end{cases}    
\end{align*}
For a partial function $g: X \sqcup Y \to X \sqcup Z$, its trace $ \mathsf{Tr}^X_{Y,Z}(g): Y \to Z$ is the partial function defined as follows: 
\begin{align*}
 \mathsf{Tr}^X_{Y,Z}(g)(y) = \begin{cases} z & \text{ if } g(y,2) = (z,2) \\
 z & \text{ if } \exists~ n \in \mathbb{N}, n \geq 1 ~ \exists~x_0, x_1, \hdots, x_n \in X \text{ s.t. } g(y,2) = (x_0,1) \\
 & \text{ and } g(x_0,1) = (x_1,1) \text{ and } \hdots \text{ and } g(x_n,1) = (z,2)  \\
\uparrow  
 & \text{otherwise}\end{cases}    
\end{align*}
Furthermore, $\mathsf{Iter}$ and $\mathsf{Tr}$ are uniform. 
\end{example}

Since restriction coproducts are actual coproducts, one can consider trace operators and parametrized iteration operators for restriction coproducts. In particular, we are interested in these operators for extensive restriction categories. 

\begin{definition}\label{def:traced-ext} A \textbf{traced extensive restriction category} is an extensive restriction category equipped with a trace operator or equivalently a parametrized iteration operator. 
\end{definition}

Our objective is to show that for an extensive restriction category, to give a parametrized iteration operator is equivalent to giving a Kleene wand. To do so, let us first revisit the axioms of a parametrized iteration operator using the matrix calculus for extensive restriction categories. So suppose we are in a traced extensive restriction category with parametrized iteration operator $\mathsf{Iter}$. First observe that in an extensive restriction category, to give a map of type $f: X \to X + A$ is equivalent to giving maps $f_1: X \to X$ and $f_2: X \to A$ such that $f_1 \perp_d f_2$, since $f$ is in fact a $1 \times 2$ matrix $f = \begin{bmatrix} f_1 & f_2
\end{bmatrix}$. So the parametrized iteration operator is really acting on an endomorphism and a map with the same domain that are $\perp_d$-disjoint maps. So we may write: 
\[ \infer{ \mathsf{Iter}^X_A\left( \begin{bmatrix} f_1 & f_2 \end{bmatrix}  \right): X \to A}{f: X \to X ~~~ g: X \to A ~~~ f \perp_d g} \]
So we see that a parametrized iteration operator and a Kleene wand are operators with the same inputs and outputs. This is our first hint that in this setting, a parametrized iteration operator and a Kleene wand are in fact the same thing. It is important to note that in \cite[Sec 4]{cockett2012timed}, they were working in a distributive extensive restriction category with joins, which by Cor \ref{cor:perp0-ext} implies that $\perp_0 = \perp_d$. The objective here however is to show that we still get the desired correspondence for arbitrary extensive restriction categories when taking Kleene wands on $\perp_d$ instead (since in an arbitrary extensive restriction category, $\perp_d$ may not equal $\perp_0$). 

Now let us consider re-expressing the axioms of a parametrized iteration operator using the matrix calculus. To avoid heavy notation, let us introduce the following notation for our parametrized iteration operator:
\begin{align} 
\mathsf{Iter}^X_A\left( \begin{bmatrix} f_1 & f_2 \end{bmatrix} \right) = \left[\!\!\! \begin{array}{c|c} f_1 & f_2  \end{array} \!\!\! \right]
\end{align}

Now starting with \textbf{[Parametrized Iteration]}, first note that by using matrix multiplication we have that: 
\[ \begin{bmatrix} f_1 & f_2 \end{bmatrix} \begin{bmatrix} \left[\!\!\! \begin{array}{c|c} f_1 & f_2  \end{array} \!\!\! \right] \\ 1_A  
 \end{bmatrix} = f_1 \left[\!\!\! \begin{array}{c|c} f_1 & f_2  \end{array} \!\!\! \right] \sqcup f_2  \]
Thus we may re-express the iteration axiom as follows (where we slightly modify the above equality using commutativity of $\sqcup$): 
\begin{description} 
\item[\textbf{[Parametrized Iteration]}] For every map $f_1: X \to X$ and $f_2: X \to A$ such that $f_1 \perp f_2$, the following equality holds: 
 \begin{equation}\label{ext-iteration}\begin{gathered}
f_2 \sqcup f_1 \left[\!\!\! \begin{array}{c|c} f_1 & f_2  \end{array} \!\!\! \right] = \left[\!\!\! \begin{array}{c|c} f_1 & f_2  \end{array} \!\!\! \right]
   \end{gathered}\end  {equation}
\end{description}
Next, for \textbf{[Naturality]} and \textbf{[Dinaturality]}, first note that $1_X + h: X + A^\prime \to X + A$ and $k + 1_A: X + A \to X^\prime + A$ are respectively the following diagonal square matrices:
\begin{align*}
1_X + h = \begin{bmatrix} 1_X & 0 \\ 
0 & h
\end{bmatrix} && k + 1_A = \begin{bmatrix} k & 0 \\ 
0 & 1_A
\end{bmatrix}
\end{align*}
Moreover, by matrix multiplication, it is easy to compute that we have the following equalities: 
\begin{align*}
\begin{bmatrix} f_1 & f_2 \end{bmatrix} \begin{bmatrix} 1_X & 0 \\ 
0 & h
\end{bmatrix} =  \begin{bmatrix} f_1 & f_2h \end{bmatrix} && k \begin{bmatrix} f_1 & f_2 \end{bmatrix} = \begin{bmatrix} kf_1 & kf_2 \end{bmatrix} &&  \begin{bmatrix} f_1 & f_2 \end{bmatrix} \begin{bmatrix} k & 0 \\ 
0 & 1_A
\end{bmatrix} = \begin{bmatrix} f_1k & f_2 \end{bmatrix}
\end{align*}
Thus we may re-express the naturality and dinaturality axioms as follows: 
\begin{description} 
  \item[\textbf{[Naturality]}] For every map $f_1: X \to X$ and $f_2: X \to A^\prime$ such that $f_1 \perp f_2$, and every map $h: A^\prime \to A$, the following equality holds: 
 \begin{equation}\label{ext-iternat1}\begin{gathered}
\left[\!\!\! \begin{array}{c|c} f_1 & f_2h  \end{array} \!\!\! \right] = \left[\!\!\! \begin{array}{c|c} f_1 & f_2  \end{array} \!\!\! \right]h 
   \end{gathered}\end  {equation}
   \item[\textbf{[Dinaturality]}] For every map $f_1: X^\prime \to X$ and $f_2: X^\prime \to A$ such that $f_1 \perp f_2$, and every map $k: X \to X^\prime$, the following equality holds:
    \begin{equation}\label{ext-iternat2}\begin{gathered}
\left[\!\!\! \begin{array}{c|c} kf_1 & kf_2  \end{array} \!\!\! \right] = k \left[\!\!\! \begin{array}{c|c} f_1k & f_2  \end{array} \!\!\! \right]
   \end{gathered}\end  {equation}
\end{description}

Then for the last axiom \textbf{[Diagonal Property]}, first note that a map $f: X \to X + X + A$  is equivalently given by three maps $\lbrace f_1: X \to X, f_2: X \to X, f_3: X \to A\rbrace$ which are decision disjoint. So $f$ is in fact a $1 \times 3$ row matrix, $f = \begin{bmatrix} f_1 & f_2 & f_3 \end{bmatrix}$. Next observe that $\nabla_X + 1_A$ is the $3 \times 2$ matrix:
\[ \nabla_X + 1_A = \begin{bmatrix} 1_X & 0  \\ 
1_X & 0  \\
0 & 1_A\end{bmatrix} \]
and so we get that: 
\[ \begin{bmatrix} f_1 & f_2 & f_3 \end{bmatrix}\begin{bmatrix} 1_X & 0  \\ 
1_X & 0  \\
0 & 1_A\end{bmatrix} = \begin{bmatrix} f_1 \sqcup f_2 & f_3 \end{bmatrix}  \]
Next note that $\begin{bmatrix} f_1 & f_2 & f_3 \end{bmatrix} = \begin{bmatrix} f_1 & \begin{bmatrix} f_2 & f_3 \end{bmatrix} \end{bmatrix}$. So we may take its iteration, which we denote by $\left[\!\!\! \begin{array}{c|cc} f_1 & f_2 & f_3 \end{array} \!\!\! \right] := \left[\!\!\! \begin{array}{c|c} f_1 & \begin{bmatrix} f_2 & f_3 \end{bmatrix}  \end{array} \!\!\! \right]: X \to X+A$. Note that $\left[\!\!\! \begin{array}{c|cc} f_1 & f_2 & f_3 \end{array} \!\!\! \right]$ is a $1 \times 2$ row matrix, and thus fully determined by post-composing it with the quasi-projections $\iota^\circ_1$ and $\iota^\circ_2$. Then by \textbf{[Naturality]}, it follows that $\left[\!\!\! \begin{array}{c|c} f_1 & \begin{bmatrix} f_2 & f_3 \end{bmatrix}  \end{array} \!\!\! \right] \iota^\circ_1 = \left[\!\!\! \begin{array}{c|c} f_1 & f_2  \end{array} \!\!\! \right]$, and similarly $\left[\!\!\! \begin{array}{c|c} f_1 & \begin{bmatrix} f_2 & f_3 \end{bmatrix}  \end{array} \!\!\! \right] \iota^\circ_2 = \left[\!\!\! \begin{array}{c|c} f_1 & f_3  \end{array} \!\!\! \right]$. Therefore, we have that:
\begin{align}\label{ext-iterdiag-3}
\left[\!\!\! \begin{array}{c|c} f_1 & \begin{bmatrix} f_2 & f_3 \end{bmatrix}  \end{array} \!\!\! \right] = \begin{bmatrix} \left[\!\!\! \begin{array}{c|c} f_1 & f_2  \end{array} \!\!\! \right] & \left[\!\!\! \begin{array}{c|c} f_1 & f_3  \end{array} \!\!\! \right] 
\end{bmatrix}
\end{align}
So we may re-express the diagonal axiom as follows: 

\begin{description} 
   \item[\textbf{[Diagonal Property]}] For every $\lbrace f_1: X \to X, f_2: X \to X, f_3: X \to A\rbrace$ which are decision disjoint, the following equality holds: 
  \begin{equation}\label{ext-iterdiag-2}\begin{gathered}
\left[\!\!\! \begin{array}{c|c} \left[\!\!\! \begin{array}{c|c} f_1 & f_2  \end{array} \!\!\! \right] & \left[\!\!\! \begin{array}{c|c} f_1 & f_3  \end{array} \!\!\! \right]   \end{array} \!\!\! \right] = \left[\!\!\! \begin{array}{c|c} f_1 \sqcup f_2 & f_3  \end{array} \!\!\! \right] 
   \end{gathered}\end  {equation}   

We can also re-express uniformity as follows: 
\end{description}
\begin{description}
\item[\textbf{[Uniform]:}] For every map $f_1: X \to X$ and $f_2: X \to A$ such that $f_1 \perp_d f_2$, and $f^\prime_1: Y \to Y$ and $f^\prime_2: Y \to A$ such that $f^\prime_1 \perp f^\prime_2$, and $h: X \to Y$ such that $hf^\prime_1 = f_1h$ and $f_2 = hf^\prime_2$, then the following equality holds: 
\begin{align}\label{ext-uni}
\left[\!\!\! \begin{array}{c|c} f_1  & f_2  \end{array} \!\!\! \right] = h \left[\!\!\! \begin{array}{c|c} f^\prime_1 & f^\prime_2  \end{array} \!\!\! \right] 
\end{align}
  \end{description}

To summarize, we have that:

\begin{lemma}\label{lem:ext-iter} For an extensive restriction category, $\mathsf{Iter}$ is a (uniform) parametrized iteration operator if and only if $\mathsf{Iter}$ satisfies (\ref{ext-iteration}), (\ref{ext-iternat1}), (\ref{ext-iternat2}), and (\ref{ext-iterdiag-2}) (and (\ref{ext-uni})). 
\end{lemma}

Now let us consider the induced trace operator. A map $f: X + A \to X + B$ is in fact a $2 \times 2$ square matrix $f = \left[ \begin{smallmatrix} f_1 & f_2 \\
f_3 & f_4 \end{smallmatrix}\right]$, where $f_1: X \to X$, $f_2: X \to B$, $f_3: A \to X$, and $f_4: A \to B$ such that $f_1 \perp_d f_2$ and $f_3 \perp_d f_4$. As such, our induced trace operator $\mathsf{Tr}$ acts on $2 \times 2$ square matrices. Inspired by our notation for our parametrized iteration operator, we use the following notation to denote our traces:
\begin{align}
\mathsf{Tr}^X_{A,B}\left( \begin{bmatrix} f_1 & f_2 \\
f_3 & f_4 \end{bmatrix}  \right) = \left[\begin{array}{c|c} f_1 & f_2 \\ \hline f_3 & f_4 \end{array}\right] 
\end{align}
Now note that we have that $\iota_1 f = \begin{bmatrix} f_1 & f_2 \end{bmatrix}$ and $\iota_2 f = \begin{bmatrix} f_3 & f_4 \end{bmatrix}$. Therefore, the trace is given as follows: 
  \begin{align}\label{ext-trace}
\left[\begin{array}{c|c} f_1 & f_2 \\ \hline f_3 & f_4 \end{array}\right] = f_4 \sqcup f_3 \left[\!\!\! \begin{array}{c|c} f_1  & f_2  \end{array} \!\!\! \right]
\end{align}
and the trace axioms can be re-expressed in terms of square matrices as follows:
\begin{description}
\item[\textbf{[Tightening]:}] 
 \begin{equation}\label{matrix-tightening1}\begin{gathered}
 \left[\begin{array}{c|c} f_1 & f_2h \\ \hline gf_3 & gf_4 h\end{array}\right] = g \left[\begin{array}{c|c} f_1 & f_2 \\ \hline f_3 & f_4 \end{array}\right] h
   \end{gathered}\end  {equation}
\item[\textbf{[Sliding]:}] 
 \begin{equation}\label{matrix-sliding}\begin{gathered} 
\left[\begin{array}{c|c} f_1k & f_2 \\ \hline f_3k & f_4 \end{array}\right] = \left[\begin{array}{c|c} kf_1 & kf_2 \\ \hline f_3 & f_4 \end{array}\right] 
 \end{gathered}\end  {equation}
\item[\textbf{[Vanishing]:}] 
 \begin{equation}\label{matrix-vanishing1}\begin{gathered} 
\left[\begin{array}{cc|c} f_1 & f_2 & f_3 \\ f_4 & f_5 & f_6 \\ \hline f_7 & f_8 & f_9 \end{array}\right] = \left[\begin{array}{c|c} \left[\begin{array}{c|c} f_1 & f_2 \\ \hline f_4 & f_5 \end{array}\right]  & \left[\begin{array}{c|c} f_1 & f_3 \\ \hline f_4 & f_6 \end{array}\right] \\   \\ \hline \\ \left[\begin{array}{c|c} f_1 & f_2 \\ \hline f_7 & f_8 \end{array}\right]  & \left[\begin{array}{c|c} f_1 & f_3 \\ \hline f_7 & f_9 \end{array}\right]  \end{array}\right] 
  \end{gathered}\end{equation}
\item[\textbf{[Superposing]:}] 
 \begin{equation}\label{matrix-superposing}\begin{gathered} 
\left[\begin{array}{cc|c} f_1 & f_2 & 0 \\ f_3 & f_4 & 0 \\ \hline 0 & 0 & g \end{array}\right] = \begin{bmatrix} \left[\begin{array}{c|c} f_1 & f_2 \\ \hline f_3 & f_4 \end{array}\right] & 0 \\
0 & g \end{bmatrix}
  \end{gathered}\end{equation}
\item[\textbf{[Yanking]:}] 
\begin{equation}\label{matrix-yanking}\begin{gathered} 
\left[\begin{array}{c|c} 0 & 1_X \\ \hline 1_X & 0 \end{array}\right] = 1_X
  \end{gathered}\end  {equation}

\item[\textbf{[Uniform]:}] 
\begin{align}
\begin{bmatrix} f_1h & f_2 \\
f_3h & f_4 \end{bmatrix} = \begin{bmatrix} hf^\prime_1 & hf^\prime_2 \\
f^\prime_3 & f^\prime_4 \end{bmatrix} && \Longrightarrow && \left[\begin{array}{c|c} f_1 & f_2 \\ \hline f_3 & f_4 \end{array}\right] = \left[\begin{array}{c|c} f^\prime_1 & f^\prime_2 \\ \hline f^\prime_3 & f^\prime _4 \end{array}\right]
\end{align}
  \end{description}
  
On the other hand, we can recapture the parametrized iteration operator from the trace as follows: 
\begin{align}
\left[\!\!\! \begin{array}{c|c} f_1  & f_2  \end{array} \!\!\! \right] = \left[\begin{array}{c|c} f_1 & f_2 \\ \hline f_1 & f_2 \end{array}\right] 
\end{align}

We are now finally ready to show the equivalence between parametrized iteration operators and Kleene wands. We begin by showing how a Kleene wand gives a parametrized iteration operator, and thus also a trace operator. 

\begin{proposition}\label{prop:wand-to-iter} Let $\mathbb{X}$ be an extensive restriction category which comes equipped with a Kleene wand $\wand$ for $\perp_d$ -- that is, $(\mathbb{X}, \wand)$ is a $\perp_d$-itegory -- then $\mathbb{X}$ has a parametrized iteration operator $\mathsf{Iter}$ defined as follows on a map ${f = \begin{bmatrix} f_1 & f_2 \end{bmatrix}: X \to X + A}$: 
\begin{align}\label{def:wand-Iter}
\left[\!\!\! \begin{array}{c|c} f_1  & f_2  \end{array} \!\!\! \right] = f_1 \wand f_2
\end{align}
and the induced trace operator $\mathsf{Tr}$ is defined as follows on a map $g = \begin{bmatrix} g_1 & g_2 \\
g_3 & g_4 \end{bmatrix}: X + A \to X + B$: 
\begin{align}\label{ext-trace-wand}
\left[\begin{array}{c|c} g_1 & g_2 \\ \hline g_3 & g_4 \end{array}\right]  = g_4 \sqcup g_3(g_1 \wand g_2)  
\end{align}
So $\mathbb{X}$ is a traced extensive restriction category. Moreover, if $\wand$ is uniform, then $\mathsf{Iter}$ and $\mathsf{Tr}$ are uniform. 
\end{proposition}
\begin{proof} First note that our parametrized iteration operator is indeed well-defined since given a map ${f: X \to X + A}$, there are unique maps $f_1: X \to X$ and $f_2: X \to A$ with $f_1 \perp_d f_2$ such that $f=\begin{bmatrix} f_1 & f_2 \end{bmatrix}$. So $f_1 \wand f_2$ is indeed well-defined. We now show that $\mathsf{Iter}$ satisfies the four axioms of a parametrized iteration operator. By Lemma \ref{lem:ext-iter}, we can prove their equivalent versions in an extensive restriction category. 
\begin{description} 
\item[\textbf{[Parametrized Iteration]}] Using {\bf [$\boldsymbol{\wand}$.1]}, we compute that:
\[ f_2 \sqcup f_1 \left[\!\!\! \begin{array}{c|c} f_1  & f_2  \end{array} \!\!\! \right]  = f_2 \sqcup f_1 \left(f_1 \wand f_2 \right) = f_1 \wand f_2 = \left[\!\!\! \begin{array}{c|c} f_1  & f_2  \end{array} \!\!\! \right] \]
   \item[\textbf{[Naturality]}] Using {\bf [$\boldsymbol{\wand}$.2]}, we compute that:
\[  \left[\!\!\! \begin{array}{c|c} f_1  & f_2 h  \end{array} \!\!\! \right]  = f_1 \wand f_2 h = (f_1 \wand f_2) h = \left[\!\!\! \begin{array}{c|c} f_1  & f_2  \end{array} \!\!\! \right]h \]
  \item[\textbf{[Dinaturality]}] Using {\bf [$\boldsymbol{\wand}$.3]}, we compute that: 
\[ \left[\!\!\! \begin{array}{c|c} kf_1  & kf_2  \end{array} \!\!\! \right] = kf_1 \wand k f_2 = k(f_1k \wand f_2) = k \left[\!\!\! \begin{array}{c|c} f_1k  & f_2  \end{array} \!\!\! \right] \]
   \item[\textbf{[Diagonal Property]}] Using {\bf [$\boldsymbol{\wand}$.4]}, we compute that: 
\begin{gather*}
\left[\!\!\! \begin{array}{c|c} f_1 \sqcup f_2  & f_3  \end{array} \!\!\! \right] = (f_1 \sqcup f_2) \wand f_3 = (f_1 \wand f_2) \wand (f_1 \wand f_3) = \left[\!\!\! \begin{array}{c|c} f_1 \wand f_2  & f_1 \wand f_3  \end{array} \!\!\! \right] = \left[\!\!\! \begin{array}{c|c} \left[\!\!\! \begin{array}{c|c} f_1  & f_2 \end{array} \!\!\! \right]  & \left[\!\!\! \begin{array}{c|c} f_1  & f_3 \end{array} \!\!\! \right]   \end{array} \!\!\! \right]
\end{gather*} 
\end{description}
Thus $\mathsf{Iter}$ is a parametrized iteration operator as desired. Moreover, from (\ref{ext-trace}), we get that (\ref{ext-trace-wand}) is indeed the formula for the induced trace operator $\mathsf{Tr}$. Lastly, suppose that $\wand$ is also uniform. Then we show that $\mathsf{Iter}$ is also uniform:
\begin{description}
\item[\textbf{[Uniform]}] Suppose that $f_1 \perp_d f_2$ and $f_1^\prime \perp_d f^\prime_2$, and that we have a map $h$ such that $hf^\prime_1 = f_1h$ and $f_2 = hf^\prime_2$. Since $\wand$ is uniform, we get that $h (f^\prime_1 \wand f^\prime_2) = f_1 \wand f_2$, or in other words, $ h  \left[\!\!\! \begin{array}{c|c} f^\prime_1  & f^\prime_2  \end{array} \!\!\! \right] = \left[\!\!\! \begin{array}{c|c} f_1  & f_2  \end{array} \!\!\! \right]$. 
  \end{description}
  Since $\mathsf{Iter}$ is uniform, we also get that $\mathsf{Tr}$ is uniform. 
\end{proof}

Conversely, we now show how a parametrized iteration operator gives a Kleene wand. 

\begin{proposition}\label{prop:iter-to-wand} Let $\mathbb{X}$ be a traced extensive restriction category with parametrized iteration operator, $\mathsf{Iter}$, and an induced trace operator $\mathsf{Tr}$, then $(\mathbb{X}, \wand)$ is a $\perp_d$-itegory, where the Kleene wand $\wand$ is defined for $\perp_d$-disjoint maps $f: X \to X$ and $g: X \to A$ as follows: 
\begin{align}\label{eq:iter-to-wand}
f \wand g :=  \left[\!\!\! \begin{array}{c|c} f  & g  \end{array} \!\!\! \right] = \left[\begin{array}{c|c} f & g \\ \hline f & g \end{array}\right]
\end{align}
Thus, $(\mathbb{X}, \wand)$ is a $\perp_d$-itegory. Moreover, if $\mathsf{Iter}$ is uniform, then $\wand$ is uniform. 
\end{proposition}
\begin{proof} Since $f \perp_d g$, we can indeed build the $1 \times 2$ matrix $\begin{bmatrix} f & g \end{bmatrix}: X \to X + A$. So we can take its iteration, and therefore our proposed Kleene wand is well-defined. We now need to prove the four axioms of a Kleene wand. 
    \begin{enumerate}[{\bf [$\boldsymbol{\wand}$.1]}]
\item Here we use \textbf{[Parametrized Iteration]} to compute that: 
\[  g \sqcup f(f \wand g) = g \sqcup f \left[\!\!\! \begin{array}{c|c} f  & g  \end{array} \!\!\! \right] = \left[\!\!\! \begin{array}{c|c} f  & g  \end{array} \!\!\! \right] = f \wand g \]
\item Here we use \textbf{[Naturality]} to compute that: 
\[  \left( f \wand g \right)h = \left[\!\!\! \begin{array}{c|c} f  & g  \end{array} \!\!\! \right] h = \left[\!\!\! \begin{array}{c|c} f  & g h \end{array} \!\!\! \right]= f \wand gh   \]
\item Here we use \textbf{[Dinaturality]} to compute that: 
\[ hf \wand hg =  \left[\!\!\! \begin{array}{c|c} hf  & hg  \end{array} \!\!\! \right] = h \left[\!\!\! \begin{array}{c|c} hf  & g  \end{array} \!\!\! \right] = h \left( fh \wand g \right) \]
\item If $f \sqcup f^\prime \perp_d g$, we will first show that $\lbrace f, f^\prime, g \rbrace$ are decision disjoint. So define $\langle f \vert f^\prime \vert g \rangle = \langle f \sqcup f^\prime \vert g \rangle (\langle f \vert f^\prime \rangle + 1_X)$. To show that this is the separating decision of $\lbrace f, f^\prime, g \rbrace$, by Lemma \ref{lem:n-decision-sep-iso}.(\ref{lem:n-decision-sep-iso.1}), we need to show that is the restriction inverse of $\left[ \begin{smallmatrix} \overline{f} \\ \overline{f^\prime} \\ \overline{g}  \end{smallmatrix}\right]$. However recall that by Lemma \ref{lem:decision-sep-iso}.(\ref{lem:decision-sep-iso.1}), $\langle f \sqcup f^\prime \vert g \rangle$ is the restriction inverse of $\left[ \begin{smallmatrix} \overline{f \sqcup f^\prime} \\ \overline{g} \end{smallmatrix}\right]$ and $\langle f \vert f^\prime \rangle$ is the restriction inverse of $\left[ \begin{smallmatrix} \overline{f} \\ \overline{f^\prime}  \end{smallmatrix}\right]$. So by construction, $\langle f \vert f^\prime \vert g \rangle$ is the restriction inverse of $\left(\begin{bmatrix} \overline{f} \\ \overline{f^\prime}  \end{bmatrix} + 1_X \right)\left[ \begin{smallmatrix}\overline{f \sqcup f^\prime} \\ \overline{g} \end{smallmatrix}\right]$, since it is the composition and coproduct of restriction isomorphisms. However observe that by \textbf{[R.3]} and Lemma \ref{lemma:disjoint-join}.(\ref{lemma:join-<}), we get that $\left[ \begin{smallmatrix} \overline{f} \\ \overline{f^\prime}  \end{smallmatrix}\right]\overline{f \sqcup f^\prime} = \left[ \begin{smallmatrix} \overline{f} \\ \overline{f^\prime}  \end{smallmatrix}\right]$. Then we have that $\left(\begin{bmatrix} \overline{f} \\ \overline{f^\prime}  \end{bmatrix} + 1_X \right) \begin{bmatrix} \overline{f \sqcup f^\prime} \\ \overline{g} \end{bmatrix} = \left[ \begin{smallmatrix}\overline{f} \\ \overline{f^\prime} \\ \overline{g}  \end{smallmatrix}\right]$. So $\langle f \vert f^\prime \vert g \rangle$ is the restriction inverse of $\left[ \begin{smallmatrix} \overline{f} \\ \overline{f^\prime} \\ \overline{g}  \end{smallmatrix}\right]$, and therefore is also the separating decision which make $\lbrace f, f^\prime, g \rbrace$ decision disjoint. Since $\lbrace f, f^\prime, g \rbrace$ are decision disjoint, we can build the row matrix $\begin{bmatrix} f & f^\prime & g \end{bmatrix}$. We first need to explain why $f \wand g$ and $f \wand f^\prime$ are $\perp_d$-disjoint. However note that from (\ref{ext-iterdiag-3}), we get that $\left[\!\!\! \begin{array}{c|c} f  & g  \end{array} \!\!\! \right] \perp_d \left[\!\!\! \begin{array}{c|c} f  & f^\prime  \end{array} \!\!\! \right]$, or in other words, $f \wand g \perp_d f \wand f^\prime$ as desired. Finally, using \textbf{[Diagonal Property]}, we compute that: 
\begin{gather*} (f \sqcup f^\prime) \wand g =  \left[\!\!\! \begin{array}{c|c} f \sqcup f^\prime  & g  \end{array} \!\!\! \right] = \left[\!\!\! \begin{array}{c|c} \left[\!\!\! \begin{array}{c|c} f  & f^\prime  \end{array} \!\!\! \right]  & \left[\!\!\! \begin{array}{c|c} f  & g  \end{array} \!\!\! \right]  \end{array} \!\!\! \right] = \left[\!\!\! \begin{array}{c|c} f \wand f^\prime  & f \wand g  \end{array} \!\!\! \right] = (f \wand f^\prime) \wand (f \wand g) 
\end{gather*}
\end{enumerate}
So we conclude that $\wand$ is a Kleene wand. Now suppose that $\mathsf{Iter}$ is also uniform. Given $f \perp_d g$ and $f^\prime \perp_d g^\prime$, and a map $h$ such that $hg^\prime = g$ and $hf^\prime= fh$, then $h \left( f^\prime \wand g^\prime \right)=f \wand g$. Then by \textbf{[Uniform]} we get that $f \wand g = \left[\!\!\! \begin{array}{c|c} f  & g  \end{array} \!\!\! \right] = h  \left[\!\!\! \begin{array}{c|c} f^\prime  & g^\prime  \end{array} \!\!\! \right]=h \left( f^\prime \wand g^\prime \right)$. So $\wand$ is also uniform, as desired. 
\end{proof}

It is immediate that the constructions in Prop \ref{prop:wand-to-iter} and Prop \ref{prop:iter-to-wand} are inverses of each other. Thus we obtain the main result of this paper: 

\begin{theorem}\label{thm:ext-wand-trace} For an extensive restriction category, there is a bijective correspondence between parametrized iteration operators (equivalently trace operators) and Kleene wands for $\perp_d$. Therefore, a traced extensive restriction category is equivalently an extensive restriction category which is an itegory with respect to $\perp_d$. 
\end{theorem}

In the case that we have infinite disjoint joins, we can use the canonical inductive Kleene wand to get a parametrized iteration operator and trace operator. 

\begin{proposition}\label{prop:ext-inf} Let $\mathbb{X}$ be an extensive restriction category which is also a disjoint $\perp_d$-restriction category, then $\mathbb{X}$ is a traced extensive restriction category with parametrized iteration operator $\mathsf{Iter}$ and trace operator $\mathsf{Tr}$ defined respectively as follows: 
\begin{align}\label{eq:inf-trace}
\left[\!\!\! \begin{array}{c|c} f_1  & f_2  \end{array} \!\!\! \right] = \bigsqcup\limits^\infty_{n=0} f_1^n f_2 && \left[\begin{array}{c|c} g_1 & g_2 \\ \hline g_3 & g_4 \end{array}\right] = g_4 \sqcup \bigsqcup\limits^\infty_{n=0} g_3 g_1^n g_2
\end{align}
Moreover, $\mathsf{Iter}$ and $\mathsf{Tr}$ are uniform. 
\end{proposition}
\begin{proof} These are the induced parametrized iteration operator and trace operator from the Kleene wand from Thm \ref{thm:countable-wand}. Since this an inductive Kleene wand, it is uniform, and therefore the parametrized iteration operator and trace operator are both uniform as well. 
\end{proof}

Moreover, we also note that it is not difficult to see that an extensive restriction category which is also a disjoint $\perp_d$-restriction category is in fact a partially additive category where the partial sum operation $\Sigma$ is given by the $\perp_d$-joins $\bigsqcup$. Moreover in this case, the canonical trace formula for a partially additive category as in \cite[Prop 3.1.4]{haghverdi2000categorical} is precisely the one given in (\ref{eq:inf-trace}) above. 

We can also consider the classical setting, where we can use the Kleene upper-star to express the parametrized iteration operator and trace operator. 

\begin{proposition}\label{prop:trace-star} Let $\mathbb{X}$ be an extensive restriction category which is also a classical $\perp_d$-restriction category equipped with a Kleene upper-star $\star$, then $\mathbb{X}$ is a traced extensive restriction category with a parametrized iteration operator $\mathsf{Iter}$ and trace operator $\mathsf{Tr}$ defined respectively by: 
 \begin{align}\label{eq:star-trace}
\left[\!\!\! \begin{array}{c|c} f_1  & f_2  \end{array} \!\!\! \right] = f_1^\star f_2 && \left[\begin{array}{c|c} g_1 & g_2 \\ \hline g_3 & g_4 \end{array}\right]  = g_4 \sqcup g_3 g_1^\star g_2 
\end{align}
\end{proposition}
\begin{proof} These are the induced parametrized iteration operator and trace operator from the induced Kleene wand from a Kleene upper-star as given in Prop \ref{prop:star-to-wand}.
\end{proof}

Either of these above formulas recapture the canonical trace and iteration of partial functions. 

\begin{example}\normalfont For $\mathsf{PAR}$, the parametrized iteration operator and trace operator given in Ex \ref{ex:PAR-trace} are precisely the ones obtain from the Kleene wand in Ex \ref{ex:PAR-wand}. Moreover, they are also given by the formulas in (\ref{eq:inf-trace}) and (\ref{eq:star-trace}). 
\end{example}

We conclude by explaining how applying the matrix construction to an itegory gives a traced extensive restriction category. Once one  has matrices, it is much easier to use the trace operator rather than the Kleene wand or the parametrized iteration operator. To understand the formula for the trace in the matrix construction, we first need to understand what the trace of an $n + m \times n +p$ matrix is in a traced extensive restriction category. This amounts to understanding the trace formulas for matrices of size $1 + m \times 1 + p$ and $n+1 \times n+1$ respectively. We will explain how in either case, it reduces to computing traces of $2 \times 2$ matrices. 

So let us start with a map $F: X + A_1 + \hdots + A_m \to X + B_1 + \hdots + B_p$, which is an $1+m \times 1+p$ matrix $F = [F_{i,j}]_{\substack{1 \leq i \leq 1+m \\ 1 \leq j \leq 1+p}}$, where $F_{1,1}: X \to X$, $F_{1,j+1}: X \to B_j$, and $F_{i+1,1}: A_k \to X$, and $F_{i+1,j+1}: A_i \to B_j$. As a short hand let us denote $\vec{A} =  A_1 + \hdots + A_m$ and $\vec{B} = B_1 + \hdots + B_p$. Then its trace $\mathsf{Tr}^X_{\vec{A}, \vec{B}}(F): A_1 + \hdots + A_m \to B_1 + \hdots + B_p$ is a $m \times p$ matrix whose components are given by $\iota_i \mathsf{Tr}^X_{\vec{A}, \vec{B}}(F) \iota^\circ_j: A_i \to B_j$. However by \textbf{[Tightening]} we get that $\iota_i \mathsf{Tr}^X_{\vec{A}, \vec{B}}(F) \iota^\circ_j = \mathsf{Tr}^X_{A_i, B_j}\left( (1_X + \iota_i) F (1_X + \iota^\circ_j) \right)$. So expanding this out, we get that: 
\[ \mathsf{Tr}^X_{\vec{A}, \vec{B}}(F)_{i,j} = \left[\begin{array}{c|c} F_{1,1} & F_{1,j+1} \\ \hline F_{i+1,1} & F_{i+1,j+1} \end{array}\right] \]
So as a matrix, the trace of $F$ is written as: 
 \begin{equation}\begin{gathered}\label{eq:matrix-trace1}
\mathsf{Tr}^X_{\vec{A}, \vec{B}}(F) \! = \! \left[\begin{array}{c|ccc} 
F_{1,1} & F_{1,2}  & \hdots & F_{1,p+1} \\
\hline F_{2,1} & F_{2,2} & \hdots & F_{2,p+1} \\
\vdots & \hdots & \ddots& \vdots \\ 
F_{m+1,1} & F_{m+1,2}& \hdots & F_{m+1,p+1} \end{array}\right]  \!= \!\begin{bmatrix} \left[\begin{array}{c|c} F_{1,1} & F_{1,2} \\ \hline F_{2,1} & F_{2,2} \end{array}\right] & \hdots & \left[\begin{array}{c|c} F_{1,1} & F_{1,p+1} \\ \hline F_{2,1} & F_{2,p+1} \end{array}\right] \\
\vdots & \ddots & \vdots \\
\left[\begin{array}{c|c} F_{1,1} & F_{1,2} \\ \hline F_{m+1,1} & F_{m+1,p+1} \end{array}\right] & \hdots & \left[\begin{array}{c|c} f_{1,1} & F_{1,p+1} \\ \hline F_{m+1,1} & F_{m+1,p+1} \end{array}\right]
\end{bmatrix} \end{gathered}\end{equation}
So the trace of an $1 + m \times 1 + p$ matrix is the $m \times p$ matrix whose components are traces of $2 \times 2$ matrices. As such, computing the trace of an $1 + m \times 1 + p$ matrix is reduced to computing traces of $2 \times 2$ matrices.  

Another way to see this is to write our map $X + A_1 + \hdots + A_m \to X + B_1 + \hdots + B_p$ as a block matrix $\left[ \begin{smallmatrix} f & P \\ Q & R \end{smallmatrix}\right]$, where $f: X \to X$ is an endomorphism, $P=\begin{bmatrix} P_1 & \hdots & P_p \end{bmatrix}: X \to B_1 + \hdots + B_p$ is a $1 \times p$ row matrix, $Q=\left[ \begin{smallmatrix} Q_1 \\ \vdots \\ Q_m \end{smallmatrix}\right]: A_1 + \hdots + A_m \to X$ a $m \times 1$ column matrix, and $R= [R_{i,j}]: A_1 +  \hdots + A_m \to B_1 + \hdots + B_p$ is an $m \times p$ matrix. Then the traced matrix $\left[\begin{array}{c|c} f & P \\ \hline Q & R \end{array}\right]$ is given by:
\begin{align}\label{eq:matrix-trace2}
\left[\begin{array}{c|c} f & P \\ \hline Q & R \end{array}\right]_{i,j} = \left[\begin{array}{c|c} f & P_i \\ \hline Q_j & R_{i,j} \end{array}\right] && \left[\begin{array}{c|c} f & P \\ \hline Q & R \end{array}\right] = \begin{bmatrix} \left[\begin{array}{c|c} f & P_1 \\ \hline Q_1 & R_{1,1} \end{array}\right] & \hdots & \left[\begin{array}{c|c} f & P_p \\ \hline Q_1 & R(1,p) \end{array}\right] \\
\vdots & \ddots & \vdots \\
\left[\begin{array}{c|c} f & P_1 \\ \hline Q_m & R(m,1) \end{array}\right] & \hdots & \left[\begin{array}{c|c} f & P_p \\ \hline Q_m & R(m,p) \end{array}\right]
\end{bmatrix} 
\end{align}

On the other hand, consider now a map of type $X_1 + \hdots + X_n + A \to X_1 + \hdots + X_n + B$, which is a $n + 1 \times n+1$ matrix. Then successive applications of \textbf{[Vanishing]} tells us that tracing out $X_1 + \hdots + X_n$ is done by first tracing out $X_1$, then tracing out $X_2$, and so on (in fact \textbf{[Yanking]} also tells us that we can do it in any order). Then we may write down the higher order \textbf{[Vanishing]} inductively on $n$ as follows. Consider a map $X_1 + \hdots + X_{n+1} + A \to X_1 + \hdots + X_{n+1} + B$, which is an $n+2 \times n+2$ matrix, which we write down as a block matrix $\left[ \begin{smallmatrix} f & T & U \\ V  & S & g \\ W & h &k \end{smallmatrix}\right]$, where $S: X_2 + \hdots + X_{n+1} \to X_2 + \hdots + X_{n+1}$ is an $n\times n$ matrix, $T: X_2 + \hdots + X_n \to X_{n+1}$ and $U: X_2 + \hdots + X_{n+1} \to B$ are $n \times 1$ column matrices, $V: X_{n+1} \to  X_1 + \hdots + X_n$ and $W: A \to  X_2 + \hdots + X_{n+1}$ are $1 \times n$ row matrices, and the rest are maps $f: X_{1} \to X_{1}$, $g: X_{1} \to B$, $h: A \to X_{1}$, and $k: A \to B$. Then by \textbf{[Vanishing]} we have that:
\begin{align}\label{eq:matrix-trace3}
\left[\begin{array}{cc|c} f & T & g \\ V & S & U \\ \hline h & W & k \end{array}\right] = \left[\begin{array}{c|c} \left[\begin{array}{c|c} f & T \\ \hline V & S \end{array}\right]  & \left[\begin{array}{c|c} f & g \\ \hline V & U \end{array}\right] \\   \\ \hline \\ \left[\begin{array}{c|c} f & T \\ \hline h & W \end{array}\right]  & \left[\begin{array}{c|c} f & g \\ \hline h & k \end{array}\right]  \end{array}\right] 
\end{align}
where on the left-hand side, the top left, top right, and bottom left inputs of the traced matrix can be reduced further using (\ref{eq:matrix-trace2}) and applying \textbf{[Vanishing]} again. Thus we see that the trace of an $n+1 \times n+1$ matrix is given by taking the trace of $2 \times 2$ matrices whose components are themselves $2 \times 2$ matrices, and so on. So again we have that computing the trace of a $n+1 \times n+1$ reduces to computing the trace of $2 \times 2$ matrices.  

Finally, combining (\ref{eq:matrix-trace2}) and (\ref{eq:matrix-trace3}) together tells us that computing the trace of an $n + m \times n + p$ matrix can be reduced to computing traces of $2 \times 2$ matrices. Therefore, for the matrix construction, it suffices to give a formula for $2 \times 2$ matrices and then extend it to $n + m \times n + p$ matrices inductively using (\ref{eq:matrix-trace2}) and (\ref{eq:matrix-trace3}).

\begin{proposition}\label{prop:matrix-ext-wand} Let $\mathbb{X}$ be a finitely disjoint $\perp$-restriction category with a Kleene wand, so $(\mathbb{X}, \wand)$ is a $\perp$-itegory, then $\mathsf{MAT}\left[ (\mathbb{X},\perp) \right]$ is a traced extensive restriction category, where the trace operator, $\mathsf{Tr}$, is defined on a matrix\footnote{Recall that in $\mathsf{MAT}\left[ (\mathbb{X},\perp) \right]$, the coproduct is given by concatenation of lists.} $F: (X,A) = (X) + (A) \to (X) + (B) = (X,B)$ as follows: 
\begin{align}
    \mathsf{Tr}^{(X)}_{(A),(B)}(F) = \left[\begin{array}{c|c} f_{1,1} & f_{1,2} \\ \hline f_{2,1} & f_{2,2} \end{array}\right] = \left[ f_{2,2} \sqcup f_{2,1} \left( f_{1,1} \wand f_{1,2} \right) \right]
\end{align}
which we then extend to matrices of type $(X_1, \hdots, X_n) + (A_1, \hdots, A_m) \to (X_1, \hdots, X_n) + (B_1, \hdots, B_m)$ using (\ref{eq:matrix-trace2}) and (\ref{eq:matrix-trace3}). Furthermore, ${\mathcal{I}}: \mathbb{X} \to \mathsf{MAT}\left[ (\mathbb{X},\perp) \right]$ preserves the Kleene wands, that is, ${\mathcal{I}(f \wand g)}= \mathcal{I}(f) \wand \mathcal{I}(g)$. 
\end{proposition}
\begin{proof} Prop \ref{prop:wand-to-iter} tells us that the proposed formula for the trace operator satisfies all the necessary axioms for $2 \times 2$ matrices. It is then not difficult to see (though very tedious to write down) that when extending to $n + m \times n + p$ matrices using (\ref{eq:matrix-trace2}) and (\ref{eq:matrix-trace3}), the trace operator axioms all hold by definition and reducing it to the $2 \times 2$ case. Thus we get that $\mathsf{MAT}\left[ (\mathbb{X},\perp) \right]$ is indeed a traced extensive restriction category. Then by Prop \ref{prop:iter-to-wand}, we get a Kleene wand on $\mathsf{MAT}\left[ (\mathbb{X},\perp) \right]$ which can be expressed using the trace as in (\ref{eq:iter-to-wand}). Then using {\bf [$\boldsymbol{\wand}$.1]}, we compute that: 
\[ \mathcal{I}(f) \wand \mathcal{I}(g) = \left[\begin{array}{c|c} \mathcal{I}(f) & \mathcal{I}(g) \\ \hline \mathcal{I}(f) & \mathcal{I}(g) \end{array}\right] = \left[\begin{array}{c|c} [f] & [g] \\ \hline [f] & [g] \end{array}\right] = \left[\begin{array}{c|c} f & g \\ \hline f & g  \end{array}\right] = \left[ g \sqcup f(f \wand g) \right] = \left[ f \wand g \right] = \mathcal{I}(f \wand g)
\]
So $\mathcal{I}$ preserves the Kleene wand as desired.
\end{proof}

If the base interference restriction category has infinite disjoint joins, then we can give more explicit formulas for the Kleene wand and its induced uniform parametrized iteration operator and uniform trace operator. 

\begin{corollary} Let $\mathbb{X}$ be a disjoint $\perp$-restriction category, then $\mathsf{MAT}\left[ (\mathbb{X},\perp) \right]$ is a traced extensive restriction category with inductive Kleene wand given as in Thm \ref{thm:countable-wand}, and induced uniform parametrized iteration operator and uniform trace operator defined as in Prop \ref{prop:ext-inf}. 
\end{corollary}
\begin{proof} By Lemma \ref{lemma:mat-inf}, $\mathsf{MAT}\left[ (\mathbb{X},\perp) \right]$ is also a strongly completely disjoint interference restriction category. Thus by Thm \ref{thm:countable-wand}, we get an inductive Kleene wand, whose induced parametrized iteration operator and trace operator are given as in Prop \ref{prop:ext-inf}. 
\end{proof}

Taking a step back, we can first use the disjoint join construction, and then apply the matrix construction to get a traced extensive restriction category. 

\begin{corollary} Let $\mathbb{X}$ be a $\perp$-restriction category, then $\mathsf{MAT}\left[ \mathsf{DJ}\left[ (\mathbb{X},\perp) \right] \right]$ is a traced extensive restriction category with inductive Kleene wand given as in Thm \ref{thm:countable-wand}, and induced uniform parametrized iteration operator and uniform trace operator defined as in Prop \ref{prop:ext-inf}. 
\end{corollary}

Of course, we can simply start with any restriction category with restriction zeroes to build a traced extensive restriction category. 

\begin{corollary} Let $\mathbb{X}$ be a restriction category with restriction zeroes, then $\mathsf{MAT}\left[ \left( \mathsf{DJ}\left[ (\mathbb{X},\perp_0) \right], \perp_0 \right) \right]$ is a traced extensive restriction category with inductive Kleene wand given as in Thm \ref{thm:countable-wand}, and induced uniform parametrized iteration operator and uniform trace operator defined as in Prop \ref{prop:ext-inf}. 
\end{corollary}

\bibliographystyle{plain}      
\bibliography{itegoriesbib}   

\begin{thebibliography}{10}

\bibitem{abramsky1996retracing}
S.~Abramsky.
\newblock Retracing some paths in process algebra.
\newblock In {\em International Conference on Concurrency Theory}, pages 1--17,
  Berlin, Heidelberg, 1996. Springer.

\bibitem{abramsky2009abstract}
S.~Abramsky and B.~Coecke.
\newblock Abstract physical traces.
\newblock {\em Theory and Applications of Categories}, 14(6):111--124, 2005.

\bibitem{abramsky2002geometry}
S.~Abramsky, E.~Haghverdi, and P.~Scott.
\newblock {Geometry of Interaction and Linear Combinatory Algebras}.
\newblock {\em Mathematical Structures in Computer Science}, 12(5):625--665,
  2002.

\bibitem{adamek2006iterative}
J.~Ad{\'a}mek, S.~Milius, and J.~Velebil.
\newblock Iterative algebras at work.
\newblock {\em Mathematical Structures in Computer Science}, 16(6):1085--1131,
  2006.

\bibitem{adamek2010equational}
J.~Ad{\'a}mek, S.~Milius, and J.~Velebil.
\newblock Equational properties of iterative monads.
\newblock {\em Information and Computation}, 208(12):1306--1348, 2010.

\bibitem{adamek2011elgot}
J.~Ad{\'a}mek, S.~Milius, and J.~Velebil.
\newblock Elgot theories: a new perspective on the equational properties of
  iteration.
\newblock {\em Mathematical Structures in Computer Science}, 21(2):417--480,
  2011.

\bibitem{bloom1993iteration}
S.~Bloom and Z.~{\'E}sik.
\newblock {\em Iteration theories}.
\newblock Springer, 1993.

\bibitem{carboni1993introduction}
A.~Carboni, S.~Lack, and R.~Walters.
\newblock Introduction to extensive and distributive categories.
\newblock {\em Journal of Pure and Applied Algebra}, 84(2):145--158, 1993.

\bibitem{cuazuanescu1990towards}
V.~C{\u{a}}z{\u{a}}nescu and G.~{\c{S}}tef{\u{a}}nescu.
\newblock Towards a new algebraic foundation of flowchart scheme theory.
\newblock {\em Fundamenta Informaticae}, 13(2):171--210, 1990.

\bibitem{cuazuanescu1994feedback}
V.~C{\u{a}}z{\u{a}}nescu and G.~{\c{S}}tef{\u{a}}nescu.
\newblock Feedback, iteration, and repetition.
\newblock In {\em Mathematical aspects of natural and formal languages}, pages
  43--61. World Scientific, 1994.

\bibitem{cuazuanescu1982again}
V.~C{\u{a}}z{\u{a}}nescu and C.~Ungureanu.
\newblock Again on advice on structuring compilers and proving them correct.
\newblock {\em INCREST Preprint Series in Mathematics}, (75), 1982.

\bibitem{cockett2012differential}
R.~Cockett, G.~Cruttwell, and J.~Gallagher.
\newblock Differential restriction categories.
\newblock {\em Theory and Applications of Categories}, 25(21):537--613, 2011.

\bibitem{cockett2012timed}
R.~Cockett, J.~D{\'\i}az-Bo{\"\i}ls, J.~Gallagher, and P.~Hrube{\v{s}}.
\newblock Timed sets, functional complexity, and computability.
\newblock {\em Electronic Notes in Theoretical Computer Science}, 286:117--137,
  2012.

\bibitem{cockett2014restriction}
R.~Cockett and R.~Garner.
\newblock Restriction categories as enriched categories.
\newblock {\em Theoretical Computer Science}, 523:37--55, 2014.

\bibitem{cockett2012range}
R.~Cockett, X.~Guo, and P.~Hofstra.
\newblock {Range categories I: General theory}.
\newblock {\em Theory and Applications of Categories}, 26(17):412--452, 2012.

\bibitem{turing-categories}
R.~Cockett and P.~Hofstra.
\newblock {Introduction to Turing Categories}.
\newblock {\em Annals of pure and applied logic}, 156(2-3):183--209, 2008.

\bibitem{cockett2002restriction}
R.~Cockett and S.~Lack.
\newblock Restriction categories i: categories of partial maps.
\newblock {\em Theoretical computer science}, 270(1-2):223--259, 2002.

\bibitem{cockett2003restriction}
R.~Cockett and S.~Lack.
\newblock Restriction categories ii: Partial map classification.
\newblock {\em Theoretical Computer Science}, 294(1-2):61--102, 2003.

\bibitem{cockett2007restriction}
R.~Cockett and S.~Lack.
\newblock Restriction categories iii: Colimits, partial limits and extensivity.
\newblock {\em Mathematical Structures in Computer Science}, 17(4):775--817,
  2007.

\bibitem{cockett2023classical}
R.~Cockett and J.-S.~P. Lemay.
\newblock Classical distributive restriction categories.
\newblock {\em Theory and Applications of Categories}, 42(6):102--144, 2024.

\bibitem{cockett2009boolean}
R.~Cockett and E.~Manes.
\newblock Boolean and classical restriction categories.
\newblock {\em Mathematical Structures in Computer Science}, 19(2):357--416,
  2009.

\bibitem{girard1989towards}
J.-Y. Girard.
\newblock Towards a geometry of interaction.
\newblock {\em Contemporary Mathematics}, 92(69-108):6, 1989.

\bibitem{goncharov2016complete}
S.~Goncharov, S.~Milius, and C.~Rauch.
\newblock Complete elgot monads and coalgebraic resumptions.
\newblock {\em Electronic Notes in Theoretical Computer Science}, 325:147--168,
  2016.

\bibitem{haghverdi2000categorical}
E.~Haghverdi.
\newblock {\em A categorical approach to linear logic, geometry of proofs and
  full completeness.}
\newblock PhD thesis, University of Ottawa (Canada), 2000.

\bibitem{haghverdi2005geometry}
E.~Haghverdi and P.~Scott.
\newblock From geometry of interaction to denotational semantics.
\newblock {\em Electronic Notes in Theoretical Computer Science}, 122:67--87,
  2005.

\bibitem{haghverdi2006categorical}
E.~Haghverdi and P.~Scott.
\newblock {A Categorical Model for the Geometry of Interaction}.
\newblock {\em Theoretical Computer Science}, 350(2-3):252--274, 2006.

\bibitem{haghverdi2010geometry}
E.~Haghverdi and P.~Scott.
\newblock {Geometry of Interaction and the Dynamics of Proof Reduction: a
  tutorial}.
\newblock {\em New Structures for Physics}, 813:357--417, 2010.

\bibitem{haghverdi2010towards}
E.~Haghverdi and P.~Scott.
\newblock Towards a typed geometry of interaction.
\newblock {\em Mathematical Structures in Computer Science}, 20(3):473--521,
  2010.

\bibitem{Hasegawa97recursionfrom}
M.~Hasegawa.
\newblock Recursion from cyclic sharing: traced monoidal categories and models
  of cyclic lambda calculi.
\newblock In {\em Typed Lambda Calculi and Applications}, pages 196--213,
  Berlin, Heidelberg, 1997. Springer.

\bibitem{hasegawa2004uniformity}
M.~Hasegawa.
\newblock The uniformity principle on traced monoidal categories.
\newblock {\em Publications of the Research Institute for Mathematical
  Sciences}, 40(3):991--1014, 2004.

\bibitem{hasegawa2009traced}
M.~Hasegawa.
\newblock On traced monoidal closed categories.
\newblock {\em Mathematical Structures in Computer Science}, 19(2):217--244,
  2009.

\bibitem{Hasegawa2023tracedmonadshopf}
M.~Hasegawa and J.-S.~P. Lemay.
\newblock Traced {M}onads and {H}opf {M}onads.
\newblock {\em {Compositionality}}, 5:1--34, October 2023.

\bibitem{jacobs2010coalgebraic}
B.~Jacobs.
\newblock From coalgebraic to monoidal traces.
\newblock {\em Electronic Notes in Theoretical Computer Science},
  264(2):125--140, 2010.

\bibitem{joyal_street_verity_1996}
A.~Joyal, R.~Street, and D.~Verity.
\newblock {Traced Monoidal Categories}.
\newblock {\em Mathematical Proceedings of the Cambridge Philosophical
  Society}, 119(3):447–468, 1996.

\bibitem{manes2012algebraic}
E.~G. Manes and M.~A. Arbib.
\newblock {\em Algebraic approaches to program semantics}.
\newblock Springer Science \& Business Media, Berlin, Heidelberg, 2012.

\bibitem{selinger1999categorical}
P.~Selinger.
\newblock Categorical structure of asynchrony.
\newblock {\em Electronic Notes in Theoretical Computer Science}, 20:158--181,
  1999.

\bibitem{selinger2010survey}
P.~Selinger.
\newblock A survey of graphical languages for monoidal categories.
\newblock In {\em New Structures for Physics}, pages 289--355. Springer,
  Berlin, Heidelberg, 2010.

\end{thebibliography}

\end{document}